\documentclass[reqno,12pt]{amsart}

\usepackage[all]{xypic}
\usepackage{enumerate}
\usepackage{hyperref}
\usepackage{a4}
\usepackage{amssymb}
\usepackage{color}

\DeclareMathOperator{\res}{res}

\DeclareMathOperator{\pr}{pr}

\DeclareMathOperator{\ind}{ind}
\DeclareMathOperator{\supp}{supp}

\DeclareMathOperator{\gr}{gr}
\DeclareMathOperator{\ev}{ev}
\DeclareMathOperator{\rank}{rank}

\DeclareMathOperator{\tr}{tr}

\DeclareMathOperator{\vvol}{vol}

\newcommand\goe{\mathfrak g}
\newcommand\poe{\mathfrak p}
\newcommand\noe{\mathfrak n}

\newcommand\llangle{\langle\!\langle}
\newcommand\rrangle{\rangle\!\rangle}

\DeclareMathOperator{\Ad}{Ad}

\DeclareMathOperator{\id}{id}

\DeclareMathOperator{\SU}{SU}

\DeclareMathOperator{\GL}{GL}

\DeclareMathOperator{\Aut}{Aut}
\DeclareMathOperator{\eend}{end}

\DeclareMathOperator{\inc}{inc}

\newcommand\Z{\mathbb Z}
\newcommand\N{\mathbb N}
\newcommand\R{\mathbb R}
\newcommand\C{\mathbb C}

\newcommand\st{2}
\newcommand\op{\textrm{op}}
\newcommand\prop{\textrm{prop}}

\newcommand{\DO}{\mathcal D\mathcal O}

\newcommand\gl{\mathfrak{gl}}

\newcommand\itemref[1]{(\ref{#1})}

\theoremstyle{plain}
  \newtheorem{theorem}{Theorem}
  \newtheorem{corollary}{Corollary}
  \newtheorem{lemma}{Lemma}
  \newtheorem*{lemma*}{Lemma}
  \newtheorem{proposition}{Proposition}
  \newtheorem*{proposition*}{Proposition}
  
  \newtheorem*{conjecture*}{Conjecture}
  
\theoremstyle{definition}
  \newtheorem{definition}{Definition}
  
  \newtheorem*{example*}{Example}
  
\theoremstyle{remark}
  \newtheorem*{remark*}{Remark}
  \newtheorem{remark}{Remark}

\begin{document}

\title{The heat asymptotics on filtered manifolds}

\author{Shantanu Dave}

%\thanks{S.~D.\ was supported by the Austrian Science Fund (FWF) grant P30233.}

\address{Shantanu Dave,
         Wolfgang Pauli Institute
         c/o Faculty of Mathematics,
         University of Vienna,
         Oskar-Morgenstern-Platz 1,
         1090 Vienna,
         Austria.}

\curraddr{Department of Mathematics, 
          Rutgers University,
          Hill Center for the Mathematical Sciences, 
          110 Frelinghuysen Rd.,
          Piscataway, NJ 08854-8019.}

\email{shantanu.dave@univie.ac.at}
\email{shantanu.dave@rutgers.edu}

\author{Stefan Haller}

\address{Stefan Haller,
         Department of Mathematics,
         University of Vienna,
         Oskar-Morgenstern-Platz 1,
         1090 Vienna,
         Austria.}

\email{stefan.haller@univie.ac.at}

%\thanks{The second author gratefully acknowledges the support of the Austrian Science Fund (FWF) through the START-Project Y963-N35 of Michael Eichmair.}

\begin{abstract}
The short-time heat kernel expansion of elliptic operators provides a link between local and global features of classical geometries.
For many geometric structures related to (non-)involutive distributions, the natural differential operators tend to be Rockland, hence hypoelliptic.
In this paper we establish a universal heat kernel expansion for formally selfadjoint non-negative Rockland differential operators on general closed filtered manifolds.
The main ingredient is the analysis of parametrices in a recently constructed calculus adapted to these geometric structures.
The heat expansion implies that the new calculus, a more general version of the Heisenberg calculus, also has a non-commutative residue.
Many of the well known implications of the heat expansion such as, the structure of the complex powers, the heat trace asymptotics, the continuation of the zeta function, as well as Weyl's law for the eigenvalue asymptotics, can be adapted to this calculus.
Other consequences include a McKean--Singer type formula for the index of Rockland differential operators.
We illustrate some of these results by providing a more explicit description of Weyl's law for Rumin--Seshadri operators associated with curved BGG sequences over 5-manifolds equipped with a rank two distribution of Cartan type.
\end{abstract}

\keywords{Filtered manifold; Rockland operator; differential operator; pseudodifferential operator; hypoelliptic operator; parametrix; heat kernel expansion; heat trace asymptotics; complex powers; zeta function; Weyl's law; McKean--Singer index formula; holomorphic families of pseudodifferential operators; non-commutative residue; generic rank two distribution in dimension five; distribution of Cartan type}

\subjclass[2010]{58J35 (58A30, 58J20, 58J40, 58J42)}

\maketitle

\setcounter{tocdepth}{2}
\tableofcontents

\section{Introduction}

Many geometric structures related to (non-)involutive distributions can be described in terms of an underlying filtered manifold.
These include contact structures, Engel manifolds, and all regular parabolic geometries.
Filtered analogues of classical (elliptic) operators are usually hypoelliptic \cite{DH17}.
It is an old dream to link local and global aspects of filtered geometry with the spectrum and index of hypoelliptic operators by studying the short-time asymptotics of heat kernels, similarly to the elliptic case.
Active research in this direction started in the 70s, see \cite{FS74,RS76,M79}, and was quite popular in 80s, see \cite{FS82,BG88}.
The main accomplishment had been the development of the Heisenberg calculus \cite{BG88,T84,P08} and its application to operators in contact and CR geometries.
On graded nilpotent Lie groups, the representation theory allowed a harmonic analysis perspective of hypoellipticity \cite{HN78,HN79,CGGP92}.

The progress on these problems has been revived by the advent of new tangent groupoid techniques to study analysis on these manifolds \cite{E10a,E10b,CP19,EY15v5,EY16}.
When the classical (elliptic) pseudodifferential calculus was described in \cite{DS14} in a coordinate-free way using the tangent groupoid, it
facilitated the construction of a pseudodifferential calculus on general filtered manifolds \cite{EY15v5,EY16}.
This calculus has helped resolve two main hitherto unsurmountable challenges in analysis.
It first allowed an explicit handle on the parametrix of a hypoelliptic operator similar to the one familiar from the elliptic case or nilpotent Lie groups as in \cite{CGGP92}.
The universality of this calculus and the existence of the parametrix \cite{DH17} provide a clean way to obtain a general short-time heat kernel expansion for a large class of differential operators.
Remarkably, this heat expansion has the same structure as the one for elliptic operators, bringing back the old expectation that the analysis should relate local geometric properties to global invariants.

The differential operators we consider satisfy a pointwise Rockland condition, a condition on their non-commutative principal symbols that guarantees that these operators are hypoelliptic.
We also present here several consequences of the heat kernel asymptotics, which are well known for elliptic operators, but are not known for Rockland operators in this generality.
These include the explicit structure of complex powers of the operators, Weyl's asymptotic formula for the growth of eigenvalues, the McKean--Singer index formula, the description of a non-commutative residue for Heisenberg calculus, and the construction of a K-homology class associated to Rockland operators.

Let us look at the manifolds, the operators and their heat kernel expansions in more detail.

\subsection{Filtered manifolds}

Filtered manifolds provide a very general setup to study geometry \cite{M93, M02}.
These geometries include foliations, contact manifolds, Engel structures on $4$-manifolds, graded nilpotent Lie groups, and all regular parabolic geometries.
The equiregular Carnot--Carath\'eodory spaces considered in \cite{Gro96} are also filtered manifolds.

A filtered manifold $M$ has a natural associated non-commutative tangent bundle, that is, a simply connected nilpotent Lie group $\mathcal T_xM$ attached to each point $x$ in $M$.
The harmonic analysis of this non-commutative tangent space $\mathcal T_xM$ is related to the analysis of differential operators on $M$.
An effective way to exhibit this relationship uses the so-called Heisenberg tangent groupoid \cite{CP19, EY16}.
The pseudodifferential calculus on filtered manifolds \cite{EY15v5} was constructed using the Heisenberg tangent groupoid.
It was further studied in \cite{DH17} where we show  that a pointwise Rockland condition implies hypoellipticity.
We remark that the idea of Heisenberg calculus is based on the work of Debourd--Skandalis \cite{DS14}, who studied the standard pseudodifferential calculus using Connes' tangent groupoid construction.

The Heisenberg calculus can be globally defined on a filtered manifold $M$ using the Heisenberg tangent groupoid, $\mathbb TM$.
This is a groupoid that provides a deformation of the pair groupoid $M\times M$ into the non-commutative tangent bundle $\mathcal TM$, which is obtained by giving
\[
\mathbb TM=\bigl(\mathcal TM\times \{0\}\bigr)\sqcup\bigl(M\times M\times \mathbb R^{\times}\bigr)
\]
a smooth structure which makes all the groupoid maps smooth.
There is an action of $\mathbb R^{\times}:=\mathbb R\setminus\{0\}$ on $\mathbb TM$ called the zoom action and the calculus is defined to be those  distributions on $M\times M$ that admit an essentially homogeneous extension to $\mathbb TM$ with respect to the zoom action.

There is an equivalent local description of the kernels in suitably chosen coordinates.
In this paper we will mostly work with the local version, however the global description allows a neater formulation of the concept of holomorphic families described below.
In the local description the kernels are described in a tubular neighborhood of the diagonal and the germs of diffeomorphism of these tubular neighborhood to the normal bundle is determined by the Euler-like vector field provided by the filtration \cite{HS18}.
Here we will stick to the original Heisenberg tangent groupoid description for global invariance of the calculus.

\subsection{Rockland differential operators}
Let us briefly recall the kind of operators studied in this article.
These are hypoelliptic operators on filtered manifolds which are elliptic in the Heisenberg calculus described above.

Let $M$ be a closed filtered manifold and suppose $E$ is a complex vector bundle over $M$.
Let $A$ be a differential operator acting on sections of $E$.
In coordinates adapted to the filtration on $M$ we can assign a (co)symbol to $A$ by freezing the coefficients at a point $x\in M$, thereby obtaining a left invariant differential operator $\sigma_x(A)$ on the osculating group $\mathcal T_xM$.
We require that each $\sigma_x(A)$ satisfies the Rockland condition and furthermore $A$ is symmetric and non-negative on $L^2(E)$.
Examples of such operators include the sub-Laplacian for many classes of sub-Riemannian geometries, the Rumin--Seshadri operators  associated with (ungraded) BGG operators on parabolic geometries, and the square of the Connes--Moscovici transverse signature operator \cite{CM95}.
On trivially filtered manifolds, these operators are elliptic in the usual sense, and on contact or Heisenberg manifolds they are elliptic in the Heisenberg calculus of \cite{BG88, P08}.

\subsection{Heat kernel asymptotics}
In this paper we study the kernel of $e^{-tA}$ for $t>0$ and the complex powers $A^{-z}$ where $z\in \mathbb C$ for a differential operator $A$ described in the paragraph above.
Rather than using the standard approach by Seeley \cite{S67}, we follow Beals--Greiner--Stanton \cite{BGS84} who use a Volterra--Heisenberg calculus on $M\times\R$ to describe the heat kernel on CR manifolds, see also \cite[Chapter~5]{P08}.
In our general setting, however, no new calculus has to be developed.
Instead, we may regard $M\times\R$ as a filtered manifold such that the heat operator $A+\frac\partial{\partial t}$ becomes Rockland on $M\times\R$.
The general Rockland theorem established in \cite[Theorem~3.13]{DH17} thus yields a parametrix for the heat operator $A+\frac\partial{\partial t}$ in the pseudodifferential operator calculus developed by van~Erp and Yuncken in \cite{EY15v5}.
The asymptotic expansion of its Schwartz kernel along the diagonal immediately yields the heat kernel asymptotics for $A$.

Let us emphasize that this approach to the heat kernel asymptotics seems to only work for Rockland differential operators.
For more general pseudodifferential operators $P$ on the filtered manifold $M$, the corresponding heat operator $\frac{\partial}{\partial t}+P$ on $M\times \mathbb R$ is no longer pseudolocal in general and hence not in the Heisenberg pseudodifferential calculus.
This implies that we can only investigate ungraded Rockland sequences in the sense of \cite{DH17} and, in particular, the results can be applied only to ungraded BGG sequences. 
For example, the BGG sequences for Cartan geometries of generic rank 2 distributions in dimension five are always ungraded, and we will study them in greater detail here.

We mention here two prominent features of the heat asymptotics of Rockland differential operators.

As is well known in the classical (trivially filtered) case, certain terms in the heat kernel asymptotics of elliptic pseudodifferential operators vanish for differential operators. 
We show that this continues to hold for Rockland differential operators on filtered manifolds.
Their expansion takes a form similar to the expansion of classical elliptic differential operators.
In particular, there are no log terms, and every other polynomial term vanishes.
To see this, we introduce the class of projective operators, see Definition~\ref{D:PDO2} below, which have the desired asymptotic expansion and contain the parametrices of Rockland differential operators.

Another key feature of the heat asymptotics of a non-negative elliptic operator is the positivity of its leading term.
This term has geometric significance: for instance, the leading term in the heat trace expansion of a Laplace operator encodes the volume of the corresponding Riemannian metric.
Let us emphasize that this positivity of the leading term remains true for non-negative Rockland differential operators.
In particular, this permits to derive a Weyl law for the eigenvalue asymptotics using a Tauberian theorem.

\subsection{Applications}
The heat kernel asymptotics of an elliptic operator has many important consequences.
We  present here several analogous consequences for Rockland differential operators.  
This requires adapting the results in the context of the Heisenberg pseudodifferential calculus described above.

1. The structure of the complex powers $A^{-z}$ follows via Mellin transform from the heat asymptotics in the usual way.
Analogous to the celebrated result of Seeley \cite{S67}, the complex powers are pseudodifferential operators in the calculus of van Erp and Yuncken \cite{EY15v5}.
We interpret the notion of a holomorphic family of pseudodifferential operators \cite{P08}, in terms of an essential homogeneity criterion on the Heisenberg tangent groupoid. 
As can be expected, $A^{-z}$ is a holomorphic family for every non-negative Rockland differential operator $A$.

2. Taking the trace of a holomorphic family, whenever it is defined, produces a holomorphic function.
We will show that the various spectral zeta functions have meromorphic continuation with simple poles. 
In particular, the algebra of integer order pseudodifferential operators on a filtered manifold admits a non-commutative residue.

3. Weyl's law for the eigenvalue asymptotics also follows from the positivity of the constant term in the heat expansion, or from the position and residue at the first pole of the spectral zeta function as usual.

We will work out Weyl's law more explicitly for Cartan geometries in dimension five using suitable geometric choices.
Surprisingly, the constant in Weyl's law is universal, depending only on the irreducible representation defining the BGG sequence.

4. Rockland operators are hypoelliptic and hence on a closed manifold they are Fredholm. 
In particular we note two consequences towards the study of their index, namely the description of the K-homology class associated to them, and a generalization of the McKean--Singer formula.

\subsection{Structure of the paper}
The remaining part of the paper is organized as follows.
In Section~\ref{S:results} we formulate our main results: Theorem~\ref{T:heat} on the heat kernel asymptotics and Theorem~\ref{T:powers} on the structure of complex powers.
Furthermore, we derive several immediate consequences: Corollary~\ref{C:trace} on the heat trace expansion, Corollary~\ref{C:zeta} on the zeta function, Corollary~\ref{C:weyl} on Weyl's law, Corollary~\ref{C:McKeanSinger} on the McKean--Singer formula, as well as Corollary~\ref{C:KM} on the K-homology class.
In Section~\ref{S:calculus} we briefly recall some background for the calculus of pseudodifferential operators on filtered manifolds and introduce the class of projective pseudodifferential operators of integral Heisenberg order, see Definition~\ref{D:PDO2}.
In Sections~\ref{S:heat} and \ref{S:power} we present proofs of Theorems~\ref{T:heat} and \ref{T:powers}, respectively.
In Section~\ref{S:fam} we discuss holomorphic families of Heisenberg pseudodifferential operators.
In Section~\ref{S:res} we construct a non-commutative residue, see Corollary~\ref{cor:ncr}.
In Section~\ref{S:Weyl} we will work out Weyl's law more explicitly for Rumin--Seshadri operators, see Corollary~\ref{C:WeylRS}, and specialize further to BGG operators on 5-manifolds equipped with a rank two distribution of Cartan type, see Corollary~\ref{C:235}.

\subsection*{Acknowledgments}
We would like to thank an anonymous referee for helpful remarks and another anonymous referee for a thorough reading and many useful comments.
S.~D.\ was supported by the Austrian Science Fund (FWF): project number P30233.
The second author gratefully acknowledges the support of the Austrian Science Fund (FWF): project number Y963-N35, the START-Program of Michael Eichmair.

\section{Statement of the main results}\label{S:results}

Recall that a \emph{filtered manifold} \cite{M93,M02,N10,DH17} is a smooth manifold $M$ together with a filtration of the tangent bundle $TM$ by smooth subbundles,
$$
TM=T^{-m}M\supseteq\cdots\supseteq T^{-2}M\supseteq T^{-1}M\supseteq T^0M=0,
$$
which is compatible with the Lie bracket of vector fields in the following sense: If $X\in\Gamma^\infty(T^pM)$ and $Y\in\Gamma^\infty(T^qM)$ then $[X,Y]\in\Gamma^\infty(T^{p+q}M)$.
Putting $\mathfrak t^pM:=T^pM/T^{p+1}M$, the Lie bracket induces a vector bundle homomorphism $\mathfrak t^pM\otimes\mathfrak t^qM\to\mathfrak t^{p+q}M$ referred to as \emph{Levi bracket.}
This turns the associated graded vector bundle $\mathfrak tM:=\bigoplus_p\mathfrak t^pM$ into a bundle of graded nilpotent Lie algebras called the \emph{bundle of osculating algebras.}
The Lie algebra structure on the fiber $\mathfrak t_xM=\bigoplus_p\mathfrak t^p_xM$ depends smoothly on the base point $x\in M$, but is not assumed to be locally trivial.
In particular, the Lie algebras $\mathfrak t_xM$ might be non-isomorphic for different $x\in M$.

Using negative degrees, we are following a convention prevalent in parabolic geometry, see \cite{CS09,M02} for instance.
The other convention, where everything is concentrated in positive degrees, is the one that has been adopted in \cite{EY15v5,EY16}.

For real $\lambda\neq0$, we let $\dot\delta_\lambda\colon\mathfrak tM\to\mathfrak tM$ denote the grading automorphism given by multiplication with $\lambda^{-p}$ on the summand $\mathfrak t^pM$.
Note that $\dot\delta_\lambda$ restricts to a Lie algebra automorphism $\dot\delta_{\lambda,x}\in\Aut(\mathfrak t_xM)$ for each $x\in M$.
We will denote the \emph{homogeneous dimension} of $M$ by
\begin{equation}\label{E:homdim}
n:=-\sum_{p=1}^mp\cdot\rank(\mathfrak t^pM).
%=-\sum_{p=1}^mp\cdot\rank(T^{-p}M/T^{-p+1}M).
\end{equation}
Its fundamental importance stems from the fact that a $1$-density $\mu$ on the vector space $\mathfrak t_xM$ scales according to 
\begin{equation}\label{E:deltamu}
(\dot\delta_{\lambda,x})_*\mu=|\lambda|^{-n}\mu,\qquad\lambda\neq0.
\end{equation}

The filtration of the tangent bundle induces a \emph{Heisenberg filtration} on differential operators.
More explicitly, if $E$ and $F$ are two vector bundles over $M$, then a differential operator $\Gamma^\infty(E)\to\Gamma^\infty(F)$ is said to have \emph{Heisenberg order} at most $r$ if, locally, it can be written as a finite sum of operators of the form $\Phi\nabla_{X_k}\cdots\nabla_{X_1}$ where $\Phi\in\Gamma^\infty(\hom(E,F))$, $\nabla$ is a linear connection on $E$, and $X_j\in\Gamma^\infty(T^{p_j}M)$ such that $-(p_1+\cdots+p_k)\leq r$.
Denoting the space of differential operators of Heisenberg order at most $r$ by $\DO^r(E,F)$, we obtain a filtration 
\begin{equation}\label{E:DOfilt}
\Gamma^\infty(\hom(E,F))=\DO^0(E,F)\subseteq\DO^1(E,F)\subseteq\DO^2(E,F)\subseteq\cdots
\end{equation}
which is compatible with composition and taking the formal adjoint.
More precisely, if $A\in\DO^r(E,F)$ and $B\in\DO^s(F,G)$ where $G$ is another vector bundle, then $BA\in\DO^{s+r}(E,G)$ and $A^*\in\DO^r(F,E)$.
As usual, the formal adjoint is with respect to standard $L^2$ inner products of the form
\begin{equation}\label{E:L2}
\llangle\psi_1,\psi_2\rrangle:=\int_Mh(\psi_1(x),\psi_2(x))dx,
\end{equation}
where $\psi_1,\psi_2\in\Gamma_c^\infty(E)$. 
Here $dx$ is a volume density on $M$ and $h$ is a fiberwise Hermitian inner product on $E$.
The formal adjoint can be characterized by the equation $\llangle A^*\phi,\psi\rrangle_E=\llangle\phi,A\psi\rrangle_F$ for all $\psi\in\Gamma^\infty_c(E)$ and $\phi\in\Gamma^\infty_c(F)$.

A differential operator $A\in\DO^r(E,F)$ has a \emph{Heisenberg principal symbol}, $\sigma^r_x(A)\in\mathcal U_{-r}(\mathfrak t_xM)\otimes\hom(E_x,F_x)$ at every point $x\in M$.
Here 
$$
\mathcal U_{-r}(\mathfrak t_xM)=\bigl\{\mathbf X\in\mathcal U(\mathfrak t_xM)\bigm|\forall\lambda\neq0:\dot\delta_{\lambda,x}(\mathbf X)=\lambda^r\mathbf X\bigr\}
$$ 
denotes the degree $-r$ subspace in the universal enveloping algebra of the graded nilpotent Lie algebra $\mathfrak t_xM$.
The Heisenberg principal symbol is compatible with composition and taking the formal adjoint, that is, 
\begin{equation}\label{E:sAB}
\sigma^{s+r}_x(BA)=\sigma^s_x(B)\sigma^r_x(A)\qquad\text{and}\qquad\sigma^r_x(A^*)=\sigma^r_x(A)^*,
\end{equation}
for all $A\in\DO^r(E,F)$ and $B\in\DO^s(F,G)$.
Actually, the Heisenberg principal symbol provides a canonical short exact sequence,
$$
0\to\DO^{r-1}(E,F)\to\DO^r(E,F)\xrightarrow{\sigma^r}\Gamma^\infty\bigl(\mathcal U_{-r}(\mathfrak tM)\otimes\hom(E,F)\big)\to0.
$$
In particular, the Heisenberg principal symbol provides a canonical isomorphism between the associated graded of the filtered algebra $\DO(E)$ with the Heisenberg filtration \eqref{E:DOfilt}, and the graded algebra $\Gamma^\infty\bigl(\bigoplus_p\mathcal U_p(\mathfrak tM)\otimes\eend(E)\bigr)$.
More details may be found in \cite[Section~1.2.5]{N10}.

For $x\in M$ we let $\mathcal T_xM$ denote the \emph{osculating group} at $x$, that is, the simply connected nilpotent Lie group with Lie algebra $\mathfrak t_xM$.
The individual osculating groups can be put together to form a bundle of nilpotent Lie groups $\mathcal TM$ over $M$ such that the fiberwise exponential map, $\exp\colon\mathfrak tM\to\mathcal TM$, becomes a diffeomorphism of locally trivial bundles over $M$.
The scaling automorphisms $\dot\delta_{\lambda,x}\in\Aut(\mathfrak t_xM)$ integrate to Lie group automorphisms $\delta_{\lambda,x}\in\Aut(\mathcal T_xM)$ which combine to form bundle diffeomorphisms $\delta_\lambda\colon\mathcal TM\to\mathcal TM$ such that $\delta_\lambda\circ\exp=\exp\circ\dot\delta_\lambda$.
The Heisenberg principal symbol $\sigma^r_x(A)$ of an operator $A\in\DO^r(E,F)$ can be regarded as a left invariant differential operator on $\mathcal T_xM$ which is homogeneous of degree $r$.
More precisely, regarding $\sigma^r_x(A)\colon C^\infty(\mathcal T_xM,E_x)\to C^\infty(\mathcal T_xM,F_x)$, we have $\sigma^r_x(A)\circ l_g^*=l_g^*\circ\sigma^r_x(A)$ and $\sigma^r_x(A)\circ\delta_{\lambda,x}^*=\lambda^r\delta_{\lambda,x}^*\circ\sigma^r_x(A)$ for all $x\in M$, $\lambda\neq0$ and $g\in\mathcal T_xM$. 
Here $l_g^*$ denotes pull back along the left translation, $l_g\colon\mathcal T_xM\to\mathcal T_xM$, $l_g(h):=gh$, and $\delta_{\lambda,x}^*$ denotes pull back along $\delta_{\lambda,x}\colon\mathcal T_xM\to\mathcal T_xM$.

A differential operator $A\in\DO^r(E,F)$ is said to satisfy the \emph{Rockland condition} \cite{R78} at $x\in M$ if, for every non-trivial irreducible unitary representation of the osculating group, $\pi\colon\mathcal T_xM\to U(\mathcal H)$, on a Hilbert space $\mathcal H$, the linear operator $\pi(\sigma^r_x(A))\colon\mathcal H_\infty\otimes E_x\to\mathcal H_\infty\otimes F_x$ is injective.
Here $\mathcal H_\infty$ denotes the subspace of smooth vectors in $\mathcal H$.
%These operators will also be referred to as \emph{Rockland operators.}
An operator is called a \emph{Rockland operator} if it satisfies the Rockland condition at every point $x\in M$.
We refer to \cite[Section~2.3]{DH17} for more details and references.

According to \cite[Theorem~3.13]{DH17} every Rockland operator $A\in\DO^r(E,F)$ admits a properly supported left parametrix $B\in\Psi^{-r}_\prop(F,E)$ such that $BA-\id$ is a smoothing operator.
Here $\Psi^{-r}$ denotes the class of \emph{pseudodifferential operators of Heisenberg order $-r$} which has recently been introduced by van Erp and Yuncken, see \cite{EY15v5} and Section~\ref{S:calculus} below.
These are operators whose Schwartz kernels have a wave front set which is contained in the conormal of the diagonal.
In particular, their kernels are smooth away from the diagonal.
Moreover, their kernels admit an asymptotic expansion along the diagonal with respect to a tubular neighborhood which is adapted to the filtration.
In particular, the left parametrix $B$ induces a continuous operator $\Gamma^\infty(F)\to\Gamma^\infty(E)$ which extends continuously to a pseudolocal operator on distributional sections, $\Gamma^{-\infty}(F)\to\Gamma^{-\infty}(E)$.
Consequently, Rockland operators are \emph{hypoelliptic}, that is, if $\psi$ is a distributional section of $E$ such that $A\psi$ is smooth on an open subset $U$ of $M$, then $\psi$ was smooth on $U$, cf.\ \cite[Corollary~2.10]{DH17}.

For the remaining part of this section we will assume $M$ to be closed.
We consider a Rockland differential operator $A\in\DO^r(E)$ of Heisenberg order $r\geq1$ which is formally selfadjoint with respect to an $L^2$ inner product of the form \eqref{E:L2}, that is, for all $\psi_1,\psi_2\in\Gamma^\infty(E)$ we have
$$
\llangle A\psi_1,\psi_2\rrangle=\llangle\psi_1,A\psi_2\rrangle.
$$

\begin{lemma}\label{L:selfadj}
With the above hypotheses, $A$ is essentially selfadjoint with compact resolvent on $L^2(E)$.
\end{lemma}

\begin{proof}
Indeed, $A$ is symmetric with dense domain $\Gamma^\infty(E)$ and thus closeable.
By regularity, the domain of its adjoint coincides with the Heisenberg Sobolev space $H^{r}(E)$, see \cite[Corollary~3.24]{DH17}.
To see that this coincides with the domain of the closure, let $\Lambda\in\Psi^r(E)$ and $\Lambda'\in\Psi^{-r}(E)$ such that $R=\Lambda'\Lambda-\id$ is a smoothing operator.\footnote{We could use $\Lambda=A$ and a left parametrix $\Lambda'=B$ as above. Alternatively, we may assume $R=0$, see \cite[Lemma~3.16]{DH17}.}
Given $\phi\in H^r(E)$, choose $\phi_j\in\Gamma^\infty(E)$ such that $\phi_j\to\Lambda\phi$ in $L^2(E)$.
Since $\Lambda'$ and $A\Lambda'$ are both bounded on $L^2(E)$, see \cite[Proposition~3.9(a)]{DH17}, we also have $\Lambda'\phi_j\to\Lambda'\Lambda\phi$ and $A\Lambda'\phi_j\to A\Lambda'\Lambda\phi$ in $L^2(E)$.
Putting $\psi_j:=\Lambda'\phi_j-R\phi\in\Gamma^\infty(E)$, we obtain $\psi_j\to\phi$ and $A\psi_j\to A\phi$ in $L^2(E)$, hence $\phi$ is in the domain of the closure of $A$.
The resolvent of $A$ is compact, since $A-z$ has a parametrix in $\Psi^{-r}(E)$ for every $z\in\C$, and these operators are compact on $L^2(E)$, see \cite[Proposition~3.9(b)]{DH17}.
\end{proof}

Assuming, moreover, that $A$ is non-negative, that is, 
$$
\llangle\psi,A\psi\rrangle\geq0
$$
for all $\psi\in\Gamma^\infty(E)$, the spectral theorem \cite[Section~VI\S 5.3]{K95} permits to construct a \emph{strongly differentiable semigroup} $e^{-tA}$ for $t\geq0$.
More precisely, for each $\psi\in L^2(E)$ the vector $e^{-tA}\psi$ is contained in the domain of $A$, and we have 
\begin{equation}\label{E:ddte-tA}
\tfrac\partial{\partial t}e^{-tA}\psi=-Ae^{-tA}\psi\qquad\text{as well as}\qquad\lim_{t\searrow0}e^{-tA}\psi=\psi.
\end{equation}
Since $A$ is non-negative, each $e^{-tA}$ is a contraction on $L^2(E)$.
According to the Schwartz kernel theorem, it thus has a distributional kernel, $k_t\in\Gamma^{-\infty}(E\boxtimes E')$.
More explicitly, we have
\begin{equation}\label{E:kt}
\langle\phi,e^{-tA}\psi\rangle
%=\int_{M\times M}(\pi_1^*\phi)k_t(\pi_2^*\psi)
=\int_{(x,y)\in M\times M}\phi(x)k_t(x,y)\psi(y)
\end{equation}
for all $t\geq0$, $\psi\in\Gamma^\infty(E)$, and $\phi\in\mathcal D(E)=\Gamma^\infty(E')$.
Here $E':=E^*\otimes|\Lambda_M|$ where $|\Lambda_M|$ denotes the bundle of $1$-densities on $M$ and $\langle-,-\rangle$ denotes the canonical pairing between $\mathcal D(E)$ and $\Gamma^{-\infty}(E)=\mathcal D'(E)$.
This pairing will also be denote by $\langle\phi,\xi\rangle=\int_M\phi\,\xi=\int_{x\in M}\phi(x)\xi(x)$, where $\phi\in\mathcal D(E)$ and $\xi\in\Gamma^{-\infty}(E)$.
In particular, the right hand side in \eqref{E:kt} denotes the canonical pairing of $\phi\boxtimes\psi\in\Gamma^\infty(E'\boxtimes E)=\mathcal D(E\boxtimes E')$ with $k_t\in\Gamma^{-\infty}(E\boxtimes E')=\mathcal D'(E\boxtimes E')$.

One main aim of this paper is to establish the following heat kernel asymptotics, generalizing a result of Beals--Greiner--Stanton for CR manifolds \cite[Theorems~5.6 and 4.5]{BGS84}, see also \cite[Theorem~5.1.26 and Proposition~5.1.15]{P08}.

\begin{theorem}[Heat kernel asymptotics]\label{T:heat}
Let $E$ be a vector bundle over a closed filtered manifold $M$.
Suppose $A\in\DO^r(E)$ is a Rockland differential operator of even \footnote{Note that there are no non-trivial formally selfadjoint and non-negative differential operators of odd Heisenberg order.} Heisenberg order $r>0$ which is formally selfadjoint and non-negative with respect to an $L^2$ inner product of the form \eqref{E:L2}, that is, $\llangle A\psi_1,\psi_2\rrangle=\llangle\psi_1,A\psi_2\rrangle$ and $\llangle A\psi,\psi\rrangle\geq0$, for all $\psi,\psi_1,\psi_2\in\Gamma^\infty(E)$.
Then $e^{-tA}$ is a smoothing operator for each $t>0$, and the corresponding heat kernels $k_t\in\Gamma^\infty(E\boxtimes E')$ depend smoothly on $t>0$.
Moreover, as $t\searrow0$, we have an asymptotic expansion
$$
k_t(x,x)\sim\sum_{j=0}^\infty t^{(j-n)/r}q_j(x),
$$
where $q_j\in\Gamma^\infty\bigl(\eend(E)\otimes|\Lambda_M|\bigr)$.
More precisely, for every integer $N$ we have $k_t(x,x)=\sum_{j=0}^{N-1}t^{(j-n)/r}q_j(x)+O(t^{(N-n)/r})$, uniformly in $x$ as $t\searrow0$.
Moreover, $q_j(x)=0$ for all odd $j$, and $q_0(x)>0$ in $\eend(E_x)\otimes|\Lambda_{M,x}|$ for each $x\in M$.
\end{theorem}

The terms $q_j\in\Gamma^\infty\bigl(\eend(E)\otimes|\Lambda_M|\bigr)$ in Theorem~\ref{T:heat} are (in principle) locally computable, they can be read off any parametrix for the heat operator $A+\frac\partial{\partial t}$, see Remark~\ref{R:loccomp} below for more details.
The leading term $q_0(x)$ can also be obtained by evaluating the heat kernel of the Heisenberg principle symbol $\sigma_x^r(A)$ at the point $(o_x,1)\in\mathcal T_xM\times(0,\infty)$, see \eqref{E:q0x} below.
Here $o_x\in\mathcal T_xM$ denotes the neutral element of the osculating group.

The spectral theorem also permits to construct complex powers $A^z$ for every $z\in\C$.
These are unbounded operators on $L^2(E)$ satisfying $A^{z_1+z_2}=A^{z_1}A^{z_2}$ for all $z_1,z_2\in\C$.
The powers are defined such that $A^z$ vanishes on $\ker(A)$ and commutes with the orthogonal projection onto $\ker(A)$.
In particular, $A^0$ is the orthogonal projection onto the orthogonal complement of $\ker(A)$, and $A^{-1}$ is the pseudoinverse of $A$.
If $z\in\N$, then $A^z$ coincides with the ordinary power, i.e., the $z$-fold product of $A$ with itself.

Another main goal of this paper is the following result about the structure of complex powers generalizing \cite[Theorems~5.3.1 and 5.3.4]{P08}.

\begin{theorem}[Complex powers]\label{T:powers}
In the situation of Theorem~\ref{T:heat}, the complex power $A^{-z}$ is a pseudodifferential operator of Heisenberg order $-zr$ for every $z\in\C$, and these powers constitute a holomorphic family of pseudodifferential operators, see Section~\ref{S:fam} below.
In particular, the kernel $k_{A^{-z}}(x,y)$ is smooth on $\{x\neq y\}\times\C$ and depends holomorphically on the variable $z\in\C$.
If $\Re(z)>n/r$, then $A^{-z}$ has a continuous kernel and its restriction to the diagonal, $k_{A^{-z}}(x,x)$, provides a holomorphic family in $\Gamma^\infty\bigl(\eend(E)\otimes|\Lambda_M|\bigr)$ for $\Re(z)>n/r$.
This family can be extended meromorphically to the entire complex plane with at most simple poles located at the arithmetic progression $(n-j)/r$ where $j\in\N_0$.
If $(n-j)/r\notin-\N_0$, then the residue of $k_{A^{-z}}(x,x)$ at $(n-j)/r$ can be expressed as
\begin{equation}\label{E:powers.res}
\res_{z=(n-j)/r}\bigl(k_{A^{-z}}(x,x)\bigr)=\frac{q_j(x)}{\Gamma((n-j)/r)},\qquad j\in\N_0,
\end{equation}
where $q_j\in\Gamma^\infty\bigl(\eend(E)\otimes|\Lambda_M|\bigr)$ are as in Theorem~\ref{T:heat}.
Moreover, $k_{A^{-z}}(x,x)$ is holomorphic at the points in $-\N_0$, taking the values
$$
k_{A^l}(x,x)=(-1)^l\,l!\,q'_{n+rl}(x),\qquad l\in\N_0.
$$
Here $q_j'(x):=q_j(x)$ for $j\neq n$ and $q_n'(x):=q_n(x)-p(x,x)$, where $p\in\Gamma^\infty(E\boxtimes E')$ denotes the Schwartz kernel of the orthogonal projection onto $\ker(A)$.
\end{theorem}

Before turning to the proof of Theorems~\ref{T:heat} and \ref{T:powers}, we will now formulate several immediate corollaries.
The next one generalizes \cite[Theorem~5.6]{BGS84}, see also \cite[Proposition~6.1.1]{P08}.

\begin{corollary}[Heat trace asymptotics]\label{C:trace}
In the situation of Theorem~\ref{T:heat}, the heat trace admits an asymptotic expansion as $t\searrow0$,
$$
\tr\bigl(e^{-tA}\bigr)\sim\sum_{j=0}^\infty t^{(j-n)/r}a_j.
$$
More precisely, for each integer $N$ we have $\tr(e^{-tA})=\sum_{j=0}^{N-1}t^{(j-n)/r}a_j+O(t^{(N-n)/r})$ as $t\searrow0$.
Moreover, $a_j=\int_M\tr_E(q_j)$ where $q_j\in\Gamma^\infty\bigl(\eend(E)\otimes|\Lambda_M|\bigr)$ is as in Theorem~\ref{T:heat}, $a_j=0$ for all odd $j$, and $a_0>0$.
\end{corollary}

\begin{proof}
Since $e^{-tA}$ is a smoothing operator, its trace can be expressed as
$$
\tr\bigl(e^{-tA}\bigr)=\int_{x\in M}\tr_E(k_t(x,x)),\qquad t>0.
$$
The asymptotic expansion of $\tr(e^{-tA})$ thus follows from the asymptotic expansion for $k_t(x,x)$ in Theorem~\ref{T:heat}.
Clearly, $a_0>0$ since $q_0(x)>0$ at each $x\in M$.
\end{proof}

\begin{corollary}[Zeta function]\label{C:zeta}
In the situation of Theorem~\ref{T:heat}, the complex power $A^{-z}$ is trace class for $\Re(z)>n/r$, and $\zeta(z):=\tr(A^{-z})$ is a holomorphic function on $\{\Re(z)>n/r\}$.
This zeta function can be extended to a meromorphic function on the entire complex plane with at most simple poles located at the arithmetic progression $(n-j)/r$ where $j\in\N_0$.
If $(n-j)/r\notin-\N_0$, then the residue of $\zeta(z)$ at $(n-j)/r$ can be expressed as
$$
\res_{z=(n-j)/r}\bigl(\zeta(z)\bigr)=\frac{a_j}{\Gamma((n-j)/r)},\qquad j\in\N_0,
$$
where $a_j$ are the constants from Corollary~\ref{C:trace}.
Moreover, $\zeta(z)$ is holomorphic at the points in $-\N_0$, taking the values
$$
\zeta(-l)=(-1)^l\,l!\,a_{n+rl}',\qquad l\in\N_0.
$$
Here $a_j':=a_j$ for $j\neq n$ and $a_n':=a_n-\dim\ker(A)$.
\end{corollary}

\begin{proof}
For $\Re(z)>n/r$ the operator $A^{-z}$ is trace class since it has a continuous kernel according to Theorem~\ref{T:powers}.
Moreover,
$$
\zeta(z)=\int_{x\in M}\tr_E\bigl(k_{A^{-z}}(x,x)\bigr)
$$
depends holomorphically on $z$ since the kernel $k_{A^{-z}}(x,x)$, considered as a family in $\Gamma^\infty(\eend(E)\otimes|\Lambda_M|)$, is holomorphic for $\Re(z)>n/r$.
Since $k_{A^{-z}}(x,x)$ can be extended meromorphically to the entire complex plane, the same holds true for $\zeta(z)$.
The pole structure, residues and special values follow immediately from the corresponding statements in Theorem~\ref{T:powers}.
\end{proof}

The next corollary generalizes \cite[Proposition~6.1.2]{P08}.

\begin{corollary}[Weyl's eigenvalue asymptotics]\label{C:weyl}
In the situation of Theorem~\ref{T:heat}, the operator $A$ is essentially selfadjoint on $L^2(E)$ with compact resolvent.
There exists a complete orthonormal system of smooth eigenvectors $\psi_j\in\Gamma^\infty(E)$ with non-negative eigenvalues $\lambda_j\geq0$, that is, $A\psi_j=\lambda_j\psi_j$ for all $j\in\N$.
Moreover, 
\begin{equation}\label{E:trl}
\tr\bigl(e^{-tA}\bigr)=\sum_{j=1}^\infty e^{-t\lambda_j}
\qquad\text{and}\qquad
\zeta(z)=\sum_{j=1}^\infty\lambda_j^{-z},
\end{equation}
for $t>0$ and $\Re(z)>n/r$, respectively.
Furthermore, 
$$
\sharp\{j\in\N\mid\lambda_j\leq\lambda\}\sim\frac{a_0\,\lambda^{n/r}}{\Gamma(1+n/r)}\qquad\textrm{as $\lambda\to\infty$,}
$$
where $a_0>0$ is the constant from Corollary~\ref{C:trace}.
\end{corollary}

\begin{proof}
We have already shown that $A$ is essentially selfadjoint with compact resolvent, see Lemma~\ref{L:selfadj}.
It is well known that the spectrum of these operators is discrete \cite[Theorem~6.29 in Chapter~III\S6.8]{K95} and real \cite[Chapter~V\S3.5]{K95}.
Since $A$ is non-negative, each eigenvalue has to be non-negative.
By hypoellipticity, its eigenfunctions are smooth, see \cite[Corollary~2.10]{DH17}.
Since $e^{-tA}$ and $A^{-z}$ are trace class for $t>0$ and $\Re(z)>n/r$, respectively, the expressions \eqref{E:trl} follow immediately, see \cite[Chapter~X\S1.4]{K95}.
Using the Tauberian theorem of Karamata, see \cite[Theorem~108]{H49} or \cite[Problem~14.2]{S01}, Weyl's law for the asymptotics of eigenvalues follows from the heat trace asymptotics in Corollary~\ref{C:trace}.
\end{proof}

To formulate the next corollary, suppose $E$ and $F$ are two vector bundles over a closed filtered manifold $M$ and let $D\in\DO^k(E,F)$ be a differential operator of Heisenberg order at most $k\geq1$ such that $D$ and $D^*$ are both Rockland.
Then $D$ induces a Fredholm operator between Heisenberg Sobolev spaces, $D\colon H_s(E)\to H_{s-2k}(F)$, for every real $s$, see \cite[Corollary~3.28]{DH17}.
Moreover, its index does not depend on $s$ and can be expressed as 
\begin{equation}\label{E:ind}
\ind(D)=\dim\ker(D)-\dim\ker(D^*).
\end{equation}
By hypoellipticity, we have $\ker(D)\subseteq\Gamma^\infty(E)$ and $\ker(D^*)\subseteq\Gamma^\infty(F)$, see \cite[Corollary~2.10]{DH17}.

Using \eqref{E:sAB} one readily checks that $D^*D$ and $DD^*$ are differential operators of Heisenberg order at most $2k$ which satisfy the Rockland condition.
Clearly, they are formally selfadjoint and non-negative.
According to Theorem~\ref{T:heat} their heat kernels admit asymptotic expansions as $t\searrow0$,
$$
k_t^{D^*D}(x,x)\sim\sum_{j=0}^\infty t^{(n-j)/2k}q_j^{D^*D}(x)
$$
and
$$
k_t^{DD^*}(x,x)\sim\sum_{j=0}^\infty t^{(n-j)/2k}q_j^{DD^*}(x),
$$
respectively. Here $q_j^{D^*D}\in\Gamma^\infty(\eend(E)\otimes|\Lambda_M|)$ and $q_j^{DD^*}\in\Gamma^\infty(\eend(F)\otimes|\Lambda_M|)$ denote the local quantities from Theorem~\ref{T:heat} for the operators $D^*D$ and $DD^*$, respectively.

\begin{corollary}[McKean--Singer index formula]\label{C:McKeanSinger}
Let $E$ and $F$ be two vector bundles over a closed filtered manifold $M$.
Moreover, let $D\in\DO^k(E,F)$ be a differential operator of Heisenberg order at most $k\geq1$ such that $D$ and $D^*$ both satisfy the Rockland condition.
Then 
\begin{equation}\label{E:MS}
\ind(D)
%=\tr\bigl(e^{-tD^*D}\bigr)-\tr\bigl(e^{-tDD^*}\bigr)
=a_n^{D^*D}-a_n^{DD^*}
\end{equation}
where $a_n^{D^*D}=\int_M\tr_E(q_n^{D^*D})$ and $a_n^{DD^*}=\int_M\tr_F(q_n^{DD^*})$.
In particular, we have $\ind(D)=0$ whenever the homogeneous dimension $n$ is odd.
\end{corollary}

\begin{proof}
The argument is the same as in the classical case.
As usual, 
$$
\tfrac\partial{\partial t}\Bigl(\tr\bigl(e^{-tD^*D}\bigr)-\tr\bigl(e^{-tDD^*}\bigr)\Bigr)
=\tr\bigl(D^*De^{-tD^*D}\bigr)-\tr\bigl(DD^*e^{-tDD^*}\bigr)
=0,
$$
for we have $De^{-tD^*D}=e^{-tDD^*}D$ and thus $\tr(D^*De^{-tD^*D})=\tr(DD^*e^{-tDD^*})$.
Hence, $\tr\bigl(e^{-tD^*D}\bigr)-\tr\bigl(e^{-tDD^*}\bigr)$ is constant in $t$ and Corollary~\ref{C:trace} yields
\begin{equation}\label{E:abc0}
\tr\bigl(e^{-tD^*D}\bigr)-\tr\bigl(e^{-tDD^*}\bigr)=a_n^{D^*D}-a_n^{DD^*},
\end{equation}
for all $t>0$.
On the other hand, $e^{-tD^*D}$ converges to the orthogonal projection onto $\ker(D^*D)$ with respect to the trace norm, as $t\to\infty$.
This follows from Weyl's law in Corollary~\ref{C:weyl} or, more directly, from Lemma~\ref{L:ktxyesti} below.
Hence,
\begin{equation}\label{E:abc1}
\lim_{t\to\infty}\tr\bigl(e^{-tD^*D}\bigr)=\dim\ker(D^*D)=\dim\ker(D).
\end{equation}
Analogously, $e^{-tDD^*}$ converges to the orthogonal projection onto $\ker(DD^*)$ and 
\begin{equation}\label{E:abc2}
\lim_{t\to\infty}\tr\bigl(e^{-tDD^*}\bigr)=\dim\ker(DD^*)=\dim\ker(D^*).
\end{equation}
Combining \eqref{E:abc1} and \eqref{E:abc2} with \eqref{E:ind}, we obtain
$$
\lim_{t\to\infty}\Bigl(\tr\bigl(e^{-tD^*D}\bigr)-\tr\bigl(e^{-tDD^*}\bigr)\Bigr)=\ind(D).
$$
Combining this with \eqref{E:abc0}, we obtain the McKean--Singer index formula \eqref{E:MS}.
If $n$ is odd, then $a_n^{D^*D}=0=a_n^{DD^*}$ according to Corollary~\ref{C:trace}, whence $\ind(D)=0$.
\end{proof}

According to Atiyah \cite{A70} elliptic differential operators represent K-homology classes of the underlying manifold, see also \cite{K75}, \cite[Section~17]{B98}, or \cite[Section~5]{R96}.
We have the following generalization for Rockland differential operators.

\begin{corollary}[K-homology class]\label{C:KM}
Let $E$ and $F$ be two vector bundles over a closed filtered manifold $M$.
Moreover, let $D\in\DO^k(E,F)$ be a differential operator of Heisenberg order at most $k\geq1$ such that $D$ and $D^*$ both satisfy the Rockland condition.
Then the following hold true:
\begin{enumerate}[(a)]
\item\label{C:KM:a} $P:=D(\id_E+D^*D)^{-1/2}=(\id_F+DD^*)^{-1/2}D$ is bounded from $L^2(E)$ to $L^2(F)$.
\item\label{C:KM:b} $P^*P-\id_E$ is compact on $L^2(E)$.
\item\label{C:KM:c} $PP^*-\id_F$ is compact on $L^2(F)$.
\item\label{C:KM:d} $[f,P]$ is compact from $L^2(E)$ to $L^2(F)$, for each $f\in C^\infty(M,\mathbb C)$.
\end{enumerate}
Hence, the operator $\begin{pmatrix}0&P^*\\P&0\end{pmatrix}$ acting on $L^2(E)\oplus L^2(F)$ together with the action of the $C^*$-algebra $C(M)$ by multiplication constitutes a graded Fredholm module, representing a K-homology class in $K_0(M)=KK(C(M),\mathbb C)$.
\end{corollary}

\begin{proof}
According to Theorem~\ref{T:powers} we have $(\id_E+D^*D)^{-1/2}\in\Psi^{-k}(E)$, and thus $P\in\Psi^0(E,F)$.
Hence, $P$ represents a bounded operator $L^2(E)\to L^2(F)$, see \cite[Proposition~3.9(a)]{DH17}.
This shows \itemref{C:KM:a}.

Clearly, $P^*P-\id_E=-(\id_E+D^*D)^{-1}\in\Psi^{-2k}(E)$.
Hence, $P^*P-\id_E$ represents a compact operator on $L^2(E)$, see \cite[Proposition~3.9(b)]{DH17}.
This shows \itemref{C:KM:b}, and \itemref{C:KM:c} can be proved analogously.

To see \itemref{C:KM:d}, we write
$$
[f,P]=[f,D](\id_E+D^*D)^{-1/2}+D[f,(\id_E+D^*D)^{-1/2}].
$$
The first summand is contained in $\Psi^{-1}(E,F)$, for we have $[f,D]\in\Psi^{k-1}(E,F)$ and $(\id_E+D^*D)^{-1/2}\in\Psi^{-k}(E)$.
The second summand is in $\Psi^{-1}(E,F)$ too, for $[f,(\id_E+D^*D)^{-1/2}]\in\Psi^{-k-1}(E)$.
Hence $[f,P]\in\Psi^{-1}(E,F)$, and thus $[f,P]$ is compact, see \cite[Proposition~3.9(b)]{DH17}.
\end{proof}

\section{Pseudodifferential operators on filtered manifolds}\label{S:calculus}

In this section we will briefly recall van~Erp and Yuncken's pseudodifferential operator calculus on filtered manifolds, see \cite{EY15v5} and \cite{DH17}.
Moreover, we will introduce a subclass of operators of integral order, characterized by an additional symmetry, and containing all differential operators.
This class will be used to show that the heat kernel expansion for a differential operator has no log terms, and every other polynomial term vanishes.

Let $M$ be a filtered manifold.
For any two complex vector bundles $E$ and $F$ over $M$, and every complex number $s$, there is a class of operators called \emph{pseudodifferential operators of Heisenberg order $s$} and denoted by $\Psi^s(E,F)$, mapping sections of $E$ to sections of $F$.
Every $A\in\Psi^s(E,F)$ has a distributional Schwartz kernel $k\in\Gamma^{-\infty}(F\boxtimes E')$ with wave front set contained in the \emph{conormal} of the diagonal.
In particular, $k$ is smooth away from the diagonal and $A$ induces a continuous operator $\Gamma^\infty_c(E)\to\Gamma^\infty(F)$ which extends continuously to a pseudolocal operator on distributional sections, $\Gamma^{-\infty}_c(E)\to\Gamma^{-\infty}(F)$.
Here $E':=E^*\otimes|\Lambda_M|$ where $E^*$ denotes the dual bundle and $|\Lambda_M|$ is the line bundle of $1$-densities on $M$.
Moreover, $F\boxtimes E'=p_1^*F\otimes p_2^*E'$ where $p_i\colon M\times M\to M$ denote the canonical projections, $i=1,2$.
As usual,
$$
\langle\phi,A\psi\rangle=\int_{(x,y)\in M\times M}\phi(x)k(x,y)\psi(y)
$$
for all $\psi\in\Gamma^\infty_c(E)$ and $\phi\in\mathcal D(F)=\Gamma^\infty_c(F')$ where $\langle-,-\rangle$ denotes the canonical pairing between $\mathcal D(F)=\Gamma^\infty_c(F')$ and $\Gamma^\infty(F)$.
We will denote the space of these conormal kernels by $\mathcal K(M\times M;E,F)$.
Moreover, we introduce the notation $\mathcal K^\infty(M\times M;E,F):=\Gamma^\infty(F\boxtimes E')$ for the subspace of smooth kernels corresponding to smoothing operators, denoted by $\mathcal O^{-\infty}(E,F)$.
The operators in the class $\Psi^s(E,F)$ can be characterized by having a conormal kernel which admits an \emph{asymptotic expansion along the diagonal} in carefully chosen coordinates.

The asymptotic expansion of the kernels can be formulated using tubular neighborhoods of the diagonal adapted to the filtration of $M$.
These are constructed using two geometrical choices: 
(1) a splitting of the filtration on $TM$, i.e., a vector bundle isomorphism $S\colon\mathfrak tM\to TM$ mapping $\mathfrak t^pM$ into $T^pM$ such that the composition with the canonical projection $T^pM\to T^pM/T^{p+1}M=\mathfrak t^pM$ is the identity on $\mathfrak t^pM$;
and (2) a linear connection $\nabla$ on $TM$ preserving the decomposition $TM=\bigoplus_pS(\mathfrak t^pM)$.
Denoting the corresponding exponential map by $\exp^\nabla\colon TM\to M$, we obtain a commutative diagram:
\begin{equation}\label{D:expexp}
\vcenter{\xymatrix{
\mathcal TM\ar[d]_-\pi&&\mathfrak tM\ar[ll]_-\exp^-\cong\ar[rr]^-{-S}_-\cong\ar[d]_{\pi^{\mathfrak tM}}&&TM\ar[rr]^-{(\pi^{TM},\exp^\nabla)}\ar[d]^{\pi^{TM}}&&M\times M\ar[d]^-{\pr_1}
\\
M\ar@/_/[u]_-o\ar@{=}[rr]&&M\ar@/_/[u]_-{o^{\mathfrak tM}}\ar@{=}[rr]&&M\ar@{=}[rr]\ar@/^/[u]^-{o^{TM}}&&M.\ar@/^/[u]^-\Delta
}}
\end{equation}
Here $\exp\colon\mathfrak tM\to\mathcal TM$ denotes the fiberwise exponential map, all downwards heading vertical arrows indicate canonical bundle projections, and the upwards pointing vertical arrows denote the corresponding zero (neutral) sections.
In particular, $\Delta(x)=(x,x)$ denotes the diagonal mapping, and $\pr_1(x,y)=x$.
Using $-S$ to identify $\mathfrak tM$ with $TM$, mediates between two common, yet conflicting, conventions we have adopted: 
The Lie algebra of a Lie group is defined using \emph{left} invariant vector fields, while the Lie algebroid of a smooth Lie groupoid is defined using \emph{right} invariant vector fields.

Restricting the composition in the top row of diagram~\eqref{D:expexp} to a sufficiently small neighborhood $U$ of the zero section in $\mathcal TM$, it gives rise to a diffeomorphism $\varphi\colon U\to V$ onto an open neighborhood $V$ of the diagonal in $M\times M$ such that the rectangle at the bottom of diagram~\eqref{D:expcoor} commutes:
\begin{equation}\label{D:expcoor}
\vcenter{\xymatrix{
\bigl(\hom(\pi^*E,\pi^*F)\otimes\Omega_\pi\bigr)|_U\ar[d]\ar[r]^-{\phi}_-\cong&(F\boxtimes E')|_V\ar[d]
\\
**[l]\mathcal TM\supseteq U\ar[r]^-{\varphi}_-\cong\ar[d]_-{\pi|_U}&**[r]V\subseteq M\times M\ar[d]^-{\pr_1|_V}
\\
M\ar@/_/[u]_-o\ar@{=}[r]&M.\ar@/^/[u]^-\Delta
}}
\end{equation}
Possibly shrinking $U$, there exists a vector bundle isomorphism $\phi$ over $\varphi$ which restricts to the tautological identification over the diagonal/zero section,
\begin{equation}\label{E:tautzero}
o^*\bigl(\hom(\pi^*E,\pi^*F)\otimes\Omega_\pi\bigr)=\hom(E,F)\otimes|\Lambda_M|=\Delta^*(F\boxtimes E'),
\end{equation}
and such that the upper rectangle in diagram~\eqref{D:expcoor} commutes.
Here $\Omega_\pi$ is the line bundle over $\mathcal TM$ obtained by applying the representation $|{\det}|$ to the frame bundle of the vertical bundle $\ker(T\pi)$.
Note that the restriction $\Omega_\pi|_{\mathcal T_xM}$ is canonically isomorphic to the 1-density bundle $|\Lambda_{\mathcal T_xM}|$ over $\mathcal T_xM$.
Moreover, we have a canonical identification $o^*\Omega_\pi=|\Lambda_{\mathfrak tM}|=|\Lambda_{TM}|$ used in \eqref{E:tautzero}.
Every pair $(\varphi,\phi)$ as above will be referred to as \emph{exponential coordinates adapted to the filtration.}
Only the germ of $(\varphi,\phi)$ along the zero section is relevant for formulating the asymptotic expansion.
We will occasionally suppress the restriction to the neighborhoods $U$ or $V$ in our notation.

Put $\mathcal K^\infty(\mathcal TM;E,F):=\Gamma^\infty\bigl(\hom(\pi^*E,\pi^*F)\otimes\Omega_\pi\bigr)$, and let $\mathcal K(\mathcal TM;E,F)$ denote the space of distributional sections of $\hom(\pi^*E,\pi^*F)\otimes\Omega_\pi$ with wave front set contained in the conormal of the zero section $o(M)$.
In particular, elements in $\mathcal K(\mathcal TM;E,F)$ are assumed to be smooth away from the zero section.
Equivalently, these can be characterized as families $a_x\in\Gamma^{-\infty}(|\Lambda_{\mathcal T_xM}|)\otimes\hom(E_x,F_x)$ which are smooth away from the origin (regular) and depend smoothly on $x\in M$.

Let $s$ be a complex number.
An element $a\in\mathcal K(\mathcal TM;E,F)$ is called \emph{essentially homogeneous of order $s$} if $(\delta_\lambda)_*a=\lambda^sa$ mod $\mathcal K^\infty(\mathcal TM;E,F)$, for all $\lambda>0$.
The space of principal cosymbols of order $s$ will be denoted by 
$$
\Sigma^s(E,F):=\left\{a\in\frac{\mathcal K(\mathcal TM;E,F)}{\mathcal K^\infty(\mathcal TM;E,F)}:\text{$(\delta_\lambda)_*a=\lambda^sa$ for all $\lambda>0$}\right\},
$$
cf.\ \cite[Definition~34]{EY15v5}.
Our notation will often not distinguish between elements in $\Sigma^s(E,F)$ and the distributions in $\mathcal K(\mathcal TM;E,F)$ representing them.
Operators in $\Psi^s(E,F)$ can be characterized as those having a conormal Schwartz kernel $k\in\mathcal K(M\times M;E,F)$ which admits an asymptotic expansion of the form
\begin{equation}\label{E:asymp}
\phi^*(k|_V)\sim\sum_{j=0}^\infty k_j
\end{equation}
where $k_j\in\Sigma^{s-j}(E,F)$.
More precisely, for every integer $N$ there exists an integer $j_N$ such that $\phi^*(k|_V)-\sum_{j=0}^{j_N}k_j$ is of class $C^N$.
Strictly speaking, the right hand side of \eqref{E:asymp} involves distributions representing the classes $k_j\in\Sigma^{s-j}(E,F)$, restricted to $U$.
Clearly, the condition expressed in \eqref{E:asymp} does not depend on the choice of these representatives.
We will continue to suppress this in our notation in similar formulas below.

It is a non-trivial fact that a kernel which admits an asymptotic expansion as above, also has an asymptotic expansion of the same form with respect to any other exponential coordinates adapted to the filtration, see \cite{EY15v5} and \cite[Remark~3.7]{DH17}.
Moreover, the leading term, $\sigma^s(A):=k_0\in\Sigma^s(E,F)$, is independent of the exponential coordinates and referred to as \emph{Heisenberg principal symbol} of $A\in\Psi^s(E,F)$.
It provides a short exact sequence
$$
0\to\Psi^{s-1}(E,F)\to\Psi^s(E,F)\xrightarrow{\sigma^s}\Sigma^s(E,F)\to0.
$$

The basic properties of the class $\Psi^s(E,F)$ and the Heisenberg principal symbol have been established in \cite{EY15v5} and are summarized in \cite[Proposition~3.4]{DH17}.
Let us mention a few.
If $A\in\Psi^s(E,F)$ and $B\in\Psi^r(F,G)$, then $BA\in\Psi^{r+s}(E,G)$ and $\sigma^{r+s}(BA)=\sigma^r(B)\sigma^s(A)$, provided one of the two operators is properly supported.
Here the multiplication of cosymbols, $\Sigma^r(F,G)\times\Sigma^s(E,F)\to\Sigma^{r+s}(E,G)$ is by fiberwise convolution on $\mathcal TM$.
Moreover, $A^t\in\Psi^s(F',E')$ and $\sigma^s(A^t)=\sigma^s(A)^t$ where $A^t$ denotes the formal transposed operator, and the transposed of $\sigma^s(A)$ is defined using the (fiberwise) inversion on $\mathcal TM$.
Assuming $r\in\N_0$, a differential operator is contained in $\Psi^r(E,F)$ if and only if it has Heisenberg order at most $r$ in the sense of Section~\ref{S:results}, and in this case the Heisenberg principal symbol discussed in Section~\ref{S:results} is related to the one introduced in the preceding paragraph by  the canonical inclusion
$$
\Gamma^\infty\bigl(\mathcal U_{-r}(\mathfrak tM)\otimes\hom(E,F)\bigr)\subseteq\Sigma^r(E,F).
$$
We have $\bigcap_{j=0}^\infty\Psi^{s-j}(E,F)=\mathcal O^{-\infty}(E,F)$, where the right hand side denotes the smoothing operators. 
The calculus is asymptotically complete.
Hence, an operator $A\in\Psi^s(E,F)$ admits a left parametrix $B\in\Psi^{-s}_\prop(F,E)$, that is to say, $BA-\id$ is a smoothing operator, if and only if $\sigma^s(A)$ admits a left inverse $b\in\Sigma^{-s}(F,E)$, that is, $b\sigma^s(A)=1$ in $\Sigma^0(E,E)$, see \cite[Theorem~60]{EY15v5}.
The general Rockland theorem in \cite[Theorem~3.13]{DH17} asserts that such a left inverse $b$ for the principal symbol exists if and only if $\sigma^s_x(A)$ satisfies the Rockland condition at each point $x\in M$.
The operator class $\Psi^s$ gives rise to a Heisenberg Sobolev scale with the expected mapping properties, see \cite[Proposition~3.21]{DH17} for details.

A more intrinsic characterization of $\Psi^s(E,F)$ can be given in terms of the tangent groupoid associated to a filtered manifold, see \cite{EY15v5,EY16,CP19}.
This is a smooth groupoid with space of units $M\times\R$ and arrows
$$
\mathbb TM=\bigl(\mathcal T^\op M\times\{0\}\bigr)\sqcup\bigl(M\times M\times\R^\times\bigr),
$$ 
defined such that the inclusions $\inc_0\colon\mathcal T^\op M\to\mathbb TM$ and $\inc_t\colon M\times M\to\mathbb TM$ for $t\neq0$ are smooth maps of groupoids.
Here we use the notation $\mathbb R^\times:=\mathbb R\setminus\{0\}$.
As with $-S$ in diagram~\eqref{D:expexp} above, the opposite groupoid $\mathcal T^\op M$ resolves two conflicting, yet common, conventions we are following, one for Lie algebras of Lie groups, and another for Lie algebroids of smooth groupoids.
The smooth structure on $\mathbb TM$ can be characterized using adapted exponential coordinates, see \cite[Theorem~16]{EY16}.
Indeed, if $\varphi\colon U\to V$ is as above, then
$$
\mathcal T^\op M\times\R\supseteq\mathbb U\xrightarrow{\Phi}\mathbb V\subseteq\mathbb TM,
\qquad
\Phi(g,t):=\begin{cases}(g,0)&\textrm{if $t=0$, and}\\\bigl(\varphi(\delta_t(\nu(g))),t\bigr)&\textrm{if $t\neq0$,}\end{cases}
$$
is a diffeomorphism from the open subset 
$$
\mathbb U:=\bigl(\mathcal T^\op M\times\{0\}\bigr)\cup\bigl\{(g,t)\in\mathcal T^\op M\times\R^\times:\delta_t(\nu(g))\in U\bigr\}
$$ 
onto the open subset $\mathbb V:=\bigl(\mathcal T^\op M\times\{0\}\bigr)\cup\bigl(V\times\mathbb R^\times\bigr)$.
Here $\nu$ denotes the inversion on the osculating groupoid, i.e., the fiberwise inversion on $\mathcal TM$, and may be regarded as an isomorphism of smooth groupoids, $\nu\colon\mathcal T^\op M\to\mathcal TM$, which intertwines the dilation $\delta_\lambda^\op$ on $\mathcal T^\op M$ with the dilation $\delta_\lambda$ on $\mathcal TM$.
The group of automorphisms $\delta_\lambda^\op$ of $\mathcal T^\op M$ extends to a group of automorphisms of the tangent groupoid called the \emph{zoom action} by putting $\delta_\lambda^{\mathbb TM}(g,0):=(\delta^\op_\lambda(g),0)$ and $\delta_\lambda^{\mathbb TM}(x,y,t):=(x,y,t/\lambda)$ for $t\in\mathbb R^\times$ and all $\lambda\neq0$, see \cite[Definition~17]{EY15v5}.
%For Connes' tangent groupoid, this zoom action has been introduced in \cite{DS14}, providing a characterization of classical pseudodifferential calculus.
%Analoguously, the tangent groupoid of a filtered manifold has been used by van~Erp and Yuncken \cite{EY15v5} to describe the pseudodifferential operator calculus on filtered manifolds discussed here.

For two vector bundles $E$ and $F$ over $M$, we consider the vector bundle 
\begin{equation}\label{E:EEFFO}
\hom\bigl(\sigma^*(E\times\R),\tau^*(F\times\R)\bigr)\otimes\Omega_\tau
\end{equation}
over $\mathbb TM$, where $\sigma,\tau\colon\mathbb TM\to M\times\R$ denote the source and target maps given by $\sigma(g,0)=(\pi(g),0)=\tau(g,0)$, $\sigma(x,y,t)=(y,t)$, and $\tau(x,y,t)=(x,t)$ where $g\in\mathcal T^\op M$, $t\in\mathbb R^\times$, and $x,y\in M$.
Moreover, $\Omega_\tau$ denotes the line bundle obtained by applying the representation $|{\det}|$ to the frame bundle of the vertical bundle of $\tau$.
We let $\mathcal K^\infty(\mathbb TM;E,F)$ denote the space of smooth sections of the vector bundle \eqref{E:EEFFO}, and we let $\mathcal K(\mathbb TM;E,F)$ denote the space of distributional sections of \eqref{E:EEFFO} with wave front set conormal to the space of units in $\mathbb TM$.
The inclusions and the scaling automorphisms give rise to a commutative diagram, $t\in\mathbb R^\times$ and $\lambda\neq0$,
\begin{equation}\label{D:KTTM}
\vcenter{\xymatrix{
\mathcal K(\mathcal T^\op M;E,F)\ar[d]^-{(\delta_\lambda^\op)_*}&\mathcal K(\mathbb TM;E,F)\ar[l]_-{\ev_0}\ar[r]^-{\ev_t}\ar[d]^-{(\delta_\lambda^{\mathbb TM})_*}&\mathcal K(M\times M;E,F)\ar@{=}[d]
\\
\mathcal K(\mathcal T^\op M;E,F)&\mathcal K(\mathbb TM;E,F)\ar[l]_-{\ev_0}\ar[r]^-{\ev_{t/\lambda}}&\mathcal K(M\times M;E,F)
}}
\end{equation}
in which all maps are multiplicative, that is to say, compatible with the convolution and transposition induced by the groupoid structures on $\mathcal T^\op M$, $\mathbb TM$, and $M\times M$, respectively.

A conormal kernel $k\in\mathcal K(M\times M;E,F)$ corresponds to an operator in $\Psi^s(E,F)$ if and only if it admits an extension across the tangent groupoid which is essentially homogeneous of order $s$.
More precisely, iff there exists $\mathbb K\in\mathcal K(\mathbb TM;E,F)$ such that $\ev_1(\mathbb K)=k$ and $(\delta^{\mathbb TM}_\lambda)_*\mathbb K=\lambda^s\mathbb K$ mod $\mathcal K^\infty(\mathbb TM;E,F)$ for all $\lambda>0$, see \cite[Definition~19]{EY15v5} or \cite[Definition~3.2]{DH17}.
In this case we have $\sigma^s(A)=\nu^*(\ev_0(\mathbb K))$.
%, where $\nu$ denotes the fiberwise inversion on $\mathcal TM$.

To express this more succinctly, let us introduce the notation
$$
\Sigma^s(\mathbb TM;E,F):=\left\{\mathbb K\in\frac{\mathcal K(\mathbb TM;E,F)}{\mathcal K^\infty(\mathbb TM;E,F)}:\text{$(\delta^{\mathbb TM}_\lambda)_*\mathbb K=\lambda^s\mathbb K$ for all $\lambda>0$}\right\}
$$
and
$$
\Sigma^s(\mathcal T^\op M;E,F):=\left\{a\in\frac{\mathcal K(\mathcal T^\op M;E,F)}{\mathcal K^\infty(\mathcal T^\op M;E,F)}:\text{$(\delta^\op_\lambda)_*a=\lambda^sa$ for all $\lambda>0$}\right\}.
$$
Then we have the following commutative diagram:
\begin{equation}\label{D:SSS}
\vcenter{\xymatrix{
0\ar[r]&\Psi^{s-1}(E,F)\ar[d]\ar@{^{(}->}[r]&\Psi^s(E,F)\ar[d]\ar[r]^-{\sigma^s}&\Sigma^s(E,F)\ar[r]&0
\\
&\frac{\Psi^{s-1}(E,F)}{\mathcal O^{-\infty}(E,F)}\ar@{^{(}->}[r]&\frac{\Psi^s(E,F)}{\mathcal O^{-\infty}(E,F)}
\\
0\ar[r]&\Sigma^{s-1}(\mathbb TM;E,F)\ar[u]^-{\ev_1}_-\cong\ar[r]^-{t}&\Sigma^s(\mathbb TM;E,F)\ar[u]^-{\ev_1}_-\cong\ar[r]^-{\ev_0}&\Sigma^s(\mathcal T^\op M;E,F)\ar[r]\ar[uu]_-{\nu^*}^-\cong&0
}}
\end{equation}
The bottom row in \eqref{D:SSS} is exact in view of \cite[Lemma~36 and Proposition~37]{EY15v5}.
The vertical arrow labeled $\ev_1$ is onto by definition of the class $\Psi^s(E,F)$.
It follows from \cite[Lemma~32 and Proposition~37]{EY15v5} that these vertical arrows are injective too.
The arrow labeled $t$ in \eqref{D:SSS} is induced by multiplication with the function $t\colon\mathbb TM\to\R$ given by the composition of the source (or target) map $\mathbb TM\to M\times\R$ with the canonical projection onto $\R$.
Note that 
\begin{equation}\label{E:deltat}
\ev_1(t\mathbb K)=\ev_1(\mathbb K)
\qquad\text{and}\qquad
\bigl(\delta^{\mathbb TM}_\lambda\bigr)_*(t\mathbb K)=\lambda t\bigl(\delta_\lambda^{\mathbb TM}\bigr)_*\mathbb K
\end{equation}
for all $\lambda\neq0$ and $\mathbb K\in\mathcal K(\mathbb TM;E,F)$.
These facts permit to define the principal symbol map $\sigma^s$ such that the diagram becomes commutative.
Moreover, the exactness of the sequence in the top row of \eqref{D:SSS} follows from the exactness of the sequence at the bottom.

By conormality, $\mathbb K$ has a Taylor expansion along $\mathcal T^\op M\times\{0\}$
\begin{equation}\label{E:taylor}
\Phi^*\bigl(\mathbb K|_{\mathbb V}\bigr)\sim\sum_{j=0}^\infty\mathbb K_jt^j
\end{equation}
where $\mathbb K_j\in\Sigma^{s-j}(\mathcal T^\op M;E,F)$ and $\mathbb K_jt^j$ is considered as a distribution on $\mathcal T^\op M\times\mathbb R$.
More precisely, for each integer $N$ there exists an integer $j_N$ such that $\Phi^*\bigl(\mathbb K|_{\mathbb V}\bigr)-\sum_{j=0}^{j_N}\mathbb K_jt^j$ is of class $C^N$.
The coefficients are related to the terms in the asymptotic expansion \eqref{E:asymp} of $k=\ev_1(\mathbb K)$ via 
\begin{equation}\label{E:kiKKi}
k_j=\nu^*(\mathbb K_j).
\end{equation}

Let us now turn to the aforementioned subclass of $\Psi^r(E,F)$ for integral $r$.

\begin{definition}\label{D:Sigma2}
Suppose $r\in\Z$.
A principal cosymbol $a\in\mathcal K(\mathcal TM;E,F)$ is said to be \emph{essentially projectively homogeneous of order $r$} if $(\delta_\lambda)_*a=\lambda^ra$ mod $\mathcal K^\infty(\mathcal TM;E,F)$ holds for all $\lambda\neq0$.
Correspondingly, we put
$$
\Sigma^r_\st(E,F):=\left\{a\in\frac{\mathcal K(\mathcal TM;E,F)}{\mathcal K^\infty(\mathcal TM;E,F)}:\text{$(\delta_\lambda)_*a=\lambda^ra$ for all $\lambda\neq0$}\right\},
$$
using the subscript $2$ to indicate this additional $\mathbb Z_2$ symmetry.
Note that 
\begin{equation}\label{E:SSst}
\Sigma^r_\st(E,F)=\bigl\{a\in\Sigma^r(E,F):(\delta_{-1})_*a=(-1)^ra\bigr\}.
\end{equation}
\end{definition}

Let us also introduce the notation 
\begin{equation}\label{E:Pr}
\mathcal P^r(\mathcal TM;E,F):=\bigl\{a\in\mathcal K^\infty(\mathcal TM;E,F):\textrm{$(\delta_\lambda)_*a=\lambda^ra$ for all $\lambda>0$}\bigr\}.
\end{equation}
Clearly, $\mathcal P^r(\mathcal TM;E,F)=0$ if $-r-n\notin\mathbb  N_0$.
For $-r-n\in\mathbb N_0$ these are smooth kernels which are polynomial along each fiber $\mathcal T_xM$.
In particular, the homogeneity in \eqref{E:Pr} remains true for all $\lambda\neq0$.
As in \cite[Lemma~3.8]{DH17} one can show:

\begin{lemma}\label{L:reps}
Suppose $a\in\mathcal K(\mathcal TM;E,F)$ and $r\in\Z$.
Then $a$ is essentially projectively homogeneous of order $r$, i.e., represents an element in $\Sigma^r_\st(E,F)$, if and only if there exist $a_\infty\in\mathcal K^\infty(\mathcal TM;E,F)$, $q\in\mathcal K(\mathcal TM;E,F)$ and $p\in\mathcal P^r(\mathcal TM;E,F)$ such that $a=a_\infty+q$ and
$$
(\delta_\lambda)_*q=\lambda^rq+\lambda^r\log|\lambda|p,\qquad\textrm{for all $\lambda\neq0$.}
$$
Here $q$ is only unique mod $\mathcal P^r(\mathcal TM;E,F)$, but $p$ is without ambiguity.
\end{lemma}

\begin{definition}\label{D:PDO2}
An operator $A\in\Psi^r(E,F)$ of integral order $r\in\Z$ is called \emph{projective} if its Schwartz kernel $k$ admits an extension across the tangent groupoid, $\mathbb K\in\mathcal K(\mathbb TM;E,F)$, which is essentially projectively homogeneous of order $r$, that is, $\ev_1(\mathbb K)=k$, and for all $\lambda\neq0$ we have $(\delta^{\mathbb TM}_\lambda)_*\mathbb K=\lambda^r\mathbb K$ mod $\mathcal K^\infty(\mathbb TM;E,F)$.
We will denote these operators by $\Psi^r_\st(E,F)$.
\end{definition}

\begin{proposition}\label{P:V}
The class $\Psi^r_\st$ has the following properties:
\begin{enumerate}[(a)]
\item\label{P:V:AB}
If $A\in\Psi^r_\st(E,F)$ and $B\in\Psi^l_\st(F,G)$ then $BA\in\Psi^{l+r}_\st(E,G)$, provided at least one of $A$ and $B$ is properly supported.
\item\label{P:V:At}
If $A\in\Psi^r_\st(E,F)$, then $A^t\in\Psi^r_\st(F',E')$.
\item\label{P:V:sigma}
If $A\in\Psi^r_\st(E,F)$, then $\sigma^r(A)\in\Sigma^r_\st(E,F)$. Moreover, 
$$
0\to\Psi^{r-1}_\st(E,F)\to\Psi^r_\st(E,F)\xrightarrow{\sigma^r}\Sigma^r_\st(E,F)\to0
$$
is a natural short exact sequence.
\item\label{P:V:SO}
$\bigcap_{r\in\Z}\Psi^r_\st(E,F)=\mathcal O^{-\infty}(E,F)$, the smoothing operators.
\item\label{P:V:DO}
For $r\in\N_0$ we have $\DO^r(E,F)=\DO(E,F)\cap\Psi^r_\st(E,F)$.
\item\label{P:V:para}
If an operator in $\Psi^r_\st(E,F)$ admits a left parametrix in $\Psi^{-r}(F,E)$, then it also admits a left parametrix $\Psi^{-r}_\st(F,E)$.
An analogous statement holds true for right parametrices.
\item\label{P:V:exp}
Suppose $A\in\Psi^r(E,F)$ and let $\phi^*(k|_V)\sim\sum_{j=0}^\infty k_j$ with $k_j\in\Sigma^{r-j}(E,F)$ denote the asymptotic expansion of its kernel along the diagonal with respect to adapted exponential coordinates.
Then $A\in\Psi^r_\st(E,F)$ if and only if $k_j\in\Sigma^{r-j}_\st(E,F)$ for all $j\in\N_0$.
\end{enumerate}
\end{proposition}

\begin{proof}
Parts \itemref{P:V:AB} and \itemref{P:V:At} can be proved exactly as in \cite{EY15v5}, see also \cite[Proposition~3.4]{DH17}.
Indeed, all maps in the commutative diagram~\eqref{D:KTTM} are multiplicative.

To see \itemref{P:V:sigma} put
$$
\Sigma^r_\st(\mathbb TM;E,F):=\left\{\mathbb K\in\frac{\mathcal K(\mathbb TM;E,F)}{\mathcal K^\infty(\mathbb TM;E,F)}:\text{$(\delta^{\mathbb TM}_\lambda)_*\mathbb K=\lambda^r\mathbb K$ for all $\lambda\neq0$}\right\}
$$
and note that 
\begin{equation}\label{E:SrTTM}
\Sigma^r_2(\mathbb TM;E,F)=\bigl\{\mathbb K\in\Sigma^r(\mathbb TM;E,F):(\delta^{\mathbb TM}_{-1})_*\mathbb K=(-1)^r\mathbb K\bigr\}.
\end{equation}
Moreover, introduce
$$
\Sigma^r_\st(\mathcal T^\op M;E,F):=\left\{a\in\frac{\mathcal K(\mathcal T^\op M;E,F)}{\mathcal K^\infty(\mathcal T^\op M;E,F)}:\text{$(\delta^\op_\lambda)_*a=\lambda^ra$ for all $\lambda\neq0$}\right\},
$$
and note that
\begin{equation}\label{E:SrTopM}
\Sigma^r_\st(\mathcal T^\op M;E,F)=\bigl\{a\in\Sigma^r(\mathcal T^\op M;E,F):(\delta^\op_{-1})_*a=(-1)^ra\bigr\}.
\end{equation}
Using the commutativity of the left square in \eqref{D:KTTM} and \eqref{E:deltat} we see that \eqref{D:SSS} restricts to a commutative diagram:
\begin{equation}\label{D:sSSS}
\vcenter{\xymatrix{
0\ar[r]&\Psi_\st^{r-1}(E,F)\ar[d]\ar@{^{(}->}[r]&\Psi^r_\st(E,F)\ar[d]\ar[r]^-{\sigma^r}&\Sigma^r_\st(E,F)\ar[r]&0
\\
&\frac{\Psi_\st^{r-1}(E,F)}{\mathcal O^{-\infty}(E,F)}\ar@{^{(}->}[r]&\frac{\Psi_\st^r(E,F)}{\mathcal O^{-\infty}(E,F)}
\\
0\ar[r]&\Sigma_\st^{r-1}(\mathbb TM;E,F)\ar[u]^-{\ev_1}_-\cong\ar[r]^-{t}&\Sigma_\st^r(\mathbb TM;E,F)\ar[u]^-{\ev_1}_-\cong\ar[r]^-{\ev_0}&\Sigma^r_\st(\mathcal T^\op M;E,F)\ar[r]\ar[uu]_-{\nu^*}^-\cong&0
}}
\end{equation}
%Note that analogous to \eqref{E:SSst}, $\Sigma^r_\st(\mathbb TM;E,F)$ can be characterized as the space of $\mathbb K\in\Sigma^r(\mathbb TM;E,F)$ which satisfy $(\delta_{-1})_*\mathbb K=(-1)^r\mathbb K$.
%Similarly, we have $a\in\Sigma^r_\st(\mathcal T^\op M;E,F)$ iff $a\in\Sigma^r(\mathcal T^\op M;E,F)$ and $(\delta_{-1})_*a=a$.
Using \eqref{E:SrTopM}, \eqref{E:SrTTM}, and the averaging operator $\frac12\bigl(\id+(-1)^{-r}(\delta^{\mathbb TM}_{-1})_*\bigr)$, one readily sees that the bottom row in the diagram above is exact, for the same is true in \eqref{D:SSS}.
Consequently, the top row in \eqref{D:sSSS} is also exact, whence \itemref{P:V:sigma}.

Part \itemref{P:V:SO} follows immediately from $\bigcap_{r\in\Z}\Psi^r(E,F)=\mathcal O^{-\infty}(E,F)$, see \cite[Corollary~53]{EY15v5} or \cite[Proposition~3.4(c)]{DH17}, for we have the obvious inclusions $\mathcal O^{-\infty}(E,F)\subseteq\Psi^r_\st(E,F)\subseteq\Psi^r(E,F)$.

The proof of \cite[Proposition~3.4(f)]{DH17} actually shows \itemref{P:V:DO}, see also \cite[Section~10.3]{EY15v5}.

To see \itemref{P:V:para}, consider $A\in\Psi^r_\st(E,F)$ and $B\in\Psi^{-r}(F,E)$ such that $BA-\id$ is a smoothing operator.
Choose $\mathbb K\in\Sigma^r_\st(\mathbb TM;E,F)$ such that $\ev_1(\mathbb K)$ represents $A$ mod smoothing operators.
Moreover, choose $\mathbb L\in\Sigma^{-r}(\mathbb TM;F,E)$ such that $\ev_1(\mathbb L)$ represents $B$ mod smoothing operators.
Since $\ev_1$ induces a multiplicative isomorphism as indicated in \eqref{D:SSS}, we conclude $\mathbb L\mathbb K=1$ in $\Sigma^0(\mathbb TM;E,E)$.
Consider $\tilde{\mathbb L}:=\frac12\bigl(\mathbb L+(-1)^r(\delta^{\mathbb TM}_{-1})_*\mathbb L\bigr)\in\Sigma^{-r}_\st(\mathbb TM;\mathbb F,\mathbb E)$.
In view of $1=(\delta^{\mathbb TM}_{-1})_*1=(\delta^{\mathbb TM}_{-1})_*(\mathbb L\mathbb K)=\bigl((\delta^{\mathbb TM}_{-1})_*\mathbb L\bigr)\bigl((\delta^{\mathbb TM}_{-1})_*\mathbb K\bigr)=\bigl((\delta^{\mathbb TM}_{-1})_*\mathbb L\bigr)\bigl((-1)^r\mathbb K\bigr)=\bigl((-1)^r(\delta^{\mathbb TM}_{-1})_*\mathbb L\bigr)\mathbb K$, we conclude $\tilde{\mathbb L}\mathbb K=1$ in $\Sigma^0_\st(\mathbb TM;E,E)$.
Hence, $\ev_1(\tilde{\mathbb L})$ gives rise to a left parametrix $\tilde B\in\Psi^{-r}_\st(F,E)$ such that $\tilde BA-\id$ is a smoothing operator.

To see \eqref{P:V:exp}, consider $\mathbb K\in\Sigma^r(\mathbb TM;E,F)$ and suppose the kernel $k=\ev_1(\mathbb K)$ represents the operator $A\in\Psi^r(E,F)$.
From \eqref{E:taylor} we get
\begin{equation}\label{E:taylor2}
\Phi^*\bigl((\delta_{-1}^{\mathbb TM})_*\mathbb K|_{\mathbb V}\bigr)\sim\sum_{j=0}^\infty(-1)^j(\delta^\op_{-1})_*(\mathbb K_j)t^j,
\end{equation}
since we have $\delta_\lambda^{\mathbb TM}(\Phi(g,t))=\Phi\bigl(\delta^\op_\lambda(g),t/\lambda\bigr)$ for all $\lambda\neq0$.
Now, $A\in\Psi^r_2(E,F)$ iff $(\delta_{-1}^{\mathbb TM})_*\mathbb K=(-1)^r\mathbb K$ in $\Sigma^r(\mathbb TM;E,F)$, see \eqref{E:SrTTM}.
Comparing \eqref{E:taylor} with \eqref{E:taylor2}, we conclude that this is the case iff $(\delta^\op_{-1})_*\mathbb K_j=(-1)^{r-j}\mathbb K_j$ in $\Sigma^{r-j}(\mathcal T^\op M;E,F)$, for all $j$.
Using \eqref{E:SrTopM}, we see that this holds iff $\mathbb K_j\in\Sigma^{r-j}_2(\mathcal T^\op M;E,F)$, for all $j$.
In view of \eqref{E:kiKKi}, this is in turn equivalent to $k_j\in\Sigma^{r-j}_2(E,F)$, for all $j$.
\end{proof}

\begin{remark}
Proposition~\ref{P:V}\itemref{P:V:At} implies that the class $\Psi^r_\st$ is invariant under taking formal adjoints too.
More precisely, if $A\in\Psi^r_\st(E,F)$, then $A^*\in\Psi^r_\st(F,E)$ where the formal adjoint is with respect to inner products of the form \eqref{E:L2}.
\end{remark}

\section{Heat kernel asymptotics}\label{S:heat}

The aim of this section is to prove Theorem~\ref{T:heat}. 
To this end let $E$ be a vector bundle over a closed filtered manifold $M$, and suppose $A\in\DO^r(E)$ is a differential Rockland operator of even Heisenberg order $r>0$ which is formally selfadjoint and non-negative with respect to the $L^2$ inner product induced by a volume density $dx$ on $M$ and a fiberwise Hermitian metric $h$ on $E$, see \eqref{E:L2}.

We turn $M\times\R$ into a filtered manifold by putting
$$
T^p(M\times\R):=\begin{cases}\pi_1^*T^pM&\textrm{if $p>-r$, and}\\\pi_1^*T^pM\oplus\pi_2^*T\R&\textrm{if $p\leq-r$.}\end{cases}
$$
Here $\pi_1\colon M\times\R\to M$ and $\pi_2\colon M\times\R\to\R$ denote the canonical projections, and we identify $T(M\times\R)=\pi_1^*TM\oplus\pi_2^*T\R$.
We will regard $A$ and $\frac\partial{\partial t}$ as operators over $M\times\R$ acting on sections of the pull back bundle $\tilde E:=\pi_1^*E$.
The fiberwise Hermitian metric $h$ on $E$ induces a fiberwise Hermitian metric $\tilde h:=\pi_1^*h$ on $\tilde E$.
Furthermore, the volume density $dx$ on $M$ and the standard volume density $dt$ on $\R$ provide a volume density $dxdt$ on $M\times\R$.
We will use the $L^2$ inner product on sections of $\tilde E$ induced by $dxdt$ and $\tilde h$, see \eqref{E:L2}.

This filtration on $M\times\R$ is motivated by the fact that the heat operator $A+\frac\partial{\partial t}$ becomes a Rockland operator.
More precisely, we have:

\begin{lemma}\label{L:rockland}
The differential operator $A+\frac\partial{\partial t}$ is a Rockland operator of Heisenberg order $r$.
The same is true for $(A+\frac\partial{\partial t})^*=A-\frac\partial{\partial t}$.
\end{lemma}

\begin{proof}
Clearly, $\frac\partial{\partial t}\in\Gamma^\infty(T^{-r}(M\times\R))$, hence $\frac\partial{\partial t}$ is a differential operator of Heisenberg order at most $r$ on $M\times\R$.
Since $\pi_1^*T^pM\subseteq T^p(M\times\R)$, the operator $A$ has Heisenberg order at most $r$ when considered on $M\times\R$.
Consequently, $A+\frac\partial{\partial t}$ has Heisenberg order at most $r$.

For $(x,t)\in M\times\R$, the canonical projections induce a canonical isomorphism of osculating algebras, 
\begin{equation}\label{E:tMR}
\mathfrak t_{(x,t)}(M\times\R)=\mathfrak t_xM\oplus\R,
\end{equation} 
where $\R$ is a central ideal in homogeneous degree $-r$.
Correspondingly, we obtain a canonical isomorphism of osculating groups,
\begin{equation}\label{E:TMR}
\mathcal T_{(x,t)}(M\times\R)=\mathcal T_xM\times\R,
\end{equation}
such that the parabolic dilation becomes
\begin{equation}\label{E:deltaMR}
\delta^{\mathcal T(M\times\R)}_{\lambda,(x,t)}(g,\tau)=\bigl(\delta^{\mathcal TM}_{\lambda,x}(g),\lambda^r\tau\bigr),
\end{equation}
where $g\in\mathcal T_xM$, $\tau\in\R$, and $\lambda\neq0$.

Let $\sigma^r_x(A)\in\mathcal U_{-r}(\mathfrak t_xM)\otimes\eend(E_x)$ denote the Heisenberg principal symbol of $A$.
For the Heisenberg principal symbol of $A+\tfrac\partial{\partial t}$ we clearly have
\begin{equation}\label{E:saddt}
\sigma^r_{(x,t)}(A+\tfrac\partial{\partial t})=\sigma^r_x(A)+T,
\end{equation}
via the canonical isomorphism induced by \eqref{E:tMR},
$$
\mathcal U_{-r}\bigl(\mathfrak t_{(x,t)}(M\times\R)\bigr)\otimes\eend\bigl(\tilde E_{(x,t)}\bigr)
=\bigl(\mathcal U(\mathfrak t_xM)\otimes\R[T]\bigr)_{-r}\otimes\eend(E_x).
$$
Here $\R[T]=\mathcal U(\R)$ denotes the polynomial algebra in one variable $T$ of homogeneous degree $-r$.

To verify the Rockland condition, consider a non-trivial irreducible unitary representation of the osculating group, $\pi\colon\mathcal T_{(x,t)}(M\times\R)\to U(\mathcal H)$, on a Hilbert space $\mathcal H$.
Since the factor $\R$ in \eqref{E:tMR} is central, $b:=\pi(T)$ acts by a scalar on $\mathcal H$, see \cite[Theorem~5 in Appendix~V]{K04}.
Hence, restricting $\pi$ via \eqref{E:TMR}, we obtain an irreducible representation $\bar\pi\colon\mathcal T_xM\to U(\mathcal H)$.
Putting $a:=\bar\pi(\sigma^r_x(A))$, \eqref{E:saddt} gives
\begin{equation}\label{E:psaddt}
\pi(\sigma^r_x(A+\tfrac\partial{\partial t}))=a+b.
\end{equation}
Let $\mathcal H_\infty$ denote the space of smooth vectors for the representation $\pi$, and note that this coincides with the space of smooth vectors of $\bar\pi$.
By unitarity, and since $A$ is assumed to be formally selfadjoint, we have $a^*=a$ and $b^*=-b$, see~\eqref{E:sAB}, hence
$$
(a-b)(a+b)=a^*a+b^*b.
$$
By positivity, it thus suffices to show that $a$ or $b$ acts injectively on $\mathcal H_\infty$, see \eqref{E:psaddt}.
If $\bar\pi$ is non-trivial, then $a$ acts injectively for $A$ is assumed to satisfy the Rockland condition.
If $\bar\pi$ is trivial, then $b$ acts injectively, for it has to be a non-trivial scalar.
\footnote{$\mathcal H$ has to be one dimensional in this case.}
\end{proof}

In view of \cite[Theorem~3.13]{DH17} and Lemma~\ref{L:rockland}, the operator $A+\frac\partial{\partial t}$ admits a properly supported parametrix $\tilde Q\in\Psi_\prop^{-r}(\tilde E)$, see also \cite[Remarks~3.18 and 3.19]{DH17}.
Hence, there exist smoothing operators $R_1$ and $R_2$ over $M\times\R$ such that 
\begin{equation}\label{E:parametrix}
\tilde Q(A+\tfrac\partial{\partial t})=\id+R_1
\qquad\text{and}\qquad
(A+\tfrac\partial{\partial t})\tilde Q=\id+R_2.
\end{equation}
By the definition of the class $\Psi^{-r}_\prop(\tilde E)$, see Section~\ref{S:calculus} or \cite[Definition~3.2]{DH17}, $\tilde Q$ has a properly supported Schwartz kernel, $k_{\tilde Q}\in\Gamma_\prop^{-\infty}(\tilde E\boxtimes\tilde E')$, whose wave front set is contained in the conormal of the diagonal.
In particular, $k_{\tilde Q}$ is smooth away from the diagonal.
Moreover, $k_{\tilde Q}$ admits an asymptotic expansion along the diagonal.

To describe the asymptotic expansion of the kernel $k_{\tilde Q}$, we will use adapted exponential coordinates for $M\times\R$ which are compatible with the product structure and the translational invariance of the operator $A+\frac\partial{\partial t}$.
We start with adapted exponential coordinates for $M$ as described in Section~\ref{S:calculus} and summarized in the commutative diagram \eqref{D:expcoor}.
The exponential coordinates on $M\times\R$ will be defined such that the following diagram commutes:
$$
\xymatrix{
\bigl(\eend(\tilde\pi^*\tilde E)\otimes\Omega_{\tilde\pi}\bigr)|_{\tilde U}\ar[d]\ar[r]^-{\tilde\phi}&(\tilde E\boxtimes\tilde E')|_{\tilde V}\ar[d]
\\
**[l]\mathcal T(M\times\R)\supseteq\tilde U\ar[r]^-{\tilde\varphi}\ar[d]_-{\tilde\pi|_{\tilde U}}&**[r]\tilde V\subseteq(M\times\R)\times(M\times\R)\ar[d]^-{\pr_1}
\\
M\times\R\ar@/_/[u]_-{\tilde o}\ar@{=}[r]&M\times\R.\ar@/^/[u]^-{\tilde\Delta}
}
$$
Here $\tilde\pi\colon\mathcal T(M\times\R)\to M\times\R$ denotes the bundle projection, $\tilde o$ is the zero section, and $\tilde\Delta$ denotes the diagonal mapping.
Moreover,
\begin{align}
\notag
\tilde V&:=\{(x,s;y,t)\in(M\times\R)\times(M\times\R):(x,y)\in V\}
\\\notag
\tilde U_{(x,s)}&:=\{(g,t)\in\mathcal T_{(x,s)}(M\times\R)=\mathcal T_xM\times\R:g\in U\}
\\\label{E:tpp}
\tilde\varphi_{(x,s)}(g,t)&:=(x,s;\varphi_x(g),s-t)
\\\label{E:tvpvp}
\tilde\phi_{(x,s)}(g,t)&:=\phi_x(g)
\end{align}
where $(x,s)\in M\times\R$ and $(g,t)\in\mathcal T_{(x,s)}(M\times\R)=\mathcal T_xM\times\R$.
Note that these are exponential coordinates as described in Section~\ref{S:calculus} associated with the splitting of the filtration $T(M\times\R)\cong\mathfrak t(M\times\R)$ induced by the splitting $TM\cong\mathfrak tM$ and the linear connection on $T(M\times\R)=\pi_1^*TM\oplus\pi_2^*T\R$ induced from the linear connection on $TM$ and the trivial connection on $T\R$.

Pulling back the Schwartz kernel of $\tilde Q$ with $\tilde\phi$, we obtain a distributional section $\tilde\phi^*(k_{\tilde Q}|_{\tilde V})$ of the vector bundle $\eend(\tilde\pi^*\tilde E)\otimes\Omega_{\tilde\pi}$ over $\tilde U$.
According to Proposition~\ref{P:V}\itemref{P:V:DO}\&\itemref{P:V:para}, the parametrix $\tilde Q$ may be assumed to be in the class $\Psi^{-r}_\st(\tilde E)$.
Hence, see Proposition~\ref{P:V}\itemref{P:V:exp}, we have an asymptotic expansion of the form
\begin{equation}\label{E:qqj}
\tilde\phi^*\bigl(k_{\tilde Q}|_{\tilde V}\bigr)\sim\sum_{j=0}^\infty\tilde q_j
\end{equation}
where $\tilde q_j\in\Sigma^{-r-j}_\st(\tilde E)$.
More precisely, for every integer $N$ there exists an integer $j_N$ such that $\tilde\phi^*(k_{\tilde Q}|_{\tilde V})-\sum_{j=0}^{j_N}\tilde q_j$ is of class $C^N$ on $\tilde U$.
According to Lemma~\ref{L:reps}, we may fix representatives $\tilde q_j\in\mathcal K\bigl(\mathcal T(M\times\R);\tilde E\bigr)$ satisfying
\begin{equation}\label{E:homog}
\bigl(\delta^{\mathcal T(M\times\R)}_\lambda\bigr)_*\tilde q_j=\lambda^{-r-j}\tilde q_j-\lambda^{-r-j}\log|\lambda|\,\tilde p_j,\qquad\lambda\neq0,
\end{equation} 
where $\tilde p_j\in\mathcal P^{-r-j}\bigl(\mathcal T(M\times\R);\tilde E\bigr)$ is smooth and strictly homogeneous,
\begin{equation*}%\label{E:pj}
\bigl(\delta^{\mathcal T(M\times\R)}_\lambda\bigr)_*\tilde p_j=\lambda^{-r-j}\tilde p_j,\qquad\lambda\neq0.
\end{equation*}
The representative $\tilde q_j$ satisfying \eqref{E:homog} is unique mod $\mathcal P^{-r-j}\bigl(\mathcal T(M\times\R);\tilde E\bigr)$.
The logarithmic terms $\tilde p_j$, however, are without ambiguity.

The heat semi group $e^{-tA}$ permits to invert the heat operator $A+\frac\partial{\partial t}$ over $M\times\R$.
Following \cite[Section~5]{BGS84} we let $C_+(\R,L^2(E))$ denote the space of continuous functions $\psi\colon\R\to L^2(E)$ for which there exists $t_0\in\R$ such that $\psi(t)=0$ for all $t\leq t_0$.
We consider the operator $Q\colon C_+(\R,L^2(E))\to C_+(\R,L^2(E))$ defined by
\begin{equation}\label{E:tQ}
(Q\psi)(s):=\int_{-\infty}^se^{-(s-t)A}\psi(t)\,dt,\qquad\psi\in C_+(\R,L^2(E)).
\end{equation}
Since $e^{-tA}$ is strongly continuous, the integrand $e^{-(s-t)A}\psi(t)$ depends continuously on $t\in(-\infty,s]$, hence the integral converges and $Q\psi\in C_+(\R,L^2(E))$, cf.\ \cite[Lemma~3.7 in Section~III\S3.1]{K95}.
Since $e^{-tA}$ is a contraction, the estimate $\|(Q\psi)(s)\|\leq|s-t_0|\sup_{t\in[t_0,s]}\|\psi(t)\|$ holds for all $\psi\in C_+(\R,L^2(E))$ which are supported on $[t_0,\infty)$.
This shows that $Q$ is continuous, when $C_+(\R,L^2(E))=\varinjlim C_{[t_0,\infty)}(\mathbb R,L^2(E))$ is equipped with the inductive limit topology over $t_0\in\mathbb R$, and $C_{[t_0,\infty)}(\mathbb R,L^2(E))$ denotes the subspace of continuous maps which are supported on $[t_0,\infty)$ carrying the topology of uniform convergence on compact subsets.
Composing $Q$ with the continuous inclusions $\Gamma_c^\infty(\tilde E)\subseteq C_+(\R,L^2(E))$ and $C_+(\R,L^2(E))\subseteq\Gamma^{-\infty}(\tilde E)$, we may thus regard it as a continuous operator, $Q\colon\Gamma^\infty_c(\tilde E)\to\Gamma^{-\infty}(\tilde E)$.
By the Schwartz kernel theorem, $Q$ has a distributional kernel, $k_Q\in\Gamma^{-\infty}(\tilde E\boxtimes\tilde E')$, that is,
\begin{equation}\label{E:ktQ}
\langle\phi,Q\psi\rangle=\int_{(x,s;y,t)\in(M\times\R)\times(M\times\R)}\phi(x,s)k_Q(x,s;y,t)\psi(y,t)
\end{equation}
for $\psi\in\Gamma^\infty_c(\tilde E)$ and $\phi\in\mathcal D(\tilde E)=\Gamma^\infty_c(\tilde E')$.
From \eqref{E:ddte-tA} and \eqref{E:tQ} we get
\begin{equation}\label{E:tQAddt}
(A+\tfrac\partial{\partial t})Q\psi=\psi=Q(A+\tfrac\partial{\partial t})\psi
\end{equation}
in the distributional sense, for all $\psi\in\Gamma^\infty_c(\tilde E)$, cf.\ \cite[Equation~(5.12)]{BGS84}.

Comparing \eqref{E:tQAddt} with \eqref{E:parametrix} we obtain, as in \cite[Section~5]{BGS84},
\begin{equation}\label{E:QtQ}
(\id+R_1)Q\psi=\tilde Q\psi=Q(\id+R_2)\psi
\end{equation}
for all $\psi\in\Gamma^\infty_c(\tilde E)$.
Note here, that $\tilde Q$ is properly supported, defining continuous operators on $\Gamma^{-\infty}(\tilde E)$ and $\Gamma^\infty_c(\tilde E)$.
Moreover, $R_1$ and $R_2$ are properly supported smoothing operators, defining continuous operators $\Gamma^{-\infty}_c(\tilde E)\to\Gamma^\infty_c(\tilde E)$ and $\Gamma^{-\infty}(\tilde E)\to\Gamma^\infty(\tilde E)$.
The equality on the left hand side in \eqref{E:QtQ} implies that $Q$ maps $\Gamma^\infty_c(\tilde E)$ continuously into $\Gamma^\infty(\tilde E)$.
Hence, $QR_2$ is a smoothing operator, for it maps $\Gamma^{-\infty}_c(\tilde E)$ continuously into $\Gamma^\infty(\tilde E)$.
Using the equality on the right hand side of \eqref{E:QtQ} we conclude that $\tilde Q-Q$ is a smoothing operator.

Since $Q$ differs from $\tilde Q$ by a smoothing operator, $k_Q$ is smooth away from the diagonal and \eqref{E:qqj} gives an asymptotic expansion 
\begin{equation}\label{E:asymtQ}
\tilde\phi^*\bigl(k_Q|_{\tilde V}\bigr)\sim\sum_{j=0}^\infty\tilde q_j.
\end{equation}
Using Taylor's theorem, we may, by adding terms in $\mathcal P^{-r-j}\bigl(\mathcal T(M\times\R);\tilde E\bigr)$ to $\tilde q_j$, assume that for every integer $N$ there exists an integer $j_N$ such that
\begin{equation}\label{E:asymtQO}
\left(\tilde\phi^*\bigl(k_Q|_{\tilde V}\bigr)-\sum_{j=0}^{j_N}\tilde q_j\right)_{(x,s)}(\tilde g)=O\bigl(|\tilde g|^N\bigr),
\end{equation}
as $\tilde g\to\tilde o_{(x,s)}\in\mathcal T_{(x,s)}(M\times\R)$, uniformly for $(x,s)$ in compact subsets of $M\times\R$.
Here $|-|$ denotes a fiberwise homogeneous norm on $\mathcal T(M\times\R)$.
Comparing \eqref{E:kt}, \eqref{E:tQ}, and \eqref{E:ktQ}, we find 
\begin{equation}\label{E:kQkt}
k_Q(x,s;y,s-t)=\begin{cases}k_t(x,y)dt&\textrm{for $t>0$, and}\\0&\textrm{for $t<0$.}\end{cases}
\end{equation}
In particular, $k_t(x,y)$ is smooth on $M\times M\times(0,\infty)$.
Furthermore, $\tilde\phi^*_{(x,s)}(k_Q|_{\tilde V})$ vanishes on $\mathcal T_xM\times(-\infty,0)\subseteq\mathcal T_xM\times\R=\mathcal T_{(x,s)}(M\times\R)$ and does not depend on $s$, see \eqref{E:tpp} and \eqref{E:tvpvp}.
Combining this with \eqref{E:asymtQO} and \eqref{E:homog}, we conclude that 
\begin{equation}\label{E:vol}
\textrm{$\tilde q_{j,x}:=\tilde q_{j,(x,s)}$ vanishes on $\mathcal T_xM\times(-\infty,0)\subseteq\mathcal T_{(x,s)}(M\times\R)$}
\end{equation}
and does not depend on $s\in\R$, for every $j\in\N_0$ and all $x\in M$.
Note here that $\mathcal T_xM\times(-\infty,0)$ is invariant under scaling, see \eqref{E:deltaMR}.
Using \eqref{E:vol} and \eqref{E:homog}, we see that $\tilde p_{j,(x,s)}$ vanishes on $\mathcal T_xM\times(-\infty,0)$.
Since $\tilde p_{j,(x,s)}$ is polynomial, we conclude $\tilde p_{j,(x,s)}=0$ for every $j\in\N_0$ and all $(x,s)\in M\times\R$.
This shows that there are no log terms in the asymptotic expansion \eqref{E:asymtQ}.
In other words,
\begin{equation}\label{E:ltqj}
\bigl(\delta^{\mathcal T(M\times\R)}_\lambda\bigr)_*\tilde q_j=\lambda^{-r-j}\tilde q_j,\qquad\lambda\neq0.
\end{equation}

Restricting \eqref{E:asymtQO} we obtain, using \eqref{E:tpp} and \eqref{E:deltaMR},
\begin{equation}\label{E:asymtQOxx}
k_Q(x,s;x,s-t)=\sum_{j=0}^{j_N}\tilde\phi_{(x,s)}\bigl(\tilde q_{j,(x,s)}(o_x,t)\bigr)+O\bigl(|t|^{N/r}\bigr),
\end{equation}
as $|t|\to0$, uniformly in $x$.
Using \eqref{E:ltqj}, \eqref{E:tvpvp}, \eqref{E:deltaMR}, and \eqref{E:deltamu}, where $n$ has to be replaced by the homogeneous dimension $n+r$ of $M\times\R$, we obtain
\begin{equation}\label{E:homogt}
\tilde\phi_{(x,s)}\bigl(\tilde q_{j,(x,s)}(o_x,t)\bigr)=t^{(j-n)/r}\tilde\phi_{(x,s)}\bigl(\tilde q_{j,(x,s)}(o_x,1)\bigr)
\end{equation}
for all $x\in M$, $s\in\R$ and $t>0$.
In Equation~\eqref{E:homogt} we are using (left) trivialization of the $1$-density bundle of the osculating group $\mathcal T_{(x,s)}(M\times\mathbb R)$ to identify the two sides.
Hence, defining $q_j\in\Gamma^\infty\bigl(\eend(E)\otimes|\Lambda_M|\bigr)$ by 
\begin{equation}\label{E:qj}
q_j(x)dt:=\tilde\phi_{(x,s)}\bigl(\tilde q_{j,(x,s)}(o_x,1)\bigr),
\end{equation} 
we obtain from \eqref{E:kQkt}, \eqref{E:asymtQOxx}, \eqref{E:homogt}, and \eqref{E:qj}
$$
k_t(x,x)=\sum_{j=0}^{N-1}t^{(j-n)/r}q_j(x)+O\bigl(t^{(N-n)/r}\bigr),
$$
as $t\searrow0$, uniformly in $x$.

Using $\lambda=-1$ in \eqref{E:ltqj} we see that $\tilde q_{j,(x,s)}(o_x,t)=0$ for all odd $j$ and $t\neq0$, see \eqref{E:deltaMR} and \eqref{E:deltamu}.
Hence, $q_j(x)=0$ for all odd $j$, see \eqref{E:qj}.

\begin{remark}\label{R:loccomp}
The asymptotic term $q_j\in\Gamma^\infty\bigl(\eend(E)\otimes|\Lambda_M|\bigr)$ in Theorem~\ref{T:heat} can be read off the homogeneous term $\tilde q_j\in\Sigma^{-r-j}(\tilde E)$ in the asymptotic expansion \eqref{E:qqj} of any parametrix $\tilde Q$ for the heat operator $A+\frac\partial{\partial t}$ on $M\times\R$.
Indeed, the representative $\tilde q_j\in\mathcal K(\mathcal T(M\times\R);\tilde E)$ used in \eqref{E:qj} is uniquely characterized by \eqref{E:homogt} and \eqref{E:vol}.
Note that these representatives are also translation invariant, corresponding precisely to the homogeneous terms on which the Volterra--Heisenberg calculus is based on, see \cite[Section~3]{BGS84} and \cite[Section~5.1]{P08}.
\end{remark}

To complete the proof of Theorem~\ref{T:heat} it remains to show $q_0(x)>0$.
This will be accomplished using the subsequent lemma concerning the heat kernel on the osculating groups, cf.\ \cite[Lemma~6.1.4]{P08}.

\begin{lemma}\label{L:symbolheat}
Consider the Heisenberg principal symbol $\sigma_x^r(A)$ as a left invariant homogeneous differential operator acting on $C^\infty(\mathcal T_xM,E_x)$.
The corresponding heat equation admits a kernel (fundamental solution) in the Schwartz space of $\eend(E_x)$-valued $1$-densities on $\mathcal T_xM$.
More precisely, there exists a family
\begin{equation}\label{E:kS}
k_t^{\sigma^r_x(A)}\in\mathcal S\bigl(|\Lambda_{\mathcal T_xM}|\bigr)\otimes\eend(E_x),
\end{equation}
depending smoothly on $t>0$, which satisfies the heat equation, i.e.,
\begin{equation}\label{E:kequ}
\tfrac\partial{\partial t}k_t^{\sigma_x^r(A)}=-\sigma_x^r(A)k_t^{\sigma_x^r(A)},
\end{equation}
and the initial condition
\begin{equation}\label{E:kinit}
\lim_{t\searrow0}k_t^{\sigma_x^r(A)}\psi=\psi,
\end{equation}
for all $\psi\in\mathcal S(\mathcal T_xM)\otimes E_x$.
Moreover, this kernel is homogeneous, that is,
\begin{equation}\label{E:khomo}
(\delta_{\lambda,x}^{\mathcal T_xM})_*k_t^{\sigma^r_x(A)}=k_{\lambda^rt}^{\sigma^r_x(A)},
\end{equation}
for all $t>0$ and $\lambda\neq0$.
Furthermore, $k_t^{\sigma^r_x(A)}(o_x)>0$ in $|\Lambda_{\mathcal T_xM,o_x}|\otimes\eend(E_x)$, for each $t>0$.
The heat kernel $k_t^{\sigma^r_x(A)}$ is uniquely characterized by \eqref{E:kequ} in the following sense:
Every family of tempered distributions $k_t''\in\mathcal S'(\mathcal T_xM)\otimes\eend(E_x)$ which is continuously differentiable in $t>0$, satisfies the heat equation \eqref{E:kequ} and the initial condition $\lim_{t\searrow0}k_t''=\delta_{o_x}$ in $\mathcal S'(\mathcal T_xM)\otimes\eend(E_x)$, coincides with $k^{\sigma^r_x(A)}_t$.
\end{lemma}

\begin{proof}
Since $\tilde Q$ is a parametrix for $A+\frac\partial{\partial t}$, we have $\sigma^r_{(x,s)}(A+\tfrac\partial{\partial t})\sigma^{-r}_{(x,s)}(\tilde Q)=1$ in $\Sigma^0_{(x,s)}(\tilde E)$, for any fixed $s\in\R$.
As $\sigma^r_{(x,s)}(A+\tfrac\partial{\partial t})=\sigma^r_x(A)+\tfrac\partial{\partial t}$ and $\tilde q_{0,x}=\tilde q_{0,(x,s)}=\sigma^{-r}_{(x,s)}(\tilde Q)$, we get 
$\bigl(\sigma^r_{x}(A)+\tfrac\partial{\partial t}\bigr)\tilde q_{0,x}=\delta_{(o_x,0)}$ mod $\mathcal K^\infty\bigl(\mathcal T_xM\times\R;\tilde E_x\bigr)$.
By homogeneity, see \eqref{E:ltqj}, we actually have
\begin{equation}\label{E:tqequ}
\bigl(\sigma^r_{x}(A)+\tfrac\partial{\partial t}\bigr)\tilde q_{0,x}=\delta_{(o_x,0)}.
\end{equation}
Hence, restricting $\tilde q_{0,x}$ to $\mathcal T_xM\times(0,\infty)\subseteq\mathcal T_xM\times\R=\mathcal T_{(x,s)}(M\times\R)$, more precisely, defining
\begin{equation}\label{E:kdef}
k^{\sigma^r_x(A)}_tdt:=\tilde q_{0,x}|_{\mathcal T_xM\times\{t\}},\qquad t>0,
\end{equation}
we obtain a smooth kernel satisfying the differential equation in \eqref{E:kequ}.
As $\tilde Q$ is also a left parametrix for $A+\frac\partial{\partial t}$, a similar argument shows
\begin{equation}\label{E:rightheq}
\tfrac\partial{\partial t}k_t^{\sigma_x^r(A)}=-k_t^{\sigma_x^r(A)}\sigma^r_x(A),
\end{equation}
where the right hand side involves the convolution of the (smooth) kernel $k_t^{\sigma_x^r(A)}$ with the (distributional) kernel of $\sigma^r_x(A)$, a distribution on $\mathcal T_xM$ which is supported at the unit element $o_x$.
The homogeneity formulated in \eqref{E:khomo} follows from \eqref{E:ltqj}, see also \eqref{E:deltaMR}.

To show that $k_t^{\sigma^r_x(A)}$ is in the Schwartz class, we fix a homogeneous norm $|-|$ and translation invariant volume density $dg$ on $\mathcal T_xM$.
Writing $k_t^{\sigma^r_x(A)}=k_t'dg$ and $\tilde q_{0,x}=q'dgdt$, homogeneity implies 
\begin{equation}\label{E:asdfjkl;}
k_t'(g)=|g|^{-n}q'\bigl(\delta_{|g|}^{-1}(g),|g|^{-r}t\bigr),
\end{equation}
for all $t>0$ and $o_x\neq g\in\mathcal T_xM$.
Since $q'$ is smooth away from the origin in $\mathcal T_xM\times\R$ and vanishes on $\mathcal T_xM\times(-\infty,0)$, see \eqref{E:vol}, it vanishes to infinite order along $\mathcal T_xM\times\{0\}$.
Combining this with \eqref{E:asdfjkl;}, we see that $|g|^mk_t'(g)$ tends to $0$, as $|g|\to\infty$, for each $m\in\N$.
The same argument shows that this remains true if $k'_t$ is replaced with $Bk_t'$ where $B$ is some left invariant homogeneous differential operator on $\mathcal T_xM$.
This shows that $k_t'$ is indeed in the Schwartz space $\mathcal S(\mathcal T_xM)\otimes\eend(E_x)$, whence \eqref{E:kS}.

By homogeneity, see \eqref{E:khomo}, the integral $\int_{\mathcal T_xM}k_t^{\sigma^r_x(A)}$ is independent of $t$.
Testing \eqref{E:tqequ} with $\chi\in\Gamma^\infty_c(\mathbb R,E_x)$, considered as function on $\mathcal T_xM\times\mathbb R$ which is constant in $\mathcal T_xM$, we obtain 
\begin{multline*}
\chi(0)
=\langle\delta_{(o_x,0)},\chi\rangle
=\langle(\sigma^r_x(A)+\tfrac\partial{\partial t})\tilde q_{0,x},\chi\rangle
=-\langle\tilde q_{0,x},\chi'\rangle
\\
=-\int_{\mathcal T_xM}k_t^{\sigma^r_x(A)}\int_0^\infty\chi'(t)dt
=\Bigl(\int_{\mathcal T_xM}k_t^{\sigma^r_x(A)}\Bigr)\chi(0),
\end{multline*}
see \eqref{E:vol} and \eqref{E:kdef}.
Consequently,
\begin{equation}\label{E:kint}
\int_{\mathcal T_xM}k_t^{\sigma^r_x(A)}=\id_{E_x},
\end{equation}
for all $t>0$.
Using \eqref{E:khomo}, this readily implies the initial condition \eqref{E:kinit}.

Let us now turn to the uniqueness of the heat kernel.
The heat equation for $k_t''$ and \eqref{E:rightheq} imply that for fixed $t>0$ the expression $k_s^{\sigma^r_x(A)}k_{t-s}''$ is independent of $s\in(0,t)$.
Moreover, the initial condition for $k_t''$ gives $\lim_{s\nearrow t}k_s^{\sigma^r_x(A)}k_{t-s}''=k_t^{\sigma^r_x(A)}$, while the initial condition \eqref{E:kinit} for $k_t^{\sigma^r_x(A)}$ implies $\lim_{s\searrow0}k_s^{\sigma^r_x(A)}k_{t-s}''=k_t''$.
Combining these observations, we obtain $k_t''=k_t^{\sigma^r_x(A)}$, as well as the semi group property
\begin{equation}\label{E:semigroup}
k^{\sigma^r_x(A)}_sk^{\sigma^r_x(A)}_t=k^{\sigma^r_x(A)}_{s+t},\qquad s,t>0.
\end{equation}
Since $A$ is formally selfadjoint, we also have $\bigl(k^{\sigma^r_x(A)}_t\bigr)^*=k^{\sigma^r_x(A)}_t$, that is,
\begin{equation}\label{E:adj}
k^{\sigma^r_x(A)}_t(g^{-1})=\bigl(k^{\sigma^r_x(A)}_t(g)\bigr)^*,
\end{equation}
for $g\in\mathcal T_xM$ and $t>0$.
Combining \eqref{E:semigroup} and \eqref{E:adj}, we obtain
\begin{equation}\label{E:q0k12}
k_{2t}^{\sigma^r_x(A)}(o_x)
=\int_{g\in\mathcal T_xM}\bigl(k^{\sigma^r_x(A)}_t(g)\bigr)^*k^{\sigma^r_x(A)}_t(g).
\end{equation}
Using \eqref{E:kint}, we conclude $k^{\sigma^r_x(A)}_{2t}(o_x)>0$ in $|\Lambda_{\mathcal T_xM,o_x}|\otimes\eend(E_x)$.
\end{proof}

From \eqref{E:qj} and \eqref{E:kdef} we get, up to the canonical identification $|\Lambda_{M,x}|=|\Lambda_{\mathcal T_xM,o_x}|$,
\begin{equation}\label{E:q0x}
q_0(x)=k_1^{\sigma^r_x(A)}(o_x).
\end{equation}
Hence, $q_0(x)>0$ in $|\Lambda_{M,x}|\otimes\eend(E_x)$, for we have $k_t^{\sigma^r_x(A)}(o_x)>0$ according to Lemma~\ref{L:symbolheat}.
This completes the proof of Theorem~\ref{T:heat}.

\section{Complex powers}\label{S:power}

In this section we will present a proof of Theorem~\ref{T:powers} following the approach in \cite[Section~5.3]{P08}.
Throughout this section $E$ denotes a vector bundle over a closed filtered manifold $M$, and $A\in\DO^r(E)$ is a Rockland differential operator of even Heisenberg order $r>0$ which is formally selfadjoint and non-negative as in Theorem~\ref{T:heat}.

Recall that by hypoellipticity, $\ker(A)$ is a finite dimensional subspace of $\Gamma^\infty(E)$, see \cite[Corollary~2.10]{DH17}.
Let $P$ denote the orthogonal projection onto $\ker(A)$.
To express the complex powers using the Mellin formula, 
\begin{equation}\label{E:MellinAz}
A^{-z}=\frac1{\Gamma(z)}\int_0^\infty t^{z-1}(e^{-tA}-P)dt,\qquad\Re(z)>0,
\end{equation}
we need the following standard estimate for the heat kernel for large time.

\begin{lemma}\label{L:ktxyesti}
Let $p\in\Gamma^\infty(E\boxtimes E')$ denote the Schwartz kernel of the orthogonal projection $P$ onto $\ker(A)$.
Then there exists $\varepsilon>0$ such that  
\begin{equation}\label{E:ktesti}
k_t(x,y)=p(x,y)+O\bigl(e^{-t\varepsilon}\bigr)
\end{equation} 
as $t\to\infty$, uniformly with all derivatives in $x$ and $y$.
\end{lemma}

\begin{proof}
Since $A$ is selfadjoint and non-negative, its spectrum is contained in $[0,\infty)$.
Moreover, since $A$ has compact resolvent, zero is isolated in its spectrum.
Hence, by functional calculus, there exists $\varepsilon>0$ such that
$$
e^{-tA}=P+O\bigl(e^{-t\varepsilon}\bigr)
$$
as $t\to\infty$ with respect to the operator norm topology on $\mathcal B\bigl(L^2(E)\bigr)$.
Writing
$$
e^{-tA}=e^{-A/2}e^{-(t-1)A}e^{-A/2}
$$
and using the fact that $e^{-A/2}$ is a smoothing operator, we conclude
\begin{equation}\label{E:etAO}
e^{-tA}=P+O\bigl(e^{-t\varepsilon}\bigr)
\end{equation}
in $\mathcal L\bigl(\Gamma^{-\infty}(E),\Gamma^\infty(E)\bigr)$ with respect to the topology of uniform convergence on bounded subsets of $\Gamma^{-\infty}(E)$, as $t\to\infty$.
By nuclearity, and since $\Gamma^{-\infty}(E)$ is the strong dual of $\Gamma^\infty(E')$, we have canonical topological isomorphisms
$$
\mathcal L\bigl(\Gamma^{-\infty}(E),\Gamma^\infty(E)\bigr)=\Gamma^\infty(E)\hat\otimes\Gamma^\infty(E')=\Gamma^\infty(E\boxtimes E'),
$$
see \cite[Equations~(50.17) and (51.4)]{T67}.
Hence, \eqref{E:etAO} is equivalent to \eqref{E:ktesti}.
\end{proof}

Suppose $\Re(z)>0$.
In view of the remarks on the spectrum at the beginning of the proof of Lemma~\ref{L:ktxyesti} above, the integral on the right hand side of the Mellin formula \eqref{E:MellinAz} converges to $A^{-z}$ with respect to the operator norm topology on $L^2(E)$.
In particular, $A^{-z}$ is a bounded operator on $L^2(E)$ and its distributional kernel can be expressed in the form
$$
k_{A^{-z}}(x,y)=
\frac1{\Gamma(z)}\int_0^\infty t^{z-1}\bigl(k_t(x,y)-p(x,y)\bigr)dt,\qquad\Re(z)>0,
$$
where the integral on the right hand side converges in the distributional sense.
Splitting the integral as usual, we may rewrite this as
\begin{multline}\label{E:mellink}
k_{A^{-z}}(x,y)=\frac1{\Gamma(z)}\int_0^1t^{z-1}k_t(x,y)dt-\frac1{\Gamma(z)}\frac{p(x,y)}z
\\+\frac1{\Gamma(z)}\int_1^\infty t^{z-1}\bigl(k_t(x,y)-p(x,y)\bigr)dt,\qquad\Re(z)>0.
\end{multline}
In view of \eqref{E:ktesti} the second integral converges with respect to the $C^\infty$-topology and defines a smooth function on $M\times M\times\C$ which depends holomorphically on $z$.
If $x\neq y$, then $k_t(x,y)$ vanishes to infinite order at $t=0$, see~\eqref{E:kQkt}, hence the first integral in \eqref{E:mellink} converges in the $C^\infty$-topology and defines a smooth function on $\{x\neq y\}\times\C$ which depends holomorphically on $z$.

To study the behavior of
\begin{equation}\label{E:defKs}
K_z(x,y):=\int_0^1t^{z-1}k_t(x,y)dt,\qquad\Re(z)>0,
\end{equation}
near the diagonal, we fix a real number $s$, use \eqref{E:kQkt}, and move to exponential coordinates, see \eqref{E:tpp} and \eqref{E:tvpvp},
\begin{equation}\label{E:mellinexp}
(\phi^*K_z)_x(g)=\int_{t\in[0,1]}t^{z-1}\bigl(\tilde\phi^*(k_Q|_{\tilde V})\bigr)_{(x,s)}(g,t),\qquad\Re(z)>0,
\end{equation}
where $x\in M$, $g\in\mathcal T_xM$.
Given an integer $N$, we choose an integer $j_N$ such that $\tilde\phi^*(k_Q|_{\tilde V})-\sum_{j=0}^{j_N}\tilde q_j$ is of class $C^N$, see \eqref{E:asymtQO}, and rewrite \eqref{E:mellinexp} in the form
\begin{equation}\label{E:splitintk}
(\phi^*K_z)_x(g)
=\int_{t\in[0,1]}t^{z-1}\left(\tilde\phi^*(k_Q|_{\tilde V})-\sum_{j=0}^{j_N}\tilde q_j\right)_{(x,s)}(g,t)
+\sum_{j=0}^{j_N}(K_{z,j})_x(g)
\end{equation}
where $K_{z,j}$ the distributional section of $\eend(\pi^*E)\otimes\Omega_\pi$ defined by
\begin{equation}\label{E:defKsj}
(K_{z,j})_x(g):=\int_{t\in[0,1]}t^{z-1}\tilde q_{j,(x,s)}(g,t),\qquad\Re(z)>-j/r.
\end{equation}
In view of \eqref{E:ltqj} and \eqref{E:deltaMR} we have
\begin{equation}\label{E:Ksjhom}
\bigl(\bigl(\delta^{\mathcal TM}_\lambda\bigr)_*K_{z,j}-\lambda^{-zr-j}K_{z,j}\bigr)_x(g)=
\lambda^{-zr-j}\int_{t\in[1,\lambda^r]}t^{z-1}\tilde q_{j,(x,s)}(g,t)
\end{equation}
for all $\lambda\geq1$, and a similar formula holds for $0<\lambda\leq1$.
Since the integral on the right hand side defines a smooth section of $\eend(E)\otimes\Omega_\pi$, we conclude $K_{z,j}\in\Sigma^{-zr-j}(E)$.
Since the integral term in \eqref{E:splitintk} is of class $C^N$, we have an asymptotic expansion $\phi^*K_z\sim\sum_{j=0}^\infty K_{z,j}$, provided $\Re(z)>0$.
Combining this with \eqref{E:mellink} and \eqref{E:defKs} we obtain an asymptotic expansion
\begin{equation}\label{E:Azasym}
\phi^*(k_{A^{-z}})\sim\sum_{j=0}^\infty\frac{K_{z,j}}{\Gamma(z)}.
\end{equation}
This shows $A^{-z}\in\Psi^{-zr}(E)$, provided $\Re(z)>0$.
Using $A^{k-z}=A^kA^{-z}$ and $A^k\in\Psi^{kr}(E)$ for $k\in\N$, we see that this remains true for all complex $z$.
%\red{We also conclude that $A^{-z}$ is a holomorphic family of pseudodifferential operators.}

For $\Re(z)>n/r$ the Heisenberg order of $A^{-z}$ has real part smaller than $-n$, hence this power has a continuous kernel, see \cite[Proposition~3.9(d)]{DH17}, and the Mellin formula \eqref{E:MellinAz} gives
\begin{equation}\label{E:mellin}
k_{A^{-z}}(x,x)=\frac1{\Gamma(z)}\int_0^\infty t^{z-1}\bigl(k_t(x,x)-p(x,x)\bigr)dt,\qquad\Re(z)>n/r.
\end{equation}
Splitting the integral as usual, we may write
\begin{multline}\label{E:splitint}
\int_0^\infty t^{z-1}\bigl(k_t(x,x)-p(x,x)\bigr)dt
=\int_1^\infty t^{z-1}\bigl(k_t(x,x)-p(x,x)\bigr)dt-\frac{p(x,x)}z
\\+\int_0^1t^{z-1}\left(k_t(x,x)-\sum_{j=0}^{N-1}t^{(j-n)/r}q_j(x)\right)dt
+\sum_{j=0}^{N-1}\frac{q_j(x)}{z-(n-j)/r}.
\end{multline}
where $q_j\in\Gamma^\infty\bigl(\eend(E)\otimes|\Lambda_M|\bigr)$ are as in Theorem~\ref{T:heat}.
In view of the estimate \eqref{E:ktesti}, the first integral on the right hand side of \eqref{E:splitint} converges for all complex $z$ and defines an entire function.
In view of Theorem~\ref{T:heat}, the second integral on the right hand side of \eqref{E:splitint} converges for $\Re(z)>(n-N)/r$ and defines a holomorphic function on $\{\Re(z)>(n-N)/r\}$.
Note that these considerations are all uniform in $x$.
Combining this with \eqref{E:mellin}, we see that $k_{A^{-z}}(x,x)$ can be extended meromorphically to all complex $z$ with poles and special values as specified in Theorem~\ref{T:powers}.
Here we also use the classical fact that $1/\Gamma(z)$ is an entire function with zero set $-\N_0$ and $(1/\Gamma)'(-l)=(-1)^ll!$ for all $l\in\N_0$.

To complete the proof of Theorem~\ref{T:powers}, it remains to show that the powers $A^{-z}$ form a holomorphic family of Heisenberg pseudodifferential operators in a sense analogous to \cite[Section~4]{P08}. 
This will be established in Section~\ref{S:fam} below.

\section{Holomorphic families of Heisenberg pseudodifferential operators}\label{S:fam}

In this section we will extend the concept of a holomorphic family of pseudodifferential operators \cite[Chapter~4]{P08} to the Heisenberg calculus on general filtered manifolds.
We will show that the complex powers discussed above do indeed form a holomorphic family as stated in Theorem~\ref{T:powers}.
Holomorphic families will also be used to construct a non-commutative residue in Section~\ref{S:res} below.

\begin{remark}\label{R:top}
Below we will use several spaces of distributions which are conormal to certain closed submanifolds.
It will be convenient to equip theses spaces with the structure of a complete locally convex vector space.
To describe this topology, suppose $\xi$ is a vector bundle over a smooth manifold $N$, and suppose $S$ is a closed submanifold of $N$.
We let $\mathcal K\subseteq\Gamma^{-\infty}(\xi)$ denote the vector space of all distributional sections of $\xi$ whose wave front set is contained in the conormal of $S$.
In particular, these distributions are smooth on $N\setminus S$, hence we have a canonical map 
\begin{equation}\label{E:topa}
\mathcal K\to\Gamma^\infty(\xi|_{N\setminus S}).
\end{equation}
We let $\pi\colon T^\perp S\to S$ denote the normal bundle of $S$ in $N$, that is, $T^\perp S:=TN|_S/TS$.
Suppose $\varphi\colon T^\perp S\to W\subseteq N$ is a tubular neighborhood, i.e.\ a diffeomorphism onto an open neighborhood $W$ of $S$ in $N$, which restricts to the identity along $S$.
Moreover, let $\phi$ be a vector bundle isomorphism $\pi^*(\xi|_S)\cong\varphi^*(\xi|_W)$.
If $a\in\mathcal K$, then $\phi^*a$ is a distributional sections of $\pi^*(\xi|_S)$ whose wave front set is conormal to the zero section $S\subseteq T^\perp S$.
In particular, $\phi^*a$ is $\pi$-fibered.
Hence, we obtain a map 
\begin{equation}\label{E:topb}
\phi^*\colon\mathcal K\to\Gamma^{-\infty}_\pi(\pi^*(\xi|_S)),
\end{equation}
where the right hand side denotes the space of all $\pi$-fibered distributional sections of $\pi^*(\xi|_S)$.
Recall that elements $a\in\Gamma^{-\infty}_\pi(\pi^*(\xi|_S))$ can be considered as families of distributions $a_x\in C^{-\infty}(T^\perp_xS,\xi_x)$ on the fibers $T_x^\perp S$ with values in $\xi_x$ which depend smoothly on $x\in S$.
Every $\chi\in\Gamma^\infty_c(\Omega_\pi)$ provides a map $\Gamma^{-\infty}_\pi(\pi^*(\xi|_S))\to\Gamma^\infty(\xi|_S)$, $a\mapsto\pi_*(a\chi)$, where $\pi_*$ denotes integration along the fibers of $\pi$.
We equip $\Gamma^{-\infty}_\pi(\pi^*(\xi|_S))$ with the weakest locally convex topology such that these maps are all continuous.
Finally, we equip $\mathcal K$ with the coarsest locally convex topology such that the maps \eqref{E:topa} and \eqref{E:topb} become continuous.
It is well known that $\mathcal K$ is complete.
Moreover, the topology does not depend on the choice of a tubular neighborhood.
\end{remark}

Suppose $E$ and $F$ are two vector bundles over a filtered manifold $M$.
Moreover, let $\Omega$ be a domain in the complex plane and suppose $s\colon\Omega\to\mathbb C$ is a holomorphic function.
We intend to make precise when a family of operators $A_z\in\Psi^{s(z)}(E,F)$, parametrized by $z\in\Omega$, is considered to be a holomorphic family.

Recall from Section~\ref{S:calculus} that $\mathcal K(\mathbb TM;E,F)$ denotes the space of all distributional sections of the vector bundle \eqref{E:EEFFO} whose wave front set is contained in the conormal of the units, $M\times\R\subseteq\mathbb TM$.
We equip $\mathcal K(\mathbb TM;E,F)$ with the topology described in Remark~\ref{R:top}, and let $\mathcal K_\Omega(\mathbb TM;E,F)$ denote the space of holomorphic curves from $\Omega$ into $\mathcal K(\mathbb TM;E,F)$. 
Moreover, we let $\mathcal K_\Omega^\infty(\mathbb TM;E,F)$ denote the space of holomorphic curves from $\Omega$ into $\mathcal K^\infty(\mathbb TM;E,F)$ where the latter space carries the $C^\infty$-topology.

\begin{definition}\label{D:hf}
Let $\Omega$ be a domain in the complex plane and suppose $s\colon\Omega\to\C$ is a holomorphic function.
A family of operators $A_z\in\Psi^{s(z)}(E,F)$, parametrized by $z\in\Omega$, is called a \emph{holomorphic family of Heisenberg pseudodifferential operators} if there exists a family $\mathbb K_z\in\mathcal K_\Omega(\mathbb TM;E,F)$ such that $\ev_1(\mathbb K_z)$ is the Schwartz kernel of $A_z$ for all $z\in\Omega$, and $\mathbb K_z$ is essentially homogeneous of order $s(z)$ in the sense that $(\delta^{\mathbb TM}_\lambda)_*\mathbb K_z-\lambda^{s(z)}\mathbb K_z$ is a family in $\mathcal K^\infty_\Omega(\mathbb TM;E,F)$, for all $\lambda>0$.
\end{definition}

\begin{lemma}\label{L:hf:comp}
Suppose $A_z\in\Psi^{s(z)}(E,F)$ and $B_z\in\Psi^{t(z)}(F,G)$ are two holomorphic families of Heisenberg pseudodifferential operators, where $z\in\Omega$ and $s,t\colon\Omega\to\C$.
Then $B_zA_z\in\Psi^{(t+s)(z)}(E,G)$ is a holomorphic family of Heisenberg pseudodifferential operators, provided at least one of the two families is properly supported (locally uniformly in $z$).
Moreover, the transpose $A_z^t\in\Psi^{s(z)}(F',E')$ is a holomorphic family of Heisenberg pseudodifferential operators.
\end{lemma}

\begin{proof}
Convolution and transposition induce bounded (bi)linear maps:
\begin{align*}
\mathcal K(\mathbb TM;F,G)\times\mathcal K_\prop(\mathbb TM;E,F)&\to\mathcal K(\mathbb TM;E,G)
\\
\mathcal K^\infty(\mathbb TM;F,G)\times\mathcal K_\prop(\mathbb TM;E,F)&\to\mathcal K^\infty(\mathbb TM;E,G)
\\
\mathcal K(\mathbb TM;F,G)\times\mathcal K_\prop^\infty(\mathbb TM;E,F)&\to\mathcal K^\infty(\mathbb TM;E,G)
\\
\mathcal K^\infty(\mathbb TM;F,G)\times\mathcal K_\prop^\infty(\mathbb TM;E,F)&\to\mathcal K^\infty(\mathbb TM;E,G)
\\
\mathcal K(\mathbb TM;E,F)&\to\mathcal K(\mathbb TM;F',E')
\\
\mathcal K^\infty(\mathbb TM;E,F)&\to\mathcal K^\infty(\mathbb TM;F',E')
\end{align*}
Since holomorphic curves remain holomorphic when composed with bounded (bi)linear mappings, convolution and transposition thus induce maps:
\begin{align*}
\mathcal K_\Omega(\mathbb TM;F,G)\times\mathcal K_{\Omega,\prop}(\mathbb TM;E,F)&\to\mathcal K_\Omega(\mathbb TM;E,G)
\\
\mathcal K_\Omega^\infty(\mathbb TM;F,G)\times\mathcal K_{\Omega,\prop}(\mathbb TM;E,F)&\to\mathcal K_\Omega^\infty(\mathbb TM;E,G)
\\
\mathcal K_\Omega(\mathbb TM;F,G)\times\mathcal K^\infty_{\Omega,\prop}(\mathbb TM;E,F)&\to\mathcal K_\Omega^\infty(\mathbb TM;E,G)
\\
\mathcal K_\Omega^\infty(\mathbb TM;F,G)\times\mathcal K^\infty_{\Omega,\prop}(\mathbb TM;E,F)&\to\mathcal K_\Omega^\infty(\mathbb TM;E,G)
\\
\mathcal K_\Omega(\mathbb TM;E,F)&\to\mathcal K_\Omega(\mathbb TM;F',E')
\\
\mathcal K_\Omega^\infty(\mathbb TM;E,F)&\to\mathcal K_\Omega^\infty(\mathbb TM;F',E')
\end{align*}
The lemma follows at once.
\end{proof}

Recall from Section~\ref{S:calculus} that $\mathcal K(\mathcal TM;E,F)$ denotes the space of all distributional sections of the vector bundle $\hom(\pi^*E,\pi^*F)\otimes\Omega_\pi$ over $\mathcal TM$ whose wave front set is contained in the conormal of the units, $M\subseteq\mathcal TM$.
We equip $\mathcal K(\mathcal TM;E,F)$ with the topology described in Remark~\ref{R:top}, and let $\mathcal K_\Omega(\mathcal TM;E,F)$ denote the space of holomorphic curves from $\Omega$ into $\mathcal K(\mathcal TM;E,F)$.
Moreover, we let $\mathcal K_\Omega^\infty(\mathcal TM;E,F)$ denote the space of holomorphic curves from $\Omega$ into $\mathcal K^\infty(\mathcal TM;E,F)$ where the latter space carries the $C^\infty$-topology.

\begin{definition}\label{D:ehO}
Let $\Omega$ be a domain in the complex plane, and suppose $s\colon\Omega\to\C$ is a holomorphic function.
A family $k_z\in\mathcal K_\Omega(\mathcal TM;E,F)$ is called \emph{essentially homogeneous} of order $s(z)$ if $(\delta_\lambda)_*k_z-\lambda^{s(z)}k_z\in\mathcal K_\Omega^\infty(\mathcal TM;E,F)$, for all $\lambda>0$.
This generalizes \cite[Definition~4.4.1]{P08} where the term \emph{almost homogeneous} is used.
\end{definition}

\begin{lemma}\label{L:hf}
Let $s\colon\Omega\to\C$ be a holomorphic function, consider a family of operators $A_z\in\Psi^{s(z)}(E,F)$ parametrized by $z\in\Omega$, and let $k_z$ denote the Schwartz kernel of $A_z$.
Then the following are equivalent:
\begin{enumerate}[(a)]
\item\label{L:hf:a}
$A_z$ is a holomorphic family of Heisenberg pseudodifferential operators.
\item\label{L:hf:b}
Away from the diagonal, the Schwartz kernels $k_z$ are smooth and depend holomorphically on $z\in\Omega$.
Moreover, with respect to some (and then every) exponential coordinates adapted to the filtration as in Section~\ref{S:calculus}, see \eqref{D:expcoor}, we have an asymptotic expansion of the form
\begin{equation}\label{E:hf:asym}
\phi^*(k_z|_V)\sim\sum_{j=0}^\infty k_{j,z}
\end{equation}
where $k_{j,z}\in\mathcal K_\Omega(\mathcal TM;E,F)$ is an essentially homogeneous family of order $s(z)-j$, cf.~Definition~\ref{D:ehO}.
More precisely, for every $z_0\in\Omega$ and every integer $N$ there exist a neighborhood $W$ of $z_0$ in $\Omega$ such that, for sufficiently large $J\in\N$, the expression $\phi^*(k_z|_V)-\sum_{j=0}^{J-1}k_{j,z}$ restricts to a holomorphic curve from $W$ into the space of $C^N$-sections of $\hom(\pi^*E,\pi^*F)\otimes\Omega_\pi$ over $U$.
\item\label{L:hf:c}
Away from the diagonal, the Schwartz kernels $k_z$ are smooth and depend holomorphically on $z\in\Omega$.
Moreover, with respect to some (and then every) exponential coordinates adapted to the filtration as in Section~\ref{S:calculus}, see \eqref{D:expcoor}, and for some (and then any) properly supported bump function $\chi\in C^\infty_\prop(\mathcal TM)$ with $\supp(\chi)\subseteq U$ and $\chi\equiv1$ in a neighborhood of the zero section, the fiber wise Fourier transform,
$$
\mathcal F\bigl(\chi\cdot\phi^*(k_z|_V)\bigr)(\xi)
=\int_{X\in\mathfrak t_xM}e^{-2\pi\mathbf i\langle\xi,X\rangle}\bigl(\chi\cdot\phi^*(k_z|_V)\bigr)(\exp(X)),\quad\xi\in\mathfrak t_x^*M,
$$
is smooth on $\Omega\times\mathfrak t^*M$, depends holomorphically on $z$, and admits a uniform asymptotic expansion of the form
\begin{equation}\label{E:hf:asymhat}
\mathcal F\bigl(\chi\cdot\phi^*(k_z|_V)\bigr)\sim\sum_{j=0}^\infty\hat k_{j,z},
\end{equation}
where $\hat k_{j,z}\in\Gamma^\infty(\hom(p^*E,p^*F)|_{\mathfrak t^*M\setminus M})$ are homogeneous of order $s(z)-j$, i.e., $(\dot\delta'_\lambda)^*\hat k_{j,z}=\lambda^{s(z)-j}\hat k_{j,z}$ for all $\lambda>0$.
Here $\dot\delta_\lambda'$ denotes the grading automorphism on $\mathfrak t^*M$ dual to $\dot\delta_\lambda$ on $\mathfrak tM$, and $p\colon\mathfrak t^*M\to M$ denotes the vector bundle projection.
More precisely, for all $a,N\in\N$, and for every homogeneous differential operator $D$ acting on $\Gamma^\infty(\hom(p^*E,p^*F))$, i.e., $(\dot\delta'_\lambda)_*D=\lambda^aD$, for every compact $K\subseteq\Omega$, and for some (and then every) fiberwise homogeneous norm on $\mathfrak t^*M$, and for some (and then every) fiberwise Hermitian metric on $\hom(E,F)$, there exists a constant $C\geq0$ such that 
\begin{equation}\label{E:esti}
\left|D\left(\mathcal F\bigl(\chi\cdot\phi^*(k_z|_V)\bigr)-\sum_{j=0}^{N-1}\hat k_{j,z}\right)(\xi)\right|
\leq C|\xi|^{s(z)-N-a}
\end{equation}
holds for all $z\in K$ and all $\xi\in\mathfrak t^*M$ with $|\xi|\geq1$.
\end{enumerate}
\end{lemma}

\begin{proof}
To see \itemref{L:hf:a}$\Rightarrow$\itemref{L:hf:b} assume $A_z$ is a holomorphic family of Heisenberg pseudodifferential operators with Schwartz kernels $k_z$.
Choose $\mathbb K\in\mathcal K_\Omega(\mathbb TM;E,F)$ as in Definition~\ref{D:hf}, i.e., $\ev_1(\mathbb K_z)=k_z$ and $(\delta^{\mathbb TM}_\lambda)_*\mathbb K_z-\lambda^{s(z)}\mathbb K_z\in\mathcal K^\infty_\Omega(\mathbb TM;E,F)$ for all $\lambda>0$.
Away from the units, $\mathbb K_z$ is smooth and depends holomorphically on $z\in\Omega$.
Clearly, this implies that the same is true for $k_z$, i.e., away from the diagonal $k_z$ is smooth and depends holomorphically on $z\in\Omega$.
Now consider adapted exponential coordinates as in Section~\ref{S:calculus}.
By conormality, $\mathbb K$ has a Taylor expansion along $\mathcal T^\op M\times\{0\}$,
\begin{equation}\label{E:taylorO}
\Phi^*\bigl(\mathbb K_z|_{\mathbb V}\bigr)\sim\sum_{j=0}^\infty\mathbb K_{j,z}t^j
\end{equation}
where $\mathbb K_{j,z}\in\mathcal K_\Omega(\mathcal T^\op M;E,F)$.
By essential homogeneity of $\mathbb K$, the family $\mathbb K_{j,z}$ is essentially homogeneous of order $s(z)-j$, that is, $(\delta_\lambda)_*\mathbb K_{j,z}-\lambda^{s(z)-j}\mathbb K_{j,z}\in\mathcal K^\infty_\Omega(\mathcal T^\op M;E,F)$.
For each $J\in\N$ we have $\Phi^*\bigl(\mathbb K_z|_{\mathbb V}\bigr)-\sum_{j=0}^{J-1}\mathbb K_{j,z}t^j=t^J\mathbb L_{J,z}$ where $\mathbb L_{J,z}\in\mathcal K_\Omega(\mathbb U;E,F)$ is essentially homogeneous of order $s(z)-J$ on $\mathbb U$.
Let $W$ be an open subset with compact closure in $\Omega$, and suppose $N\in\N$.
Then, for sufficiently large $J$, the family $\mathbb L_{J,z}$ restricts to a holomorphic curve from $W$ into the space of $C^N$-sections over $\mathbb U$.
Indeed, this follows from a parametrized version of \cite[Theorem~52]{EY15v5}.
We conclude that $\Phi^*\bigl(\mathbb K_z|_{\mathbb V}\bigr)-\sum_{j=0}^{J-1}\mathbb K_{j,z}t^j$ restricts to a holomorphic curve from $W$ into the the space of $C^N$-sections over $\mathbb U$.
Evaluating at $t=1$, we obtain the asymptotic expansion \eqref{E:hf:asym} with $k_{j,z}=\nu^*(\mathbb K_{j,z})$.

To see \itemref{L:hf:b}$\Rightarrow$\itemref{L:hf:a} we observe that $t^jk_{j,z}\in\mathcal K_\Omega(\mathbb U;E,F)$ is essentially homogeneous of order $s(z)$. 
Parametrizing the proof of \cite[Theorem~59]{EY15v5}, we obtain a holomorphic family $\mathbb L_z\in\mathcal K_\Omega(\mathbb U;E,F)$ which is essentially homogeneous of order $s(z)$ and such that $\ev_1(\mathbb L_z)\sim\sum_{j=0}^\infty k_{j,z}$ in the same sense as \eqref{E:hf:asym}.
Clearly, this implies that $\phi^*(k_z|_V)-\ev_1(\mathbb L_z)$ is a holomorphic curve into the space of smooth sections over $\mathbb U$.
Adding a term in $\mathcal K_\Omega^\infty(\mathbb U;E,F)$ to $\mathbb L_z$, we may, moreover, assume $\phi^*(k_z|_V)=\ev_1(\mathbb L_z)$.
Using a bump function one readily constructs $\mathbb K_z\in\mathcal K_\Omega(\mathbb TM;E,F)$, essentially homogeneous of order $s(z)$, such that $\Phi^*(\mathbb K_z)$ coincides with $\mathbb L_z$ on neighborhood of $M\times\R$.
Hence, $k_z$ coincides with $\ev_1(\mathbb K_z)$ in a neighborhood of the diagonal.
Clearly, this implies \itemref{L:hf:a}.

Let us now turn to the implication \itemref{L:hf:b}$\Rightarrow$\itemref{L:hf:c}:
Note first, that $\chi\cdot\phi^*(k_z|_V)\in\mathcal K_{\Omega,\prop}(\mathcal TM;E,F)$, hence its fiberwise Fourier transform is smooth on $\mathfrak t^*M\times\Omega$ and depends holomorphically on $z\in\Omega$.
Fix a bump function $\rho\in C^\infty(\mathfrak t^*M)$ which vanishes in a neighborhood of the zero section and such that $\rho(\xi)=1$ whenever $|\xi|\geq1$.
Since $k_{z,j}$ is essentially homogeneous of order $s(z)-j$, there exist unique $\hat k_{j,z}\in\Gamma^\infty(\hom(p^*E,p^*F)|_{\mathfrak t^*M\setminus M})$, strictly homogeneous of order $s(z)-j$, such that
$\rho\cdot\bigl(\mathcal F(\chi\cdot k_{z,j})-\hat k_{z,j}\bigr)$ is in the fiberwise Schwartz space $\mathcal S(\mathfrak t^*M;\hom(p^*E,p^*F))$.
More precisely, for every compact $K\subseteq\Omega$, every homogeneous differential operator $D$, and every integer $b\in\N$, there exists a constant $C\geq0$ such that 
$$
\left|D\left(\mathcal F\bigl(\chi\cdot k_{j,z}\bigr)-\hat k_{j,z}\right)(\xi)\right|\leq C|\xi|^{-b}
$$
holds for all $z\in K$ and $\xi\in\mathfrak t^*M$ with $|\xi|\geq1$.
In view of the expansion \ref{E:hf:asym}, we conclude that for every compact $K\subseteq\Omega$, every $b\in\N$, every homogeneous differential operator $D$, and every sufficiently large $J\in\N$, there exists a constant $C\geq0$ such that 
$$
\left|D\left(\mathcal F\left(\chi\cdot\left(\phi^*(k_z|_V)-\sum_{j=0}^{J-1}k_{j,z}\right)\right)\right)(\xi)\right|\leq C|\xi|^{-b}
$$
holds for all $z\in K$ and $\xi\in\mathfrak t^*M$ with $|\xi|\geq1$.
Combining the preceding two estimates, we see that for every compact $K\subseteq\Omega$, every $b\in\N$, every homogeneous differential operator $D$, and every sufficiently large $J\in\N$, there exists a constant $C\geq0$ such that
$$
\left|D\left(\mathcal F\bigl(\chi\cdot\phi^*(k_z|_V)\bigr)-\sum_{j=0}^{J-1}\hat k_{j,z}\right)(\xi)\right|
\leq C|\xi|^{-b}
$$
holds for all $z\in K$ and $\xi\in\mathfrak t^*M$ with $|\xi|\geq1$.
If $D$ is homogeneous of degree $a\in\N$, i.e., $(\dot\delta'_\lambda)_*D=\lambda^aD$, then $|D\hat k_{j,z}(\xi)|\leq C|\xi|^{s(z)-j-a}$, whence \eqref{E:hf:asymhat}.

Let us now turn to the implication \itemref{L:hf:c}$\Rightarrow$\itemref{L:hf:b}:
Note first that the homogeneous terms $\hat k_{j,z}$ in the asymptotic expansion \eqref{E:hf:asymhat} depend holomorphically on $z\in\Omega$.
This follows from \eqref{E:esti}, cf.\ \cite[Remark~4.2.2]{P08}.
Defining $k_{j,z}:=\mathcal F^{-1}(\rho\cdot\hat k_{j,z})$, we thus have $k_{j,z}\in\mathcal K_\Omega(\mathcal TM;E,F)$.
Moreover, for each $\lambda>0$,
$$
(\delta_\lambda)_*k_{j,z}-\lambda^{s(z)-j}k_{j,z}
=\mathcal F^{-1}\left(\lambda^{s(z)-j}\cdot\bigl((\dot\delta'_\lambda)^*\rho-\rho\bigr)\cdot\hat k_{j,z}\right)
$$
is a holomorphic curve in the fiberwise Schwartz space $\mathcal S(\mathcal TM;\hom(\pi^*E,\pi^*F)\otimes\Omega_\pi)$.
In particular, $k_{j,z}$ is essentially homogeneous of order $s(z)-j$.
If $W$ is an open subset with compact closure in $\Omega$, then \eqref{E:esti} implies that, for sufficiently large $J\in\N$, the expression $\chi\cdot\phi^*(k_z|_V)-\sum_{j=0}^{J-1}k_{j,z}$ is a holomorphic curve in the space of $C^N$ sections, whence \eqref{E:hf:asym}.
\end{proof}

\begin{remark}
As pointed out in the proof above, the homogeneous terms $\hat k_{j,z}$ in the asymptotic expansion \eqref{E:hf:asymhat} depend holomorphically on $z\in\Omega$.
\end{remark}

The characterization in Lemma~\ref{L:hf} shows that the concept of holomorphic families considered here generalizes the definition for Heisenberg manifolds, see \cite[Section~4]{P08}.

Let us now complete the proof of Theorem~\ref{T:powers} by establishing the following generalization of \cite[Theorem~5.3.1]{P08}.

\begin{lemma}\label{L:Azhf}
The complex powers $A^{-z}\in\Psi^{-zr}(E)$ considered in Theorem~\ref{T:powers} constitute a holomorphic family.
\end{lemma}

\begin{proof}
Writing $A^{-z}=A^kA^{-(z+k)}$ and using Lemma~\ref{L:hf:comp}, we see that it suffices to show that the powers form a holomorphic family on $\Omega=\{\Re(z)>0\}$.
To prove this, we will use the characterization in Lemma~\ref{L:hf}\itemref{L:hf:b} and proceed as in \cite[Lemma~5.3.2]{P08}.

As observed above, away from the diagonal the Schwartz kernels $k_{A^{-z}}$, see \eqref{E:mellink}, are smooth and depend holomorphically on $z\in\C$.
Moreover, the distributions $K_{z,j}$ defined in \eqref{E:defKsj} form a family in $\mathcal K_\Omega(\mathcal TM;E,F)$ which is essentially homogeneous of order $-rz-j$ in the sense of Definition~\ref{D:ehO}, see \eqref{E:Ksjhom}.
Furthermore, the $C^N$ sections defined by the integral in \eqref{E:splitintk} depend holomorphically on $z\in\Omega$.
We conclude that the asymptotic expansion \eqref{E:Azasym} established above is in fact the asymptotic expansion required in Lemma~\ref{L:hf}\itemref{L:hf:b}.
\end{proof}

\section{Non-commutative residue}\label{S:res}

In this section we consider the analogue of Wodzicki's non-commutative residue \cite{Wo87,Gu85} for filtered manifolds.
Let $E$ be a vector bundle over a closed filtered manifold $M$ of homogeneous dimension $n$. 
We will fix a strictly positive Rockland differential operator $D\in\DO^r(E)$ of even Heisenberg order $r>0$.
For a pseudodifferential operator $A\in\Psi^k(E)$ of integer Heisenberg order $k\in\mathbb Z$ the operator $AD^{-z}\in\Psi^{k-rz}(E)$ is trace class for $\mathfrak R(z)>(n+k)/r$ and hence defines a holomorphic function
\[
\zeta_A(z):=\tr(AD^{-z}),\qquad \mathfrak R(z)>(n+k)/r,
\]
see Theorem~\ref{T:powers} above and \cite[Proposition~3.9(d)]{DH17}.
In view of Lemma~\ref{L:hf:comp}, the subsequent proposition can be applied to $R(z)=AD^{-z}$.

We continue to use the canonical identification $|\Lambda_M|=o^*\Omega_\pi$ between the bundle of $1$-densities on $M$ and the restriction along the zero section of the bundle of vertical densities on $\mathcal TM$, see Section~\ref{S:calculus}.
In particular, we will use the induced identification 
\begin{equation}\label{E:PnLM}
\Gamma^\infty(|\Lambda_M|\otimes\eend(E))=\mathcal P^{-n}(\mathcal TM;E,E).
\end{equation}
Recall that the right hand side in \eqref{E:PnLM} denotes the space of smooth sections of $\eend(\pi^*E)\otimes\Omega_\pi$ which are homogeneous of degree $-n$ and thus fiberwise translation invariant (constant).

\begin{proposition}\label{prop:meromorphic}
Let $E$ be a vector bundle over a closed filtered manifold $M$ of homogeneous dimension $n$.
Moreover, suppose $r>0$ and let $R(z)\in\Psi^{d-rz}(E)$ be a holomorphic family of Heisenberg pseudodifferential operators on $M$ with Schwartz kernels $k_z$.
Then the restriction to the diagonal, $k_z(x,x)$, provides a holomorphic curve in $\Gamma^\infty(\eend(E)\otimes|\Lambda_M|\bigr)$ for $\Re(z)>(n+d)/r$.
This curve can be extended meromorphically to the entire complex plane with at most simple poles located at $z_j:=(n+d-j)/r$ with $j\in\mathbb N_0$.
Moreover, the residue at $z_j$ can be computed locally by the expected formula, 
\begin{equation}\label{E:extres}
\res_{z=z_j}\bigl(k_z(x,x)\bigr)=\frac1r\int_{\xi\in\mathfrak t_x^*M,|\xi|=1}\hat k_{j,z_j}(\xi)d\xi=\frac{p_j(o_x)}r,
\end{equation}
where $\hat k_{j,z}$ are as in Lemma~\ref{L:hf}\itemref{L:hf:c}, and $p_j\in\mathcal P^{-n}(\mathcal TM;E,E)$ are characterized by $(\delta_\lambda)_*k_{j,z_j}=\lambda^{-n}k_{j,z_j}+\lambda^{-n}\log|\lambda|\,p_j$ for $\lambda>0$ with $k_{j,z}$ as in Lemma~\ref{L:hf}\itemref{L:hf:b}. 
\footnote{By adding smooth cosymbols, the asymptotic terms $k_{j,z_j}$ may be chosen have this property. 
They are not unique, but $p_j$ is without ambiguity, cf.\ \cite[Lemma~3.8]{DH17} or Lemma~\ref{L:reps} above.}

In particular, the function $\tr(R(z))$ is analytic on the half plane $\mathfrak R(z)>(n+d)/r$, and has a meromorphic continuation to all of the complex plane with only simple poles possibly located at $z_j$ with $j\in\mathbb N_0$.
Moreover, 
\begin{equation}\label{E:restr}
\res_{z=z_j}\tr(R(z))=\frac1r\int_{x\in M}\int_{\xi\in\mathfrak t_x^*M,|\xi|=1}\tr_E\bigl(\hat k_{j,z_j}(\xi)\bigr)d\xi=\frac1r\int_M\tr_E(p_j|_M).
\end{equation}
\end{proposition}

\begin{remark}
The volume density $d\xi$ used in \eqref{E:extres} and \eqref{E:restr} above is defined such that a $1$-density on $\mathfrak t_xM$ in the Schwartz class, $k\in\mathcal S(|\Lambda_{\mathfrak t_xM}|)$, and its Fourier transform, $\hat k\in\mathcal S(\mathfrak t_x^*M)$, a Schwartz class function on $\mathfrak t_x^*M$, are related by
\begin{equation}\label{E:fourier.inv}
\hat k(\xi)=\int_{X\in\mathfrak t_xM}e^{-2\pi\mathbf i\langle\xi,X\rangle}k
\qquad\text{and}\qquad
k(X)=\int_{\mathfrak t_x^*M}e^{2\pi\mathbf i\langle\xi,X\rangle}\hat k(\xi)d\xi,
\end{equation}
where $X\in\mathfrak t_xM$ and $\xi\in\mathfrak t_x^*M$.
Hence, $d\xi$ is the translation invariant $1$-density on $\mathfrak t_x^*M$ with values in $|\Lambda_{\mathfrak t_xM}|$ corresponding to the canonical pairing between $|\Lambda_{\mathfrak t_x^*M}|$ and $|\Lambda_{\mathfrak t_xM}|$.
Given a homogeneous norm $|-|$ on $\mathfrak t_xM$, we may contract $d\xi$ with the fundamental vector field of the dilation action, $\frac\partial{\partial\lambda}|_{\lambda=1}\dot\delta'_\lambda$, and restrict to obtain a $1$-density on the unit sphere $\{\xi\in\mathfrak t_x^*M:|\xi|=1\}$ with values in $|\Lambda_{\mathfrak t_xM}|$ which will also be denoted by $d\xi$ and can be characterized by 
\begin{equation}\label{E:fubini}
\int_{\mathfrak t_x^*M}f(\xi)d\xi=\int_0^\infty\lambda^{n-1}d\lambda\int_{\xi\in\mathfrak t_x^*M,|\xi|=1}f(\dot\delta_\lambda'(\xi))d\xi,
\end{equation}
for all compactly supported smooth functions $f$ on $\mathfrak t_x^*M\setminus\{0\}$.
\end{remark}

\begin{proof}[Proof of Proposition~\ref{prop:meromorphic}]
To begin with, let $\Re(z)>(n+d)/r$.
We fix adapted exponential coordinates and consider the corresponding asymptotic expansion of the kernel $k_z$ along the diagonal as in Lemma~\ref{L:hf}\itemref{L:hf:c}, see \eqref{E:hf:asymhat}.
Using Fourier inversion formula, see \eqref{E:fourier.inv}, we may write
$$
k_z(x,x)
=\bigl(\phi^*(k_z|_V)\bigr)(o_x)
=\int_{\mathfrak t^*_xM}\mathcal F\bigl(\chi\cdot\phi^*(k_z|_V)\bigr)(\xi)d\xi.
$$
Splitting the integral at the unit sphere in $\mathfrak t_x^*M$ and using the asymptotic expansion in \eqref{E:hf:asymhat} this may be rewritten in the form
\begin{align}\label{E:continuation}
k_z(x,x)&=\int_{\xi\in\mathfrak t_x^*M,|\xi|\leq1}\mathcal F\bigl(\chi\cdot\phi^*(k_z|_V)\bigr)(\xi)d\xi
\notag
\\&\qquad+\int_{\xi\in\mathfrak t_x^*M, |\xi|\geq1}\left(\mathcal F\bigl(\chi\cdot\phi^*(k_z|_V)\bigr)-\sum_{j=0}^N\hat k_{j,z}\right)(\xi)d\xi
\\\notag
&\qquad+\sum_{j=0}^N\int_{\xi\in\mathfrak t_x^*M,|\xi|\geq1}\hat k_{j,z}(\xi)d\xi.
\end{align}
The continuation of the left hand side in \eqref{E:continuation} can be achieved by studying the right hand side.
Clearly, the first integral on the right hand side of \eqref{E:continuation} converges and defines a smooth section of $\eend(E)\otimes|\Lambda_M|$ which depends holomorphically on $z\in\Omega$ for the integrand is smooth and analytic in $z$.
The second integral converges for $\Re(z)>(d+n-N)/r$ and defines a smooth section of $\eend(E)\otimes|\Lambda_M|$ which depends holomorphically on these $z$ in view of the estimates \eqref{E:esti}.
Using \eqref{E:fubini} and the homogeneity of $\hat k_{j,z}$, we rewrite the remaining terms in the form
\begin{multline*}
\int_{\xi\in\mathfrak t_x^*M,|\xi|\geq1}\hat k_{j,z}(\xi)d\xi=\int_1^\infty\lambda^{d-rz-j+n-1}d\lambda\int_{\xi\in\mathfrak t_x^*M,|\xi|=1}\hat k_{j,z}(\xi)d\xi
\\=\frac1{rz-d-n+j}\int_{\xi\in\mathfrak t_x^*M,|\xi|=1}\hat k_{j,z}(\xi)d\xi.
\end{multline*}
Hence, this term admits a meromorphic continuation to the entire complex plane with a single pole located at $z_j=(d+n-j)/r$.
This pole is simple with residue 
$$
\res_{z=z_j}\int_{\xi\in\mathfrak t_x^*M,|\xi|\geq1}\hat k_{j,z}(\xi)d\xi
=\frac1r\int_{\xi\in\mathfrak t_x^*M,|\xi|=1}\hat k_{j,z_j}(\xi)d\xi.
$$
Combining these observation, we obtain the meromorphic extension of $k_z(x,x)$ and the first equality in \eqref{E:extres}.
The second equality in \eqref{E:extres} is well known.

For $\mathfrak R(z)>(n+d)/r$ the Schwartz kernel $k_z$ of the operator $R(z)$ is continuous, and we can compute the trace by
$$
\tr(R(z))=\int_{x\in M}\tr_E(k_z(x,x)).
$$
The meromorphic extension of $\tr(R(z))$ thus follows from the meromorphic extension of $k_z(x,x)$, and so does the formula for the residues.
\end{proof}

We can now define a functional on integer order pseudodifferential operators $A\in\Psi^\infty(E):=\bigcup_{k\in\mathbb Z}\Psi^k(E)$ by
\[
\tau(A):=r\cdot\res_{z=0}(\zeta_A(z)),
\]
where $\zeta_A(z)=\tr(AD^{-z})$, cf.\ Proposition~\ref{prop:meromorphic} and Lemma~\ref{L:hf:comp}.

The following corollary generalizes a result for CR manifolds obtained by R.~Ponge in his thesis, see \cite{P01}.

\begin{corollary}\label{cor:ncr}
The functional $\tau\colon\Psi^{\infty}(E)\rightarrow\mathbb C$ is a non-trivial trace on the algebra of integer order pseudodifferential operators.
More precisely, we have:
\begin{enumerate}[(a)]
\item\label{C:ncr:tr} 
$\tau([A,B])=0$ for all $A,B\in\Psi^\infty(E)$.
\item\label{C:ncr:indep}
$\tau(A)$ does not depend on the positive Rockland differential operator $D$.
\item\label{C:ncr:A}
If the order of $A\in\Psi^\infty(E)$ is less than $-n$ then $\tau(A)=0$.
\item\label{C:ncr:Apj}
If $A$ has order $d\geq-n$, and $\phi^*(k_A|_V)\sim\sum_{j=0}^\infty k_j$ is an asymptotic expansion of its kernel with $k_j$ essentially homogeneous of order $d-j$, chosen such that $(\delta_\lambda)_*k_{d+n}=\lambda^{-n}k_{d+n}+\lambda^{-n}\log|\lambda|p_A$ for $\lambda>0$ with $p_A\in\mathcal P^{-n}(\mathcal TM;E,E)$, then 
$$
\tau(A)=\int_M\tr_E(p_A|_M).
$$
\item\label{C:ncr:DOA}
If $B\in\Gamma^\infty(\eend(E))$ and $j\in\mathbb N_0$, then
\begin{equation}\label{E:tauAD}
\tau\left(BD^{(j-n)/r}\right)=\frac{r}{\Gamma((n-j)/r)}\int_M\tr_E(Bq_j),
\end{equation}
where $q_j$ denotes the term in the heat kernel asymptotics of $D$, that is $k_{e^{-tD}}(x,x)\sim\sum_{j=0}^\infty t^{(j-n)/r}q_j(x)$, see Theorem~\ref{T:heat}.
\item\label{C:ncr:measure}
The map $\mu\colon C^\infty(M)\to\mathbb C$, $\mu(f):=\tau(fD^{-\frac{n}{r}})=\frac{r}{\Gamma(n/r)}\int_Mf\tr_E(q_0)$ is a positive smooth measure on $M$.
\item\label{C:ncr:nt}
%$\tau(D^{-\frac{n}{r}})=\frac{1}{\Gamma(n)}\int_M tr(q_0(x))$, where $q_0$ is the first term in the heat expansion and is positive.
In particular, $\mu(1)=\tau(D^{-\frac{n}{r}})=r\cdot a_0/\Gamma(\frac nr)>0$, where $a_0$ is the leading term in the heat trace expansion of $D$, that is, $\tr(e^{-tD})\sim\sum_{j=0}^\infty t^{(j-n)/r}a_j$, see Corollary~\ref{C:trace}.
\end{enumerate}
\end{corollary}

\begin{proof}
To see \itemref{C:ncr:tr}, we observe that, for $\Re(z)$ very large,
\[
\tr([A,B]D^{-z})=\tr(A[B,D^{-z}])+\tr([AD^{-z},B])=\tr(A[B,D^{-z}]).
\]
Note that $[B,D^{-z}]$ is a holomorphic family which vanishes at $z=0$, for $D$ is assumed to be strictly positive.
Hence, the holomorphic family $A[B,D^{-z}]$ is of the form $zR(z)$ for some holomorphic family $R(z)$. 
As $\tr(R(z))$ can at most have a simple pole at $z=0$, the function $\tr(A[B,D^{-z}])$ has no pole there and hence $\tau([A,B])=0$.

To see \itemref{C:ncr:indep}, let $\tilde D\in\Psi^{\tilde r}(E)$ be another strictly positive Rockland differential operator of order $\tilde r>0$, and suppose $A$ is of order $k$.
Hence, $AD^{-z/r}-A\tilde D^{-z/\tilde r}$ is a holomorphic family of order $k-z$ which vanishes at $z=0$.
As before, we conclude that $\tr\bigl(AD^{-z/r}-A\tilde D^{-z/\tilde r}\bigr)$ is holomorphic at $z=0$.
Consequently,
$$
\res_{z=0}\tr(AD^{-z/r})=\res_{z=0}\tr(A\tilde D^{-z/\tilde r}).
$$
Equivalently, $r\cdot\res_{z=0}\tr(AD^{-z})=\tilde r\cdot\res_{z=0}\tr(A\tilde D^{-z})$.

To see \itemref{C:ncr:A}, observe that for operators $A$ of order less than $-n$ the zeta function $\zeta_A(z)$ is clearly analytic at $z=0$ and hence has no residue there.

Part \itemref{C:ncr:Apj} follows immediately from the formula for the residue in \eqref{E:restr}.

To see \itemref{C:ncr:DOA}, suppose $B\in\Gamma^\infty(\eend(E))$ and observe that we clearly have
$$
k_{BD^{(j-n)/r}D^{-z}}(x,x)=B(x)k_{D^{(j-n)/r-z}}(x,x),
$$
for all $x\in M$.
Using Theorem~\ref{T:powers}, we obtain
$$
\res_{z=0}\left(k_{BD^{(j-n)/r}D^{-z}}(x,x)\right)
=B(x)\res_{z=(n-j)/r}\left(k_{D^{-z}}(x,x)\right)
=\frac{B(x)q_j(x)}{\Gamma((n-j)/r)},
$$
see Equation~\ref{E:powers.res}.
Applying $\tr_E$ and integrating over $x\in M$, we obtain \eqref{E:tauAD}.

The assertion about $\mu(f)$ in \itemref{C:ncr:measure} follows from \eqref{E:tauAD} by specializing to the multiplication operator $B=f$, where $f\in C^\infty(M)$, and observing that $\tr_E(q_0(x))>0$ at each $x\in M$, for we have $q_0(x)>0$ according to Theorem~\ref{T:heat}.
Hence, $\mu$ is a positive functional.

The last statement about $\tau(D^{-\frac{n}{r}})$ in \itemref{C:ncr:nt} follows from \eqref{E:tauAD} by specializing to $A=\id_E$ and using $\int_M\tr_E(q_0)=a_0>0$, see Corollary~\ref{C:trace}.
In particular, $\tau$ is non-trivial.
\end{proof}

\begin{remark}
Using \eqref{E:PnLM}, we may identify the logarithmic term $p_A=p_A|_M$ associated with a pseudodifferential operator $A\in\Psi^\infty(E)$, see Corollary~\ref{cor:ncr}\itemref{C:ncr:Apj}, as an $\eend(E)$ valued density on $M$.
This Wodzicki density is intrinsic to $A$, i.e., it does not depend on the exponential coordinates used for the asymptotic expansion.
The independence follows from the evident formula $p_{BA}=Bp_A$ for all $B\in\Gamma^\infty(\eend(E))$, the expression for the residue in Corollary~\ref{cor:ncr}\itemref{C:ncr:Apj}, and the intrinsic definition of the residue using $D$.
\end{remark}

In case of a trivially filtered manifold $M$ the trace $\tau$ is the non-commutative residue which was introduced in Wodzicki and Guillemin \cite{Wo87,Gu85}, with many applications including to index theory \cite{CM95}.
In Connes \cite{Co88} the non-commutative residue is shown to coincide with Dixmier trace for classical pseudodifferential operators of order $-n$.
We expect such a relationship to hold also in Heisenberg calculus extending part \itemref{C:ncr:measure} of Corollary~\ref{cor:ncr}.

\section{Weyl's law for Rumin--Seshadri operators}\label{S:Weyl}

In this section we will work out the eigenvalue asymptotics for Rumin--Seshadri operators associated with Rockland sequences of differential operators whose Heisenberg principle symbol sequence is the same at each point.
More precisely, we will assume that the principle symbol sequence, the $L^2$ inner product, and the volume density on any two osculating groups are simultaneously isometric. 
In this situation the constant appearing in Weyl's law is the product of the volume of the underlying manifold with a constant depending on the isometry class of the principal symbol sequence, see Corollary~\ref{C:WeylRS} below.
Similar formulas for CR and contact geometry can be found in \cite[Section~6]{P08}.

In the second part of this section we will specialize to a particular parabolic geometry in five dimensions associated with the exceptional Lie group $G_2$.
Every irreducible representation of $G_2$ gives rise to a Rockland sequence of differential operators known as a (curved) BGG sequence \cite{CSS01}.
Moreover, there is a distinguished class of fiberwise Hermitian inner products such that the Heisenberg principal symbol sequences, at any two points, are indeed isometric. 
In this situation, the constant in Weyl's law is universal, depending only on the representation of $G_2$, see Corollary~\ref{C:235} below.

\subsection{Rumin--Seshadri operators}

Let $E_i$ be vector bundles over a closed filtered manifold $M$.
Consider a sequence of differential operators 
\begin{equation}\label{E:rockseq}
\cdots\to\Gamma^\infty(E_{i-1})\xrightarrow{A_{i-1}}\Gamma^\infty(E_i)\xrightarrow{A_i}\Gamma^\infty(E_{i+1})\to\cdots
\end{equation}
where $A_i\in\DO^{r_i}(E_i,E_{i+1})$ is of Heisenberg order at most $r_i\geq1$.
Recall that such a sequence is called Rockland if the Heisenberg principal symbol sequence
$$
\cdots\to\mathcal H_\infty\otimes E_{i-1,x}\xrightarrow{\pi(\sigma_x^{r_{i-1}}(A_{i-1}))}\mathcal H_\infty\otimes E_{i,x}\xrightarrow{\pi(\sigma^{r_i}_x(A_i))}\mathcal H_\infty\otimes E_{i+1,x}\to\cdots
$$
is (weakly) exact, for each $x\in M$ and every non-trivial irreducible unitary representation $\pi\colon\mathcal T_xM\to U(\mathcal H)$ on a Hilbert space $\mathcal H$, see \cite[Definition~2.14]{DH17}.
Here $\mathcal H_\infty$ denotes the subspace of smooth vectors.

Fix a volume density $dx$ on $M$.
Moreover, let $h_i$ be a fiberwise Hermitian metric on $E_i$, and consider the associated standard $L^2$-inner product
\begin{equation}\label{E:hidx}
\llangle\psi_1,\psi_2\rrangle_{L^2(E_i)}:=\int_Mh_i(\psi_1,\psi_2)dx,
\end{equation}
where $\psi_1,\psi_2\in\Gamma^\infty(E_i)$.
Let $A_i^*\in\DO^{r_i}(E_{i+1},E_i)$ denote the formal adjoint characterized by $\llangle A_i^*\phi,\psi\rrangle_{L^2(E_i)}=\llangle\phi,A_i\psi\rrangle_{L^2(E_{i+1})}$ for all $\psi\in\Gamma^\infty(E_i)$ and $\phi\in\Gamma^\infty(E_{i+1})$, cf.\ \cite[Remark~2.5]{DH17}.

Assume $r_i\geq1$ and choose positive integers $s_i$ such that 
\begin{equation}\label{E:risi}
r_{i-1}s_{i-1}=r_is_i=:\kappa.
\end{equation}
It is easy to see that the Rumin--Seshadri operator $\Delta_i\in\DO^{2\kappa}(E_i)$ defined by 
\begin{equation}\label{E:RS}
\Delta_i:=\bigl(A_{i-1}A_{i-1}^*\bigr)^{s_{i-1}}+\bigl(A_i^*A_i\bigr)^{s_i}
\end{equation}
is Rockland, see \cite{RS12} and \cite[Lemma~2.18]{DH17}.
Clearly, $\Delta_i$ is formally selfadjoint and non-negative.

\begin{definition}[Model sequence]\label{D:model}
(a) A \emph{model sequence} consists of a finite dimensional graded nilpotent Lie algebra $\noe$, a volume density $\mu\in|\Lambda_\noe|$, finite dimensional Hermitian vector spaces $F_i$, integers $r_i\geq1$, and $B_i\in\mathcal U_{-r_i}(\mathfrak\noe)\otimes\hom(F_i,F_{i+1})$.
A model sequence gives rise to a sequence of left invariant homogeneous differential operators on the corresponding simply connected Lie group $N$,
$$
\cdots\to C^\infty(N,F_{i-1})\xrightarrow{B_{i-1}}C^\infty(N,F_i)\xrightarrow{B_i}C^\infty(N,F_{i+1})\to\cdots,
$$
and provides left invariant $L^2$ inner products on the spaces $C^\infty(N,F_i)$ which are induced by the invariant volume density on $N$ corresponding to $\mu$ and the Hermitian inner products on $F_i$, cf.\ \eqref{E:L2}.

(b) A model sequence is called \emph{Rockland} if 
$$
\cdots\to\mathcal H_\infty\otimes F_{i-1}\xrightarrow{\pi(B_{i-1})}\mathcal H_\infty\otimes F_i\xrightarrow{\pi(B_i)}\mathcal H_\infty\otimes F_{i+1}\to\cdots
$$
is exact for every non-trivial irreducible unitary representation $\pi\colon N\to U(\mathcal H)$ where $\mathcal H_\infty$ denotes the subspace of smooth vectors.

(c) Two model sequences, $\bigl(\noe,\mu,B_i\in\mathcal U_{-r_i}(\mathfrak\noe)\otimes\hom(F_i,F_{i+1})\bigr)$ and $\bigl(\noe',\mu',B_i'\in\mathcal U_{-r_i'}(\mathfrak\noe')\otimes\hom(F_i',F_{i+1}')\bigr)$ are said to be \emph{isometric} if $r_i=r_i'$, and there exists an isomorphism of graded Lie algebras $\dot\varphi\colon\noe\to\noe'$ mapping $\mu$ to $\mu'$, and there exist unitary isomorphisms $\phi_i\colon F_i\to F_i'$ such that the induced isomorphisms
$$
\dot\varphi\otimes\hom(\phi_i^{-1},\phi_{i+1})\colon\mathcal U_{-r_i}(\mathfrak\noe)\otimes\hom(F_i,F_{i+1})\to\mathcal U_{-r_i'}(\mathfrak\noe')\otimes\hom(F_i',F_{i+1}')
$$ 
map $B_i$ to $B_i'$, for all $i$.
\end{definition}

If $E_i$ are Hermitian vector bundles over a filtered manifold $M$ equipped with a volume density, then the Heisenberg principal symbols of a sequence of differential operators $A_i\in\DO^{r_i}(E_i,E_{i+1})$ give rise to a model sequence 
\begin{equation}\label{E:hps}
\bigl(\mathfrak t_xM,\mu_x,\sigma^{r_i}_x(A_i)\in\mathcal U_{-r_i}(\mathfrak t_xM)\otimes\hom(E_{x,i},E_{x,i+1})\bigr)
\end{equation}
at every point $x\in M$.
Here $\mu_x\in|\Lambda_{\mathfrak t_xM}|$ denotes the volume density induced by $dx$ via the canonical isomorphism $|\Lambda_{\mathfrak t_xM}|=|\Lambda_{T_xM}|$.

The following generalizes \cite[Proposition~6.1.5]{P08}.

\begin{corollary}\label{C:WeylRS}
If $\bigl(\noe,\mu,B_i\in\mathcal U_{-r_i}(\mathfrak\noe)\otimes\hom(F_i,F_{i+1})\bigr)$ is a Rockland model sequence, see Definition~\ref{D:model}, then there exist constants $\alpha_i>0$ with the following property.
Suppose $E_i$ are Hermitian vector bundles over a closed filtered manifold $M$, let $dx$ be a volume density on $M$, and suppose $A_i\in\DO^{r_i}(E_i,E_{i+1})$ is a sequence of differential operators such that the Heisenberg principal symbol sequence, see \eqref{E:hps}, is isometric to the model sequence at each point $x\in M$.
Moreover, fix positive integers $s_i$ as in \eqref{E:risi}, and let $\Delta_i\in\DO^{2\kappa}(E_i)$ denote the associated Rumin--Seshadri operators, see \eqref{E:hidx} and \eqref{E:RS}.
Then, as $\lambda\to\infty$,
\begin{equation}\label{E:WeylRS}
\sharp\left\{\begin{array}{c}\textrm{eigenvalues of $\Delta_i$ less than $\lambda$}\\\textrm{counted with multiplicities}\end{array}\right\}
\sim\alpha_i\cdot\vvol(M)\lambda^{n/2\kappa}.
\end{equation}
Here $\vvol(M):=\int_Mdx$, and $n$ denotes the homogeneous dimension of $\noe$.
\end{corollary}

\begin{proof}
Let $x\in M$.
By assumption, there exists an isomorphism of graded nilpotent Lie algebras $\dot\varphi\colon\mathfrak t_xM\to\noe$ mapping the volume density induced by $dx$ on $\mathfrak t_xM$ to the volume density $\mu$ on $\noe$, 
and there exist unitary isomorphisms of vector spaces $\phi_i\colon E_{x,i}\to F_i$ such that the induced isomorphisms 
$$
\dot\varphi\otimes\hom(\phi_i^{-1},\phi_{i+1})\colon\mathcal U_{-r_i}(\mathfrak t_xM)\otimes\hom(E_{x,i},E_{x,i+1})\to\mathcal U_{-r_i}(\mathfrak\noe)\otimes\hom(F_i,F_{i+1})
$$
map $\sigma^{r_i}_x(A_i)$ to $B_i$.
We conclude that the induced isomorphism
$$
\dot\varphi\otimes\hom(\phi_i^{-1},\phi_i)\colon\mathcal U_{-2\kappa}(\mathfrak t_xM)\otimes\eend(E_{x,i})\to\mathcal U_{-2\kappa}(\mathfrak\noe)\otimes\eend(F_i)
$$
maps $\sigma^{2\kappa}_x(\Delta_i)$ to $C_i:=(B_{i-1}B_{i-1}^*)^{s_{i-1}}+(B_i^*B_i)^{s_i}$.
Hence, the fundamental solutions of the heat equation are related by
\begin{equation}\label{E:qwerty}
k_t^{\sigma^{2\kappa}_x(\Delta_i)}=\varphi^*\bigl(\phi_i^{-1}k_t^{C_i}\phi_i\bigr)
\end{equation}
where $\varphi\colon\mathcal T_xM\to N$ denotes the isomorphism of Lie groups integrating $\dot\varphi$.
Let $\alpha_i>0$ the unique number such that 
\begin{equation}\label{E:alphai}
\tr_{F_i}\bigl(k^{C_i}_1(o_N)\bigr)=\alpha_i\Gamma(1+n/2\kappa)\mu.
\end{equation}
Note that $\alpha_i$ only depends on the model sequence and (potentially) on the integers $s_i$, but not on the actual geometry.
From \eqref{E:qwerty} and \eqref{E:q0x}, we obtain
$$
\tr_{E_i}\bigl(q_0^{\Delta_i}\bigr)=\alpha_i\Gamma(1+n/2\kappa)dx.
$$
where $q_0^{\Delta_i}\in\Gamma^\infty\bigl(\eend(E_i)\otimes|\Lambda_M|\bigr)$ is as in Theorem~\ref{T:heat}.
Hence, for the corresponding $a_0^{\Delta_i}=\int_M\tr_{E_i}\bigl(q_0^{\Delta_i}\bigr)$, see Corollary~\ref{C:trace}, we obtain
$$
a_0^{\Delta_i}=\alpha_i\Gamma(1+n/2\kappa)\vvol(M).
$$
The statement in \eqref{E:WeylRS} thus follows from Corollary~\ref{C:weyl}.

It remains to show that the constants $\alpha_i$ do not depend on the choice of positive integers $s_i$ as in \eqref{E:risi}.
To this end, suppose $u\in\N$ and consider $s_i':=us_i$.
Then $r_{i-1}s_{i-1}'=r_is_i'=\kappa'$ with $\kappa'=u\kappa$, cf.~\eqref{E:risi}.
Since $A_i$ is a Rockland sequence, we have $\sigma^{k_i}_x(A_i)\sigma^{k_{i-1}}_x(A_{i-1})=0$.
Hence, $\Delta'_i:=(A_{i-1}A_{i-1}^*)^{s_{i-1}'}+(A_i^*A_i)^{s_i'}$ and $(\Delta_i)^u$ have the same Heisenberg principal symbol, i.e.\ $\sigma^{2\kappa'}_x(\Delta_i')=\sigma_x^{2u\kappa}((\Delta_i)^u)$, and thus the constants in their Weyl asymptotics coincide, see Corollary~\ref{C:weyl}, Remark~\ref{R:loccomp}, and \eqref{E:q0x}.
Clearly, the number of eigenvalues less than $\lambda$ of $(\Delta_i)^u$ coincides with the number of eigenvalues of $\Delta_i$ which are smaller than $\lambda^{1/u}$.
From \eqref{E:WeylRS} we thus obtain
$$
\sharp\left\{\begin{array}{c}\textrm{eigenvalues of $\Delta_i'$ less than $\lambda$}\\\textrm{counted with multiplicities}\end{array}\right\}
\sim\alpha_i\cdot\vvol(M)\lambda^{n/2\kappa'}.
$$
We conclude that the constant $\alpha_i$ does not depend on the choice of $s_i$.
\end{proof}

\subsection{Generic rank two distributions in dimension five}\label{SS:235}

In the remaining part of this section we will specialize to a particular geometry in five dimensions and discuss certain natural Rockland sequences over it, known as curved BGG sequences, which have been constructed by \v Cap, Slov\'ak and Sou\v cek, see \cite{CSS01}.

Let $M$ be $5$-manifold, and suppose $H\subseteq TM$ is a smooth subbundle of rank two.
Recall that $H$ is said to be of Cartan type \cite{BH93} if it is bracket generating with growth vector $(2,3,5)$.
More precisely, $H$ is of Cartan type if every point in $M$ admits an open neighborhood $U$ and there exist sections $X,Y\in\Gamma^\infty(H|_U)$ such that $X,Y,[X,Y],[X,[X,Y]],[Y,[X,Y]]$ is a frame of $TM|_U$.
Putting $T^{-1}M:=H$ and denoting the rank three bundle spanned by Lie brackets of sections of $H$ by $T^{-2}M:=[H,H]$, the $5$-manifold $M$ thus becomes a filtered manifold,
\begin{equation}\label{E:TMG2}
TM=T^{-3}M\supseteq T^{-2}M\supseteq T^{-1}M\supseteq T^0M=0.
\end{equation}
The osculating algebras $\mathfrak t_xM$ are all isomorphic to the graded nilpotent Lie algebra $\noe=\noe_{-3}\oplus\noe_{-2}\oplus\noe_{-1}$, uniquely characterized (up to isomorphism) by the fact that $\noe_{-1}$ admits a basis $\xi,\eta$ such that $[\xi,\eta]$ is a basis of $\noe_{-2}$ and $[\xi,[\xi,\eta]],[\eta,[\xi,\eta]]$ forms a basis of $\noe_{-3}$.

Rank two distributions of Cartan type are also known as generic rank two distributions in dimension five, see \cite{S08,CS09a,HS09,HS11,SW17}, the condition on $H$ being open with respect to the $C^2$-topology.
Their history can be traced back to Cartan's celebrated ``five variables paper'' \cite{C10}.
Whether a 5-manifold admits a Cartan distribution is well understood in the open case, see \cite[Theorem~2]{DH16}.
For closed 5-manifolds, however, this problem remains open and has served as a major motivation to develop the analysis in this paper.

A rank two distributions of Cartan type can equivalently be described as regular normal Cartan geometry of type $(G,P)$ where $G$ denotes the split real form of the exceptional Lie group $G_2$ and $P$ denotes a maximal parabolic subgroup corresponding to the shorter root, see \cite{C10,S08} and \cite[Theorem~3.1.14 and Section~4.3.2]{CS09}.
The Lie algebra of $G$ admits a grading,
\begin{equation}\label{E:g2}
\goe=\goe_{-3}\oplus\goe_{-2}\oplus\goe_{-1}\oplus\goe_0\oplus\goe_1\oplus\goe_2\oplus\goe_3,
\end{equation}
i.e.\ $[\goe_i,\goe_j]\subseteq\goe_{i+j}$ for all $i$ and $j$, such that
$$
P=\{g\in G\mid\forall i:\Ad_g(\goe^i)\subseteq\goe^i\},
$$
where $\goe^i:=\bigoplus_{i\leq j}\goe_j$ denotes the associated filtration.
The Lie algebra of $P$ is $\poe=\goe_0\oplus\goe_1\oplus\goe_2\oplus\goe_3$.
The corresponding Levi group, 
\begin{equation}\label{E:G0}
G_0=\{g\in G\mid\forall i:\Ad_g(\goe_i)=\goe_i\}\cong\GL(2,\R),
\end{equation}
has Lie algebra $\goe_0\cong\gl(2,\R)$.
Furthermore, $\poe_+=\goe_1\oplus\goe_2\oplus\goe_3$ is a nilpotent ideal in $\poe$ and $P_+:=\exp(\poe_+)$ is a closed normal subgroup in $P$ such that the inclusion $G_0\subseteq P$ induces a natural isomorphism of groups, $G_0=P/P_+$. 
The graded nilpotent Lie algebra $\goe_-:=\goe_{-3}\oplus\goe_{-2}\oplus\goe_{-1}$ is isomorphic to $\noe$.

The Cartan geometry consists of a principal $P$-bundle $\mathcal G\to M$ and a regular Cartan connection $\omega\in\Omega^1(\mathcal G;\goe)$ satisfying a normalization condition, see \cite{CS09}.
The Cartan connection induces an isomorphism $TM=\mathcal G\times_P\goe/\poe$.
By regularity, the filtration on $TM$, see \eqref{E:TMG2}, coincides with the filtration induced by the filtration $\goe/\poe=\goe^{-3}/\poe\supseteq\goe^{-2}/\poe\supseteq\goe^{-1}/\poe\supseteq\goe^0/\poe=0$.
Note that $P_+$ acts trivially on the associated graded $\gr(\goe/\poe)$, and the inclusion $\goe_-\subseteq\goe$ induces an isomorphism of $G_0=P/P_+$ modules, $\goe_-=\gr(\goe/\poe)$.
Regularity also implies that the induced isomorphism $\mathfrak tM=\gr(TM)=\mathcal G\times_P\gr(\goe/\poe)=\mathcal G_0\times_{G_0}\goe_-$ is an isomorphism of bundles of graded nilpotent Lie algebras.
Here $\mathcal G_0:=\mathcal G/P_+$ is considered as a principal $G_0$-bundle over $M$.

The canonically associated Cartan geometry permits to construct natural sequences of differential operators over $M$ known as curved BGG sequences, see \cite{CSS01} and \cite{CS12,CD01}.
Every finite dimensional complex representation $\rho\colon G\to\GL(\mathbb E)$ gives rise to a sequence of natural differential operators,
\begin{equation}\label{E:BGG}
0\to\Gamma(E_0)\xrightarrow{D_0}\Gamma(E_1)\xrightarrow{D_1}\Gamma(E_2)\xrightarrow{D_2}\Gamma(E_3)\xrightarrow{D_3}\Gamma(E_4)\xrightarrow{D_5}\Gamma(E_5)\to0
\end{equation}
where $E_i:=\mathcal G_0\times_{G_0}H^i(\goe_-;\mathbb E)$ and $H^i(\goe_-;\mathbb E)$ denotes the $i$-th Lie algebra cohomology with coefficients in the representation of $\goe_-$ obtained by restricting $\rho$.
In \cite[Corollary~4.23 and Example~4.24]{DH17} it has been shown that this is a Rockland sequence for every irreducible representation $\rho$.
We will denote the Heisenberg order of $D_i$ by $k_i$.

Let $\theta$ be a Cartan involution on $\goe$ such that 
\begin{equation}\label{E:theta}
\theta(\goe_i)=\goe_{-i}
\end{equation} 
for all $i$, cf.\ \cite[Proof of Proposition~3.3.1]{CS09}.
Denoting the Killing form of $\goe$ by $B$, we obtain a Euclidean inner product on $\goe$, given by $B_\theta(X,Y):=-B(X,\theta(Y))$ where $X,Y\in\goe$.
The summands in the decomposition \eqref{E:g2} are orthogonal with respect to $B_\theta$.
Let $\Theta$ denote the global Cartan involution on $G$ corresponding to $\theta$.
Then $K:=\{g\in G:\Theta(g)=g\}\cong\bigl(\SU(2)\times\SU(2)\bigr)/\Z_2$ is a maximal compact subgroup of $G$, and $B_\theta$ is invariant under $K$.
Moreover, $\Theta$ restricts to a Cartan involution on $G_0$, and $K_0:=\{g\in G_0:\Theta(g)=g\}\cong O(2,\R)$ is a maximal compact subgroup of $G_0$.
Let $h$ be a Hermitian inner product on $\mathbb E$ such that 
\begin{equation}\label{E:h}
h(\rho'(X)v,w)=-h(v,\rho'(\theta(X))w)
\end{equation} 
for all $X\in\goe$ and $v,w\in\mathbb E$ where $\rho'\colon\goe\to\gl(\mathbb E)$ denotes the Lie algebra representation corresponding to $\rho$, cf.\ \cite[Proof of Proposition~3.3.1]{CS09}.
In particular, $h$ is invariant under the action of $K$.
We equip the $G_0$-module $C^i(\goe_-;\mathbb E)=\Lambda^i\goe_-^*\otimes\mathbb E$ with the Hermitian inner product $h_i$ induced by $B_\theta$ and $h$.
Note that $h_i$ is invariant under $K_0$.
These inner products play an important role in the construction of the curved BGG sequences since Kostant's codifferential, $\partial^*\colon C^{i+1}(\goe_-;\mathbb E)\to C^i(\goe_-;\mathbb E)$, is adjoint to the Chevalley--Eilenberg codifferential $\partial\colon C^i(\goe_-;\mathbb E)\to C^{i+1}(\goe_-;\mathbb E)$ with respect to $h_i$, see \cite[Proposition~3.3.1]{CS09}.
The induced $K_0$-invariant Hermitian inner product on $H^i(\goe_-;\mathbb E)$ will also be denoted by $h_i$.
We let $\mu\in|\Lambda_{\goe_-}|$ denote the volume density associated with the Euclidean inner product on $\goe_-$ induced by $B_\theta$.
Clearly, $\mu$ is $K_0$-invariant.

To obtain Hermitian inner products on the bundles $E_i$, we choose a reduction of the structure group of $\mathcal G_0\to M$ to $K_0\subseteq G_0$.
This is a principal $K_0$-bundle $\mathcal K_0\to M$ together with a $K_0$-equivariant bundle map $\mathcal K_0\subseteq\mathcal G_0$.
This choice amounts to fixing a fiberwise Euclidean metric on the rank two bundle $H$.
It provides an isomorphism of principal $G_0$-bundles, $\mathcal G_0=\mathcal K_0\times_{K_0}G_0$, and thus
\begin{equation}\label{E:Eiassoc}
E_i=\mathcal G_0\times_{G_0}H^i(\goe_-;\mathbb E)=\mathcal K_0\times_{K_0}H^i(\goe_-;\mathbb E).
\end{equation}
Moreover,
$$
\gr(TM)=\mathcal G_0\times_{G_0}\goe_-=\mathcal K_0\times_{K_0}\goe_-.
$$
We equip $E_i$ with the fiberwise Hermitian metric induced from the $K_0$-invariant Hermitian metric on $H^i(\goe_-;\mathbb E)$, and we equip $M$ with the volume density $dx$ induced from the $K_0$-invariant volume density $\mu$ on $\goe_-$.

\begin{corollary}\label{C:235}
Let $M$ be a closed 5-manifold equipped with a rank two distribution of Cartan type.
Moreover, let $\rho$ be an irreducible complex representation of the exceptional Lie group $G_2$, consider the associated curved BGG sequence \eqref{E:BGG}, and let $\Delta_i$ denote the corresponding Rumin--Seshadri operator, see \eqref{E:RS}, constructed using the volume density $dx$ on $M$ and the fiberwise Hermitian metrics on $E_i$ described above. 
Then, as $\lambda\to\infty$,
\begin{equation}\label{E:Weyl235}
\sharp\left\{\begin{array}{c}\textrm{eigenvalues of $\Delta_i$ less than $\lambda$}\\\textrm{counted with multiplicities}\end{array}\right\}
\sim\alpha_i(\rho)\cdot\vvol(M)\lambda^{5/\kappa}
\end{equation}
where $\alpha_i(\rho)>0$ and $\vvol(M)=\int_Mdx$.
The constant $\alpha_i(\rho)$ does not depend on $M$ or the distribution $H$, nor does it depend on the Cartan involution $\theta$ satisfying \eqref{E:theta} or the Hermitian inner product $h$ satisfying \eqref{E:h} or the reduction of structure group of $\mathcal G_0$ along $K_0\subseteq G_0$, and it is also independent of the choice of positive integers $s_i$ as in \eqref{E:risi}; it only depends on $i$ and the representation $\rho$.
\end{corollary}

\begin{proof}
Fix $x\in M$.
Every $G_0$-equivariant identification $(\mathcal G_0)_x\cong G_0$ induces an isomorphism of graded nilpotent Lie algebras $\dot\varphi\colon\mathfrak t_xM\to\goe_-$ and isomorphisms of vector spaces $\phi_i\colon E_{i,x}\to H^i(\goe_-;\mathbb E)$ which intertwine the principal Heisenberg symbol $\sigma^{k_i}_x(D_i)$ with $D_i^{G_-}$, the corresponding left invariant operator on the nilpotent Lie group $G_-=\exp(\goe_-)$.
If the frame $(\mathcal G_0)_x\cong G_0$ is induced from a $K_0$-equivariant identification $(\mathcal K_0)_x\cong K_0$, then the isomorphisms $\phi_i$ are unitary and $\dot\varphi$ maps the volume density $\mu_x$ on $\mathfrak t_xM$ induced by $dx$ to the volume density $\mu$ on $\goe_-$. 
Hence, at every point $x\in M$, the principal symbol sequence
$$
\bigl(\mathfrak t_xM,\mu_x,\sigma^{k_i}_x(D_i)\in\mathcal U_{-k_i}(\mathfrak t_xM)\otimes\hom(E_{i,x},E_{i+i,x})\bigr)
$$
is isometric to the model sequence
$$
\bigl(\goe_-,\mu,D_i^{G_-}\in\mathcal U_{-k_i}(\goe_-)\otimes\hom(H^i(\goe_-;\mathbb E),H^{i+1}(\goe_-;\mathbb E))\bigr),
$$
cf.\ Definition~\ref{D:model}. 
According to Corollary~\ref{C:WeylRS} there exists $\alpha_i(\rho)>0$ such that \eqref{E:Weyl235} holds, and this constant does not depend on the choice of $s_i$ as in \eqref{E:risi}.

It remains to show that $\alpha_i(\rho)$ does not depend on $\theta$ or $h$.
To this end, suppose $\tilde\theta$ is another Cartan involution on $\goe$ such that $\tilde\theta(\goe_i)=\goe_{-i}$, for all $i$, cf.\ \eqref{E:theta}.
Then there exists $g\in G$ such that $\tilde\theta=\Ad_g^{-1}\circ\,\theta\circ\Ad_g=\Ad_{g^{-1}\Theta(g)}\circ\,\theta$.
In view of the Cartan decomposition, we may write $g=k\exp(X)$ with $k\in G$ and $X\in\goe$ such that $\Theta(k)=k$ and $\theta(X)=-X$.
Hence, $g^{-1}\Theta(g)=\exp(-2X)$, and we obtain $\tilde\theta=\Ad_{\exp(-2X)}\circ\,\theta$.
Since $(\tilde\theta\theta^{-1})(\goe_i)=\goe_i$, we have $\exp(-2X)\in G_0$, see \eqref{E:G0}.
Using $\theta(X)=-X$, we actually conclude $X\in\goe_0$, hence $g_0:=\exp(X)\in G_0$, and
\begin{equation}\label{E:ttheta}
\tilde\theta=\Ad_{g_0}^{-1}\circ\,\theta\circ\Ad_{g_0}.
\end{equation}
For the corresponding global Cartan involution $\tilde\Theta$ we get $\tilde\Theta(g)=g_0^{-1}\Theta(g_0gg_0^{-1})g_0$, for all $g\in G$, and thus the corresponding maximal compact subgroups are related by $\tilde K=\{g\in G:\tilde\Theta(g)=g\}=g_0^{-1}Kg_0$ and
\begin{equation}\label{E:tK0}
\tilde K_0=\{g\in G_0:\tilde\Theta(g)=g\}=g_0^{-1}K_0g_0.
\end{equation}

Furthermore, suppose $\tilde h$ is another Hermitian inner product on $\mathbb E$ such that $\tilde h(\rho'(X)v,w)=-\tilde h(v,\rho'(\tilde\theta(X))w)$ for all $X\in\goe$ and $v,w\in\mathbb E$, cf.~\eqref{E:h}.
Using $\rho'(\Ad_{g_0}(X))=\rho(g_0)\rho'(X)\rho(g_0)^{-1}$ and \eqref{E:ttheta} as well as \eqref{E:h}, we obtain the relation $h(\rho(g_0)\rho'(X)v,\rho(g_0)w)=-h(\rho(g_0)v,\rho(g_0)\rho'(\tilde\theta(X))w)$.
Hence, by Schur's lemma, there exists a constant $C>0$ such that, for all $v,w\in\mathbb E$,
\begin{equation}\label{E:th}
\tilde h(v,w)=C\cdot h\bigl(\rho(g_0)v,\rho(g_0)w\bigr).
\end{equation}
By invariance of the Killing form and \eqref{E:ttheta}, we have 
\begin{equation}\label{E:Bttheta}
B_{\tilde\theta}(X,Y)=B_\theta\bigl(\Ad_{g_0}(X),\Ad_{g_0}(Y)\bigr).
\end{equation}
Hence, the $\tilde K_0$-invariant Hermitian inner product $\tilde h_i$ induced by $B_{\tilde\theta}$ and $\tilde h$ on $H^i(\goe_-;\mathbb E)$ is related to the $K_0$-invariant Hermitian inner product $h_i$ via
\begin{equation}\label{E:thi}
\tilde h_i(\xi,\eta)=C\cdot h_i(g_0\cdot\xi,g_0\cdot\eta)
\end{equation}
where $\xi,\eta\in H^i(\goe_-;\mathbb E)$ and $g_0\cdot\xi$ denotes the $G_0$-action.

Suppose, moreover, $\tilde{\mathcal K}_0\to\mathcal G_0$ is a reduction of the structure group along the inclusion $\tilde K_0\subseteq G_0$.
Let $\mathcal K_0$ denote the principal $K_0$-bundle with underlying manifold $\tilde{\mathcal K}_0$ where the right action by $k\in K_0$ is given by the principal $\tilde K_0$-action of $g_0^{-1}kg_0$ on $\tilde{\mathcal K}_0$, see \eqref{E:tK0}.
We consider the reduction of structure group $\mathcal K_0\to\mathcal G_0$ along the inclusion $K_0\subseteq G_0$ obtained by composing the reduction $\tilde{\mathcal K}_0\to\mathcal G_0$ with the diffeomorphism $\mathcal G_0\to\mathcal G_0$ provided by right action with $g_0^{-1}$.
Using \eqref{E:thi}, we conclude that via the canonical identifications
$$
\tilde{\mathcal K}_0\times_{\tilde K_0}H^i(\goe_-;\mathbb E)=E_i=\mathcal K_0\times_{K_0}H^i(\goe_-;\mathbb E),
$$
cf.~\eqref{E:Eiassoc}, the Hermitian inner products on $E_i$ induced by $\tilde h_i$ and $h_i$, respectively, differ by the constant $C$.
Similarly, using \eqref{E:Bttheta}, we see that the induced volume densities on $M$ coincide.
This implies that the corresponding $L_2$-inner products on $\Gamma(E_i)$ differ by a constant which is independent of $i$, hence they give rise to the same formally adjoint operators.

Summarizing, we see that the Rumin--Seshadri operators associated with $\tilde\theta$, $\tilde h$, and a reduction $\tilde{\mathcal K_0}\to\mathcal G_0$ coincide with the Rumin--Seshadri operators associated with $\theta$, $h$, and the reduction $\mathcal K_0\to\mathcal G_0$ described in the previous paragraph.
We conclude that the constant $\alpha_i(\rho)$ does not depend on $\theta$ or $h$.
\end{proof}

Let us conclude this section with some remarks on the constant $\alpha_0(\rho_0)$ in Corollary~\ref{C:235} for the trivial representation $\rho_0$.
For the nilpotent Lie group $G_-$ all irreducible unitary representations have been calculated explicitly in \cite[Proposition~8]{D58}.
For each $(\lambda,\mu,\nu)\in\mathbb R^3$ such that $(\lambda,\mu)\neq(0,0)$ there is an irreducible unitary representation of $G_-$ on $L^2(\mathbb R)$ with the standard inner product $\langle\psi_1,\psi_2\rangle=\int\bar\psi_1(\theta)\psi_2(\theta)d\theta$ where $d\theta$ denotes the Lebesgue measure.
Moreover, these representations are mutually non-equivalent.
There are further, more singular representations, but they form a set of measure zero with respect to the Plancherel measure.
On the representations parametrized by $(\lambda,\mu,\nu)$ above, the Plancherel measure is $d\lambda\,d\mu\,d\nu$, i.e.\ the Lebesgue measure on $\mathbb R^3$, see \cite[Equation~(26)]{D58}.
In the representation corresponding to $(\lambda,\mu,\nu)$, the subLaplacian on $G_-$ takes the form of a Laplacian with quartic potential,
$$
-\hat\Delta^{\lambda,\mu,\nu}=\frac1{\lambda^2+\mu^2}\frac{d^2}{d\theta^2}-\frac{\bigl((\lambda^2+\mu^2)^2\theta^2+\nu\bigr)^2}{4(\lambda^2+\mu^2)},
$$
see \cite[Section~3.3.2]{BGR14}.
Denoting the solution of
$$
\tfrac\partial{\partial t}\Psi_t=-\hat\Delta^{\lambda,\mu,\nu}\Psi_t,\qquad\lim_{t\searrow0}\Psi_t=\delta_{\bar\theta}
$$
by $\Psi^{\lambda,\mu,\nu}_t(\theta,\bar\theta)$, the heat kernel of the subLaplacian at the origin in $G_-$ can be expressed in the form
$$
k_t^{G_-}(o)=\int_{(\lambda,\mu)\neq(0,0)}d\lambda\,d\mu\,d\nu\int_{-\infty}^\infty d\theta\,\Psi_t^{\lambda,\mu,\nu}(\theta,\theta),
$$
see \cite[Corollary~29]{ABGR09} and Theorem~\cite[Equation~(25)]{BGR14}.
Integrating out the circular symmetry and using \eqref{E:alphai}, we obtain the following expression for one of the constants in Corollary~\ref{C:235}
$$
\alpha_0(\rho_0)=\frac{2\pi}{5!}\int_0^\infty\mu d\mu\int_{-\infty}^\infty d\nu\int_{-\infty}^\infty d\theta\,\Psi_1^{0,\mu,\nu}(\theta,\theta).
$$
No explicit formulas for $\Psi_t^{\lambda,\mu,\nu}(\theta,\bar\theta)$ appear to be know \cite{BGR14}.
For some 2-step nilpotent groups explicit formulas for the heat kernel of the subLaplacian can be found in \cite{C79,G77,BGG00}.

\bibliographystyle{abbrv}
  \bibliography{hypo}

\end{document}